\documentclass[a4paper,12pt]{preprint}
\usepackage[margin=1.3in]{geometry}  
\usepackage[full]{textcomp}

\usepackage{mathalfa}
\usepackage{ulem}
\usepackage{microtype}
\usepackage{url}
\usepackage[shortlabels]{enumitem}

\usepackage{booktabs}

\usepackage{enumitem}
\usepackage{amsmath}
\usepackage{amsfonts} 
\usepackage{amssymb}             

\usepackage{comment}

\usepackage{mathrsfs}
\usepackage{booktabs}
\usepackage{mathtools}

\usepackage{tikz}
\usepackage{amsthm}                
\usepackage{mhequ}
\usepackage{hyperref}

\hypersetup{pdfstartview=}

\addtocontents{toc}{\protect\hypersetup{hidelinks}}

\DeclareSymbolFont{sfoperators}{OT1}{ptm}{m}{n}
\DeclareSymbolFontAlphabet{\mathsf}{sfoperators}

\makeatletter
\def\operator@font{\mathgroup\symsfoperators}
\makeatother

\numberwithin{equation}{section}

\newtheorem{thm}{Theorem}[section]

\newtheorem{lem}[thm]{Lemma}
\newtheorem{prop}[thm]{Proposition}

\newtheorem{assumption}[thm]{Assumption}
\theoremstyle{remark}

\newtheorem{rmk}[thm]{Remark}
\DeclareMathOperator{\id}{id}

\makeatletter
\def\th@newremark{\th@remark\thm@headfont{\bfseries}}

\def\bdiamond{\mathop{\mathpalette\bdi@mond\relax}}
\newcommand\bdi@mond[2]{%
	\vcenter{\hbox{\m@th
			\scalebox{\ifx#1\displaystyle 2.6\else1.8\fi}{$#1\diamond$}%
	}}%
}

\def\bDiamond{\mathop{\mathpalette\bDi@mond\relax}}
\newcommand\bDi@mond[2]{%
	\vcenter{\hbox{\m@th
			\scalebox{\ifx#1\displaystyle 2.6\else1.2\fi}{$#1\Diamond$}%
	}}%
}
\makeatletter

\definecolor{darkgreen}{rgb}{0.1,0.7,0.1}
\definecolor{darkred}{rgb}{0.7,0.1,0.1}
\definecolor{darkblue}{rgb}{0,0,0.7}
\newcommand{\EE}{\mathbb{E}}     

\newcommand{\NN}{\mathbb{N}}
\newcommand{\PP}{\mathbb{P}}     

\newcommand{\RR}{\mathbb{R}}      
\newcommand{\TT}{\mathbb{T}}

\newcommand{\ZZ}{\mathbb{Z}}      

\newcommand{\aA}{\mathcal{A}}
\newcommand{\bB}{\mathcal{B}}
\newcommand{\cC}{\mathcal{C}}
\newcommand{\dD}{\mathcal{D}}

\newcommand{\hH}{\mathcal{H}}
\newcommand{\iI}{\mathcal{I}}

\newcommand{\lL}{\mathcal{L}}
\newcommand{\mM}{\mathcal{M}}
\newcommand{\nN}{\mathcal{N}}
\newcommand{\oO}{\mathcal{O}}
\newcommand{\pP}{\mathcal{P}}

\newcommand{\sS}{\mathcal{S}}
\newcommand{\tT}{\mathcal{T}}
\newcommand{\uU}{\mathcal{U}}
\newcommand{\vV}{\mathcal{V}}
\newcommand{\wW}{\mathcal{W}}
\newcommand{\xX}{\mathcal{X}}
\newcommand{\yY}{\mathcal{Y}}
\newcommand{\zZ}{\mathcal{Z}}

\makeatletter 
\newcommand{\ex}{{\operator@font ex}}
\newcommand{\cov}{{\operator@font cov}}
\newcommand{\var}{{\operator@font var}}
\newcommand{\corr}{{\operator@font corr}}
\newcommand{\diam}{{\operator@font diam}}
\newcommand{\Av}{{\operator@font Av}}
\newcommand{\trig}{{\operator@font trig}}
\newcommand{\Enh}{{\operator@font Enh}}
\newcommand{\EEnh}{\overline {\operator@font Enh}}

\makeatother
\newcommand{\Laplace}{\Delta}

\newcommand{\sgn}{\text{sgn}}

\newcommand{\one}{\mathbf{1}}

\renewcommand{\r}{\mathbf{r}}

\newcommand{\emS}{\ensuremath{\mathscr{S}}}
\newcommand{\emR}{\ensuremath{\mathscr{R}}}

\newcommand{\emL}{\ensuremath{\mathscr{L}}}
\newcommand{\emP}{\ensuremath{\mathscr{P}}}

\newcommand{\eps}{\varepsilon}
\newcommand{\md}{\mathrm{d}}
\renewcommand{\d}{\partial}

\newcommand{\dist}{\operatorname{dist}}
\newcommand{\spann}{\operatorname{span}}

\newcommand{\bracket}[1]{\langle #1 \rangle}

\usepackage{amsmath}
\allowdisplaybreaks[4]

\title{Interface fluctuations for $1$D stochastic Allen-Cahn equation revisited}
\author{Weijun Xu$^1$ \quad Wenhao Zhao$^2$ \quad Shuhan Zhou$^3$}
\institute{Beijing International Center for Mathematical Research, Peking University, China. \email{weijunxu@bicmr.pku.edu.cn}
\and EPFL, Switzerland. \email{wenhao.zhao@epfl.ch}
\and Peking University, China. \email{zhoushuhan@stu.pku.edu.cn
}
}

\begin{document}
\maketitle

\begin{abstract}
    We revisit the interface fluctuation problem for the $1$D Allen-Cahn equation perturbed by a small space-time white noise. We show that if the initial data is a standing wave solution to the deterministic equation, then under proper long time scale, the solution is still close to the family of traveling wave solutions. Furthermore, the motion of the interface converges to an explicit stochastic differential equation. This extends the classical result in \cite{Fun95} to full small noise regime, and recovers the result in \cite{BBDMP98}. 

    The proof builds on the analytic framework in \cite{Fun95}. Our main novelty is the construction of a series of functional correctors that are designed to recursively cancel potential divergences. Moreover, to show these correctors are well-behaved, we develop a systematic decomposition of Fr\'echet derivatives of the deterministic Allen-Cahn flow of all orders. This decomposition is of its own interest, and may be useful in other situations as well. 
\end{abstract}

\setcounter{tocdepth}{2}
\tableofcontents

\section{Introduction}

In this article, we study the long-time interface behavior of the following $1$D stochastic Allen-Cahn equation:
\begin{equation}\label{e:main_eqn}
    \partial_t u_\eps =  \Delta u_\eps +  f(u_\eps) + \eps^{\gamma} a_\eps \dot{W}\;.
\end{equation}
Here, $\gamma>0$ is a fixed parameter, $a_\eps = a(\sqrt{\eps} \cdot)$ for some smooth function $a$ with compact support in $(-1,1)$, $\dot{W}$ is a one-dimensional space-time white noise (time derivative of the $\lL^2$ cylindrical Wiener process $W$), and $f$ satisfies the following assumption.

\begin{assumption} \label{as:f}
    The function $f: \RR \rightarrow \RR$ is smooth with all derivatives bounded (including $f$ itself), and satisfies the followings:
    \begin{enumerate}
        \item $f$ has exactly three zeros: $\pm 1$ and $0$, and $f'(\pm 1) < 0$, $f'(0) > 0$. 
        \item $f$ is odd, that is, $-f(u) = f(-u)$. 
    \end{enumerate}
\end{assumption}

The assumption that $f$ has bounded derivatives is mainly for technical simplicity. It can be relaxed to requiring that $f$ has polynomial growth and its first derivative is bounded from above (see \cite[Theorem~2.1]{Fun95} for details). The choice of the cutoff on the interval of length $\oO(1/\sqrt{\eps})$ is to be consistent with the set up in \cite{Fun95}. Indeed, by a change of parameter, one can make the cutoff to be $a \big( \eps^\theta \cdot \big)$ for every $\theta > 0$. 

The Allen-Cahn equation is a popular model to study dynamical phase separation. The deterministic part of \eqref{e:main_eqn} has been well understood. Let $m:\RR\rightarrow[-1,1]$ be the unique increasing solution of the steady-state problem
\begin{equation}
\label{e:m}
    \Delta m + f(m) = 0\;, \quad m(\pm \infty) = \pm 1\;, \quad m(0)=0\;.
\end{equation}
For $\theta \in \RR$, write $m_\theta := m(\cdot - \theta)$. The solution space for \eqref{e:m} without the restriction $m(0) = 0$ is the one dimensional manifold
\begin{equation*}
    \mM:=\{m_\theta = m(\cdot - \theta)\,: \; \; \theta \in \RR\}\;.
\end{equation*}
If there is no noise in the equation, solutions starting close enough to $\mM$ will converge to $\mM$ as $t \rightarrow +\infty$. In particular, solutions starting from $\mM$ do not evolve with time and stay there forever. \cite{CP89, FH89} studied the deterministic part of \eqref{e:main_eqn} on the interval $[-\frac{1}{\sqrt{\eps}}, \frac{1}{\sqrt{\eps}}]$ with Neumann boundary condition, where the only stationary solutions are constant functions $\pm 1$ on that interval. They showed that if one starts with an initial condition with one interface and very close to the restriction of $m_\theta$ on the interval $[-\frac{1}{\sqrt{\eps}}, \frac{1}{\sqrt{\eps}}]$ for some $\theta$, then the interface location moves very slowly at the speed $e^{-c/\sqrt{\eps}}$ for some $c>0$.

The situation is different with the effect of noise. Even though one starts with $u_\eps[0] \in \mM$ for \eqref{e:main_eqn}, the noise will push it away from $\mM$. One expects that at proper long time (polynomial in $\eps$ this time), the solution is still close to $\mM$, but the location of the interface moves according to an approximate diffusion process. 

To investigate the right time scale to observe the above behavior, we note that due to the factor $\eps^{\gamma}$, it takes $\eps^{-2\gamma}$ time for the effect of noise to accumulate to size $\oO(1)$. \cite{BDMP95} showed that there exists a process $\theta_t^\eps$ converging to Brownian motion as $\eps \rightarrow 0$ such that $u[\eps^{-2\gamma} t]$ is close to $m_{\theta_t^\eps}$. On the other hand, since noise exists on the interval $[-\frac{1}{\sqrt{\eps}}, \frac{1}{\sqrt{\eps}}]$, it is possible to look at longer time scale for the interface to move a distance up to $\oO(1/\sqrt{\eps})$. This was achieved in \cite{Fun95}, where the author showed that there exists a process $\theta_t^\eps$ converging in law to a limiting diffusion such that $u[\eps^{-2\gamma-1} t]$ is close to $m_{\theta_t^\eps / \sqrt{\eps}}$. There was a restriction $\gamma>5$ in \cite{Fun95} due to some potentially divergent terms arising from the derivation of the SDE, and one assumes the noise to be sufficiently small to balance them out. 

Later, \cite{BBDMP98} also obtained essentially the same result at time scale $\eps^{-2\gamma-1}$ with $\gamma>0$, without the restriction that $\gamma$ being large. To reach this time scale, the authors developed a coupling technique that allows one to compare the laws of two processes up modulo translation, and hence to control the potentially non-summable errors. These probabilistic techniques are very different from the analytic methods in \cite{Fun95}. 

In this article, we show that the methods in \cite{Fun95} can be refined to also cover the full range $\gamma>0$ under the time scale $\eps^{-2\gamma-1}$. The main novelty is to introduce a series of functional correctors to cancel potential divergences. We first state the main theorem.

\begin{thm} \label{thm:main}
Fix $\gamma>0$ and $\kappa \in (0,\gamma)$ arbitrary. Let $\xi_0\in \RR$, and $u_\eps$ be the solution to \eqref{e:main_eqn} with initial data $u_\eps[0] = m_{\xi_0 / \sqrt{\eps}} \in \mM$. Let $v_{\eps}(t,x) := u_\eps(\eps^{-2\gamma-1}t,x)$. Then there exists a process $\xi^{\eps}_t\in \cC(\RR^+,\RR)$ satisfying both of the following: 
\begin{enumerate}
\item For every $T>0$ and $N>0$, we have
    \begin{equation} \label{e:closeness_process}
    \PP\Big(\sup_{t \in [0,T]} \|v_{\eps}[t] - m_{\xi^{\eps}_t/\sqrt{\eps}}\|_{\lL^{\infty}} > \eps^{\gamma-\kappa}\Big) \lesssim \eps^N
    \end{equation}
    for all $\eps \in (0,1)$. 

\item For every $T>0$, $\xi^{\eps}_t$ converges weakly on $\cC([0,T],\RR)$ to a limiting process $\xi$, which satisfies the It\^o SDE
    \begin{equation} \label{e:limit_sde}
        \md \xi_t = \alpha_1 a(\xi_t)\,\md B_t + \alpha_2 a(\xi_t)a'(\xi_t) \,\md t
    \end{equation}
    with initial data $\xi_0$. Here $B_t$ is the standard Brownian motion, and $\alpha_1$, $\alpha_2$ are constants given in \eqref{e:alpha}. 
\end{enumerate}
\end{thm}

An intermediate step towards proving Theorem~\ref{thm:main} is to show $u_\eps[t]$ being close to the stable manifold $\mM$. We single it out in the following statement.

\begin{thm} \label{thm:Linftycloseness}
Fix $\kappa' \in (0,\gamma)$. There exist $\alpha^* \in (0, \frac{1}{5})$ and $p^*\in[1,+\infty)$ with $\alpha^* p^* > 1$ such that for every $N, N'>0$, we have
\begin{equation*}
    \PP \; \Big( \sup \limits_{t \in [0, \eps^{-N}]} \dist_{\wW^{\alpha^*,p^*}}(u_\eps[t], \mM) > \eps^{\gamma - \kappa'} \Big) \lesssim \eps^{N'}
\end{equation*}
for all $\eps \in (0,1)$. Here, $\dist_{\wW^{\alpha^*, p^*}}(v,\mM)$ denotes the $\wW^{\alpha^*, p^*}$-distance of $v$ from $\mM$ as defined in \eqref{e:W_distance}. 
\end{thm}

\begin{rmk}
    The reason we use $\lL^\infty$-distance in Theorem~\ref{thm:main} and $\wW^{\alpha^*, p^*}$-distance in Theorem~\ref{thm:Linftycloseness} is as follows. The $\lL^\infty$-setting is convenient for most of the analysis regarding the deterministic Allen-Cahn flow, in particular the induction in Theorem~\ref{thm:derivativesOfPsik} below. On the other hand, it is a non-separable space and hence not suitable for applying It\^o's formula. Hence, we use $\wW^{\alpha^*, p^*}$ for the application of It\^o's formula, and then turn back to $\lL^\infty$ for all the subsequent analysis. The validity of this switch in our context is guaranteed by Lemma~\ref{lem:restriction_derivative}, the embedding $\wW^{\alpha^*, p^*} \hookrightarrow \lL^\infty$ if $\alpha^* p^* > 1$, and the relevant $\lL^\infty$-based bounds established in Section~\ref{sec:corrector_construction}. We refer to Section~\ref{sec:proof_of_main} for details. 
\end{rmk}

\begin{rmk}
Theorem~\ref{thm:main} can be reformulated as a sharp interface limit problem as follows. Let
\begin{equation*}
    \bar{u}_\eps (t,x) := u_\eps (t/\eps, x/\sqrt{\eps})\;.
\end{equation*}
Then $\bar{u}_\eps$ satisfies (in law) the equation
\begin{equation*}
    \d_t \bar{u}_\eps = \Delta \bar{u}_\eps + \frac{1}{\eps} f\big(\bar{u}_\eps \big) + \eps^{\gamma-\frac{1}{4}} \, a \, \dot{W}\;, \qquad \bar{u}_\eps [0] = u_\eps(0, \, \cdot/\sqrt{\eps})\;.
\end{equation*}
This is a sharp interface limit formulation with a stochastic perturbation. Let $\sgn_{\xi}$ be the sign function with separation point $\xi$ in the sense that $\sgn_\xi (x) = 1$ for $x \geq \xi$ and $-1$ for $x < \xi$. 

Also, for the same process $\xi_t^\eps$ in Theorem~\ref{thm:main}, one can show \eqref{e:closeness_process} remains true if one replaces $\lL^\infty$-norm by $\lL^p$-norm for sufficiently large $p$ ($p \geq p^*$ in Theorem~\ref{thm:Linftycloseness} suffices). Then by definition of $\bar{u}_\eps$ and scaling properties, we have
\begin{equation*}
    \big\| \bar{u}_\eps [\eps^{-2\gamma} t] - \sgn_{\xi^\eps_t} \big\|_{\lL^p} \leq \eps^{\frac{1}{2p}} \Big( \big\| v_\eps[t] - m_{\xi^\eps_t / \sqrt{\eps}} \big\|_{\lL^p} + \|m_{\xi^\eps_t} - \sgn_{\xi^\eps_t} \|_{\lL^p} \Big) \lesssim \eps^\theta
\end{equation*}
for some $\theta>0$ with high probability. This says the solution $\bar{u}_\eps$ converges to the sign function $\sgn_{\xi^\eps_t}$ with the separation point $\xi^\eps_t$ evolving according to the SDE \eqref{e:limit_sde} in the limit.
\end{rmk}

We now briefly discuss our strategy to prove Theorem~\ref{thm:main}. Since $\wW^{\alpha^*, p^*} \hookrightarrow \lL^\infty$ if $\alpha^* p^* > 1$, by choosing $\kappa'<\kappa$, Theorem~\ref{thm:Linftycloseness} immediately implies the existence of a process $\xi_t^\eps$ such that \eqref{e:closeness_process} holds. However, proving convergence of the interface process to the limiting diffusion is much more challenging. In order to emphasize our main new inputs, we first assume Theorem~\ref{thm:Linftycloseness}, and discuss how to prove the convergence of the interface process to the SDE \eqref{e:limit_sde}. 

As mentioned above, the main challenge at the time scale $\eps^{-2\gamma-1}$ is that there are potentially divergent terms with negative powers of $\eps$ in the derivation of the limiting equation. They arise from the quadratic variation process when applying It\^o's formula. Our main idea to deal with them is to introduce a series of functional correctors, which are constructed to cancel out the potential divergences inductively. The idea is similar in spirit to homogenization, but with significant practical differences. The construction of the correctors here involves solving infinite-dimensional PDEs (for unknown $\Psi: \bB \rightarrow \RR$) of the type
\begin{equ}
    \langle D\Psi(v), \nN(v) \rangle = \hH(v)\;,
\end{equ}
where $\hH:\bB \rightarrow \RR$ is a functional on Banach space $\bB$, and $\nN: \dD(\nN)\rightarrow \bB$ is an (unbounded and nonlinear) operator on $\bB$ with certain nice properties. This is carried out in Section~\ref{sec:convergence}. Similar ideas have been used in \cite{Hai09, Energy_dissipation_slow} to construct Lyapunov function for finite-dimensional Hamiltonian systems with noise, and also in \cite{Lorenz_noise} to ``correct" the choice of Lyapunov function for a stochastically perturbed Lorenz system. One difference in our infinite-dimensional situation is that showing the well-posedness of the``correctors" is much more sophisticated. This is the main content of Sections~\ref{sec:deterministic} and~\ref{sec:proof_of_Theorems}, where we systematically analyze the Fr\'echet derivatives of all orders of the deterministic solution flow. We hope the techniques developed here may play a role in the study of the long-time behavior of other stochastic PDEs (especially singular ones).

\bigskip

We now discuss related works that might point to interesting future questions. 

\medskip

\textbf{Non-odd $f$.} Our analysis crucially relies on $f$ being odd. In particular, the case $n=0$ in Assertion 4 in Theorem~\ref{thm:derivativesOfPsik} relies on $f$ being odd. For non-odd $f$ but still satisfying the balanced condition
\begin{equation*}
    \int_{-1}^{1} f(u) \, \md u = 0\;,
\end{equation*}
the solution $m$ to \eqref{e:m} and its translations still form a stable manifold. But in general one does not expect to be able to reach time scale $\eps^{-2\gamma-1}$. Under the time scaling $\eps^{-2\gamma}$, \cite[Section~5.2]{MR1654352} and \cite{non-symmetric} showed that $u[\eps^{-2\gamma}t]$ is close to $m_{\theta_t^\eps}$, where $\theta_t^\eps$ now converges to $c_1 B_t + c_2 t$ for some $c_1, c_2 \in \RR$. In particular, $c_2=0$ if $f$ is odd. 

\cite[Section~5.1]{MR1654352} also showed that one can further enlarge the time scale by $\eps^{-\frac{1}{2}}$, and $u[\eps^{-2\gamma-\frac{1}{2}} t]$ is close $m_{\theta_t^\eps / \sqrt{\eps}}$, where $\theta_t^\eps$ converges to the solution $\theta_t$ of the deterministic ODE
\begin{equation*}
    \md \theta_t = c_2 a^2 (\theta_t) \, \md t\;.
\end{equation*}
If $c_2 = 0$ (which is the case for odd $f$), one can look at a further time scale by $\eps^{-\frac{1}{2}}$, obtaining the SDE in \cite{Fun95} and in Theorem~\ref{thm:main}. 

\medskip

\textbf{Invariant measures.} \cite{Web10} and \cite{OWW14} studied the invariant measure associated to the dynamics with Dirichlet boundary condition on the interval $[-L_\eps, L_\eps]$: 
\begin{equation} \label{e:invariant}
    \d_t u_\eps = \Delta u_\eps + f(u_\eps) + \eps^\gamma \dot{W}\;, \qquad u_\eps(t, \pm L_\eps) = \pm 1\;.
\end{equation}
The solution $u_\eps$ is naturally extended to a function on $\RR$ by setting its value to be $1$ for $x > L_\eps$, and $-1$ for $x < -L_\eps$. The precise parameters and set up in \cite{Web10} and \cite{OWW14} are different from \eqref{e:invariant}, but can all be translated into that form. 

\cite{Web10} considered the case $L_\eps = \frac{1}{\sqrt{\eps}}$, and showed that for $\gamma>\frac{3}{8}$, the invariant measure for \eqref{e:invariant} (extended to functions on $\RR$ as above) are exponentially concentrated near $\mM$. \cite{OWW14} further explored the situation with an interval of exponential size. More precisely, they showed that there exists $c_*>0$ such that as long as $1 \ll L_\eps \lesssim e^{c \eps^{-2\gamma}}$ with $c< c_*$, then the measure exhibits only one interface with high probability. If in addition $L_\eps \gg |\log \eps|$, then the location of the interface is approximately uniformly distributed. These phenomena are consistent with the dynamics at our polynomial time scale. Hence, it is reasonable to expect a similar phenomenon for the dynamics on intervals of exponential scale $\exp(c\eps^{-2 \gamma})$ with small $c$, and we hope the techniques developed in this paper can be useful towards proving corresponding results. However, it is expected that the behavior will be different if $c$ is large, which would be a very interesting and challenging question. 

\medskip

\textbf{Others.} \cite{Wes21} and \cite{Web14} studied initial data with several interfaces for the deterministic and stochastic Allen-Cahn equations. Such initial data is also studied in \cite{ABK12} for the stochastic Cahn-Hilliard equation. The results in the stochastic case were proved for sufficiently small noise (corresponding to large $\gamma$). 

Finally, we would like to mention the work \cite{limit_cycle} on the general situation where the deterministic part has a limit cycle instead of just one stationary solution as the limit. The setting in \cite{limit_cycle} are for finite-dimensional ODEs. It will be interesting to investigate the (S)PDE situation.

\subsection*{Notations}

For every non-negative integer $k$, every $\alpha \in (k, k+1)$ and $p \in[1,+\infty)$, let $\cC^\alpha(\RR)$ and $\wW^{\alpha,p}(\RR)$ be the standard H\"older and Sobolev spaces with norms
\begin{align*}
    \|g\|_{\cC^\alpha} &:= \sum_{j=0}^{k} \|g^{(j)}\|_{\lL^\infty} + \sup_{x \neq y} \frac{|g^{(k)}(x)-g^{(k)}(y)|}{|x-y|^{\alpha-k}}\;, \\
    \|g\|_{\wW^{\alpha,p}} &:= \sum_{j=0}^k\|g^{(j)}\|_{\lL^p} + \int_{\RR^2} \frac{|g^{(k)}(x) - g^{(k)}(y)|^p}{|x-y|^{1+(\alpha-k) p}} \, \md x \,\md y\;.
\end{align*}
Let $\cC_b^\infty(\RR)$ be the space of smooth functions with bounded derivatives of all orders. Recall $a \in \cC_c^\infty (\RR)$ is a fixed cutoff function with compact support in $(-1,1)$. We write $a_\eps(x) = a(\sqrt{\eps} x)$. For $\vec{y} = (y_1, \dots, y_n)$, we write $|\vec{y}| = \sum_{j} |y_j|$. For every $\lambda \in \RR$ and $p\in[1,+\infty]$, we define the weighted $\lL^p$-norm by 
\begin{equation}\label{e:Llambda}
    \|g\|_{\lL^p_{\lambda}(\RR^n)} := \Vert e^{\lambda |\cdot|}g\Vert_{\lL^p(\RR^n)}=\left(\int_{\RR^n} e^{p\lambda |\vec{y}|} |g(\vec{y})|^p \,\md \vec{y}\right)^{\frac{1}{p}}\;,
\end{equation}
where $p=+\infty$ corresponds to the supremum norm. Define the operators $\{\emS_\theta\}_{\theta \in \RR}$ and $\emR$ by
\begin{equation*}
    (\emS_\theta v)(z) := v(z-\theta)\;, \qquad (\emR v)(z):= - v(-z)\;.
\end{equation*}
Throughout, $m$ denotes the unique increasing solution to the stationary equation \eqref{e:m} subject to $\pm 1$ boundary conditions and centering condition $m(0)=0$. Let
\begin{equation*}
    \mM = \{\emS_\theta m: \theta \in \RR\}
\end{equation*}
be the manifold of translations of $m$. We also write $m_\theta = \emS_\theta m$, and in particular $m_0 = m$. Let
\begin{equation*}
    \dist (v, \mM) := \inf_{\theta} \|v - m_\theta\|_{\lL^\infty}
\end{equation*}
be the $\lL^\infty$-distance of $v$ from $\mM$. For $r > 0$, we write 
\begin{equation*}
    \vV_{r} = \{ v: \dist(v, \mM) < r \}\;.
\end{equation*}
We also write 
\begin{equation}\label{e:W_distance}
    \dist_{\wW^{\alpha,p}}(v,\mM):=\inf_{\theta} \|v - m_\theta\|_{\wW^{\alpha,p}}
\end{equation}
be the $\wW^{\alpha,p}$-distance of $v$ from $\mM$. 
Let $(F^t)_{t\geq 0}$ denote the deterministic Allen-Cahn flow in the sense that $F^t(v)$ satisfies the equation
\begin{equation} \label{e:deterministic_flow}
    \d_t F^t(v) = \Delta F^t(v) + f \big( F^t(v) \big)\;, \qquad F^0(v) = v\;.
\end{equation}
Let $\beta>0$ be such that Propositions~\ref{prop:fermi} and~\ref{pr:LinftyEC} hold. It in particular implies the existence of a functional $\zeta: \vV_\beta \rightarrow \RR$ such that $F^t(v) \rightarrow m_{\zeta(v)}$ exponentially fast in $\lL^\infty$ uniformly in $v \in \vV_\beta$. This functional $\zeta$ will play an essential role throughout the article. For $K\geq 0$, we write
\begin{equation*}
    \vV_{r,K} = \big\{ v \in \vV_r: |\zeta(v)|\leq K \big\}\;.
\end{equation*}
In particular, we have
\begin{equation*}
    \vV_{r,0} = \big\{ v \in \vV_r: \zeta(v)=0 \big\}\;.
\end{equation*}
Another special role is played by the generator of the linearized operator of $F^t$ at $v=m$, which we denote by (with a flipped sign)
\begin{equation} \label{e:op_A}
    \aA = -\Delta - f'(m)\;.
\end{equation}
Let $\pP$ denote the $\lL^2$ projection onto the one dimensional space spanned by $m'$ in the sense that
\begin{equation} \label{e:projection}
    \pP v := \frac{\bracket{v, m'}}{\|m'\|_{\lL^2}^2} \, m'\;,
\end{equation}
and $\pP^\perp := \id - \pP$. By the exponential decay of $m'$ (Lemma~\ref{lem:statationary_exponential_decay}), $\pP$ extends to all locally integrable functions with a small exponential growth. 

For $n \in \NN$ and Banach spaces $\xX$ and $\yY$, let $\emL^n(\xX^n,\yY)$ be the class of all bounded multi-linear maps from $\xX^n$ to $\yY$ equipped with the norm
\begin{equation*}
    \|\Phi\|_{\emL^n(\xX^n,\yY)} :=\sup \Big\{ \|\Phi(g_1,\dots,g_n)\|_\yY : \; \|g_j\|_\xX \leq 1 \text{ for all } 1\leq j \leq n \Big\}\;.
\end{equation*}
We also use bracket notation to write
\begin{equation*}
     \left\langle \Phi, \; (g_1, \dots, g_n) \right\rangle = \Phi (g_1, \dots, g_n)\;.
\end{equation*}
Most of the times, we use $\xX = \lL^\infty(\RR)$ and $\yY = \RR$. In this situation, we say $\Phi \in \emL^n \big( \xX^n, \yY \big)$ has a kernel if there exists $T \in \lL^1 \big( \RR^n,\RR \big)$ such that
\begin{equation} \label{e:kernel_defn}
    \langle \Phi, \; (g_1, \cdots, g_n) \rangle = \int_{\RR^n} T(\vec{y}) \, \prod_{j=1}^{n} g_j (y_j) \, \md \vec{y}\;.
\end{equation}
Finally, we use $c$, $C$ to denote generic constants whose values may change from line to line. We also use the notation $A \lesssim B$ to denote $A \leq C \cdot B$ for some constant $C$ independent of the parameter $\eps$ (and other parameters that are normally clear from the context).

\subsection*{Organization of the article}

The rest of the article is organized as follows. In Section~\ref{sec:convergence}, we give a construction of the above-mentioned functional correctors, and state their two key properties -- well-posedness (Theorems~\ref{thm:derivativesOfPsicork} and~\ref{thm:zeta_continuity_bound}) and cancellation effects (Theorem~\ref{thm:functional_equation}). These are the main novel ingredients of this article. We then prove Theorem~\ref{thm:main} by first assuming these properties as well as Theorem~\ref{thm:Linftycloseness}. In Section~\ref{sec:deterministic}, we develop a systematic decomposition of Fr\'echet derivatives of the deterministic flow of all orders, and use it to prove Theorems~\ref{thm:derivativesOfPsicork}, ~\ref{thm:zeta_continuity_bound} and~\ref{thm:functional_equation} in Section~\ref{sec:proof_of_Theorems}. In Section~\ref{sec:closeness}, we turn back to the proof of Theorem~\ref{thm:Linftycloseness}, which completes the proof of the main result. In the appendices, we collect some useful results used in the main text.

\subsection*{Acknowledgement}

We are very grateful to Tadahisa Funaki for his encouragement during various stages of this work, and for suggesting that homogenization ideas can be useful for proving the convergence in Section~\ref{sec:convergence} as well as for helpful feedback on the first draft of this manuscript. We also thank Martin Hairer for an interesting discussion on the threshold value $\gamma=0$, Yuning Liu and Jiajun Tong for various helpful discussions and for their patience in listening to us, and Paulo Butt\`a for telling us the reference \cite{BCS20}, which contains an elegant proof of Proposition~\ref{prop:Linfty_decay}. 

We are also extremely indebted to two referees for their careful reading of the manuscript and providing sharp and helpful comments. In particular, one referee pointed out some gaps in our previous treatment of Fr\'echet derivatives, as well as giving suggestions on the formulation and presentation of the results. The process of filling the gap as well as modifying the formulation of the results have greatly improved our understanding of the problem and the presentation of the article. The presentation on the scaling heuristic on Page 3 was suggested by the referee. 

W. Xu was supported by the National Science Foundation China via the standard project grant (no. 8200906145) and the Ministry of Science and Technology via the National Key R\&D Program of China (no. 2020YFA0712900). Part of the work was done when the first author was visiting NYU Shanghai in Autumn 2023. We thank its Institute of Mathematical Sciences for hospitality.

\section{Convergence}
\label{sec:convergence}

\subsection{Overview of the proof}
\label{sec:convergence_overview}

To prove Theorem~\ref{thm:main}, we first identify a candidate for the ``interface" of $v_{\eps}[t] = u_\eps[\eps^{-2\gamma-1}t]$. For this purpose, we need the concept of the limiting functional $\zeta$, whose existence and basic properties are stated in Proposition~\ref{pr:LinftyEC} later. Also recall from \eqref{e:deterministic_flow} that $(F^t)_{t \geq 0}$ is the solution flow of the deterministic equation. 

Assuming Theorem~\ref{thm:Linftycloseness} and Proposition~\ref{pr:LinftyEC}, a natural candidate for the rescaled interface process is
\begin{equation*}
    \xi_t^{\eps,0} := \sqrt{\eps} \zeta \big( v_\eps[t] \big)\;,
\end{equation*}
which is the one considered in \cite{Fun95}. Theorem~\ref{thm:Linftycloseness} and Proposition~\ref{pr:LinftyEC} guarantee the well-posedness of $\xi_t^{\eps,0}$, and imply that it satisfies \eqref{e:closeness_process}. The challenge is to prove its convergence to the SDE \eqref{e:limit_sde}. Note that $v_\eps$ satisfies (in law)
\begin{equation} \label{e:v_eps_eqn}
    \d_t v_\eps = \eps^{-2\gamma-1} \big( \Delta v_\eps + f(v_\eps) \big) + \eps^{-\frac{1}{2}} a_\eps \dot{W}\;, \qquad v_\eps[0] = m_{\xi_0/\sqrt{\eps}}\;.
\end{equation}
By It\^o's formula, one formally sees that $\xi_t^{\eps,0}$ satisfies the SDE
\begin{equation} \label{e:approximate_SDE_funaki}
    \frac{\md}{\md t} \xi_t^{\eps,0} = \bracket{D\zeta(v_{\eps}[t]), a_\eps \dot{W}_t} + \frac{1}{2} b_t^{\eps,0}\;,
\end{equation}
where
\begin{equation*}
    b_t^{\eps,0} = \frac{1}{\sqrt{\eps}} \int_\RR a_\eps^2(y) \, D^2\zeta(y,y;v_\eps[t]) \,\md y\;. 
\end{equation*}
Here, $D \zeta$ and $D^2 \zeta$ denote Fr\'echet derivatives of $\zeta$. The above expression is formal at this stage. In particular, the precise differential structure of the Fr\'echet derivatives needs to be specified, and we have also formally treated $D^2 \zeta (v_\eps[t])$ as a continuous function by evaluating it on the diagonal. These will be addressed in Section~\ref{sec:corrector_construction}. The negative power $\eps^{-2\gamma-1}$ in \eqref{e:v_eps_eqn} does not appear in the drift $b^{\eps,0}_t$ due to the magical cancellation from the identity\footnote{Roughly speaking, the identity \eqref{e:magical_cancellation} is due to the fact that the observable $\zeta$ is invariant under the flow $F^t$ induced by operator $\nN(v) = \Laplace v + f(v)$, and hence $D \zeta$ should be orthogonal to it. See \cite[Theorem~7.4]{Fun95} or Theorem~\ref{thm:functional_equation} below for more details.}
\begin{equ}
\label{e:magical_cancellation}
    \bracket{D \zeta(v), \Delta v + f(v)} = 0\;.
\end{equ}
It is shown in \cite[Lemma~8.2]{Fun95} that the diffusion term in \eqref{e:approximate_SDE_funaki} converges to a proper stochastic integral with respect to Brownian motion. As for the drift $b_t^{\eps,0}$, Taylor expanding $a^2$ near $\xi_t^{\eps,0}$, one gets
\begin{equation} \label{e:natural_drift}
    \begin{split}
    b_t^{\eps,0} = &\eps^{-\frac{1}{2}} a^2 ( \xi_t^{\eps,0} ) \int_{\RR} D^2\zeta (y,y;v_\eps[t]) \,\md y\\
    &+ (a^2)'(\xi_t^{\eps,0}) \int_{\RR} \big( y- \zeta(v_\eps[t]) \big) \; D^2\zeta (y,y;v_\eps[t]) \,\md y + \text{error}_\eps\;.
    \end{split}
\end{equation}
One can also follow the same procedure as in \cite[Lemma~8.1]{Fun95} to show that the second term on the right-hand side above converges to a finite drift term, and $\text{error}_\eps$ vanishes as $\eps \rightarrow 0$. The first term causes trouble due to the $\eps^{-\frac{1}{2}}$ factor. Note that the integral of $D^2 \zeta (y,y;v)$ over $y \in \RR$ is $0$ for $v \in \mM$ due to its anti-symmetry in the $y$ variable (\cite[Corollary~7.1]{Fun95}). If one replaces $v \in \mM$ by $v_\eps[t]$ which is close to $\mM$, one hopes it gives certain smallness to cancel the $\eps^{-\frac{1}{2}}$ factor in front of it. Such smallness is quantified in \cite[Theorem~7.3]{Fun95} in terms of $\dist(v_\eps[t], \mM)$, and this is the main reason for the restriction on $\gamma$ being large in \cite{Fun95}. However, for small $\gamma$, there is no reason to expect $\dist(v_\eps[t], \mM)$ to be small enough to cancel out the $\eps^{-\frac{1}{2}}$ factor, and hence it is unclear whether the first term on the right-hand side of \eqref{e:natural_drift} still vanishes as $\eps \rightarrow 0$.

The idea is to introduce a correction term $\psi^{\eps,1}_{cor}$ which, is itself small on one hand, and when combined with the $\eps^{-2\gamma-1}$ factor in the equation for $v_\eps$, cancels out the potentially divergent term in $b_t^{\eps,0}$ above. To see what $\psi^{\eps,1}_{cor}$ should be, we apply It\^o's formula to $\xi^{\eps,1}_t := \xi_t^{\eps,0} + \sqrt{\eps} \psi^{\eps,1}_{cor}(v_\eps[t])$ to get
\begin{equation} \label{e:xi_formal_eq}
    \begin{split}
    \frac{\md}{\md t} \xi^{\eps,1}_t =& \frac{\md}{\md t} \xi_t^{\eps,0} + \eps^{-2\gamma-\frac{1}{2}}\bracket{D\psi^{\eps,1}_{cor}(v_{\eps}), \Laplace v_{\eps} + f(v_{\eps})}\\
    &+\bracket{D\psi^{\eps,1}_{cor}(v_{\eps}), a_\eps \, \Dot{W}_t} + \frac{1}{2\sqrt{\eps}} \int_\RR a_\eps^2(y) \, D^2 \psi^{\eps,1}_{cor}(y,y;v_{\eps}) \,\md y\;,
    \end{split}
\end{equation}
where we write $v_\eps$ for $v_\eps[t]$ for simplicity. Again, the above expression is still formal, in particular including the evaluation of $D^2 \psi_{cor}^{\eps,1}$ on the diagonal. We will give rigorous statements in Theorem~\ref{thm:derivativesOfPsicork} below. 

Back to \eqref{e:xi_formal_eq}, we seek $\psi^{\eps,1}_{cor}$ so that the potentially bad term in $\frac{\md}{\md t} \xi_t^{\eps,0}$ (first term on the right-hand side of \eqref{e:natural_drift}) is canceled out by the second term on the right-hand side of \eqref{e:xi_formal_eq}. This gives the equation
\begin{equ}
\label{e:corrector_equation}
    \eps^{-2\gamma-\frac{1}{2}}\bracket{D\psi^{\eps,1}_{cor}(v_{\eps}), \Laplace v_{\eps} + f(v_{\eps})}+\frac{a^2_\eps\big(\zeta(v_{\eps})\big)}{2 \sqrt{\eps}}  \int_\RR D^2 \zeta(y,y;v_{\eps})\,\md y = 0\;.
\end{equ}
It turns out that
\begin{equ} \label{e:corrector_formula_1}
    \psi^{\eps,1}_{cor}(v) := \frac{1}{2}\eps^{2\gamma} a^2_\eps\big(\zeta(v)\big) \int_0^{+\infty}\int_{\RR} D^2\zeta \big(y,y;F^t(v)\big)\,\md y \,\md t
\end{equ}
satisfies \eqref{e:corrector_equation}, where we recall from \eqref{e:deterministic_flow} that $(F^t)_{t \geq 0}$ is the solution flow of the deterministic equation. One can see heuristically that the expression \eqref{e:corrector_formula_1} satisfies \eqref{e:corrector_equation} by formally differentiating $\left.\frac{\md}{\md \delta}\right|_{\delta=0}\psi^{\eps,1}_{cor}\big(F^{\delta}(v)\big)$, and noting that $F^t \circ F^\delta = F^{t+\delta}$. 

Assume for the moment that the formula \eqref{e:corrector_formula_1} for $\psi_{cor}^{\eps,1}$ indeed satisfies the functional PDE \eqref{e:corrector_equation}. Then combining \eqref{e:approximate_SDE_funaki} and \eqref{e:natural_drift}, we see \eqref{e:xi_formal_eq} is reduced to
\begin{equation} \label{e:xi_first_order}
    \begin{split}
    \frac{\md}{\md t} \xi_t^{\eps,1} = &\bracket{D \zeta \big( v_\eps[t] \big) + D \psi_{cor}^{\eps,1} \big(v_\eps[t]\big), \, a_\eps \dot{W}_t}\\
    &+ (a^2)' (\xi_t^{\eps,0} ) \int_{\RR} \big( y - \zeta ( v_\eps[t]) \big) \, D^2 \zeta (y,y;v_\eps[t]) \,\md y\\
    &+ \frac{1}{\sqrt{\eps}} \int_{\RR} a_\eps^2(y) \, D^2 \psi_{cor}^{\eps,1}(y,y;v_\eps[t]) \,\md y + \text{error}_\eps\;.
    \end{split}
\end{equation}
The formula \eqref{e:corrector_formula_1} suggests that $D \psi_{cor}^{\eps,1}$ and $D^2 \psi_{cor}^{\eps,1}$ have order $\eps^{2\gamma}$ (which we will prove in a much more general situation later). Hence, one expects the first two terms on the right-hand side above to converge to nontrivial limits as $\eps \rightarrow 0$, and the third term to be of order $\eps^{2\gamma - \frac{1}{2}}$. It will vanish as $\eps \rightarrow 0$ if $\gamma>\frac{1}{4}$. This can indeed be justified and will give a rigorous convergence theorem for all $\gamma>\frac{1}{4}$.

If $\gamma\leq \frac{1}{4}$, the third term on the right-hand side of \eqref{e:xi_first_order} is still potentially divergent. But the potential divergence is now only $\eps^{2\gamma-\frac{1}{2}}$, which has an $\eps^{2\gamma}$ factor improvement than without using the corrector $\psi_{cor}^{\eps,1}$. This suggests that one may introduce a higher-order corrector beyond \eqref{e:corrector_formula_1} to cancel out this potentially divergent term. This will cover the range $\gamma \in (\frac{1}{8}, \frac{1}{4}]$. For even smaller $\gamma$, one can proceed to introduce recursively higher-order correctors, each canceling out the potentially divergent term from the previous one, and creating another potentially divergent term which is $\eps^{2 \gamma}$ better than the previous one. Since $\gamma>0$ is fixed, the procedure necessarily ends with finitely many steps. 

The above discussions give outline of the main strategy of the proof to Theorem~\ref{thm:main}. The rest of this section is organized as follows. In Section~\ref{sec:corrector_construction}, we introduce higher-order correctors $\psi_{cor}^{\eps,\ell}$ and state some of their properties. In Section~\ref{sec:proof_of_main}, we define a modified candidate process $\xi^\eps$ based on these correctors and a regularized and stopped version of $v_\eps$. We then prove the convergence of $\xi^\eps$ to the SDE \eqref{e:limit_sde} by assuming the properties of correctors stated in Section~\ref{sec:corrector_construction}. The verification of these assumed properties is postponed to later sections.

\subsection{Construction of the functional correctors}
\label{sec:corrector_construction}

We first introduce a localized and regularized version of the space-time white noise $\dot{W}$. Let $L_\eps = \frac{2}{\sqrt{\eps}}$. For $k \in \ZZ$, define the $L_\eps$-periodic functions $e_k^\eps$ by
\begin{equation*}
    e_k^\eps (x) := \frac{1}{\sqrt{L_\eps}} e^{2 \pi i \frac{k}{L_\eps} \cdot x}\;, \qquad x \in \RR\;.
\end{equation*}
The family $\{e_k^\eps\}_{k \in \ZZ}$ is an orthonormal basis of $L^2 (\RR / L_\eps \ZZ)$. Let $\{B_k\}_{k \in \ZZ}$ be complex Brownian motions with
\begin{equation*}
    \EE B_k(t) = \EE \big( B_k(t) \big)^2 = 0\;, \quad  \EE |B_k(t)|^2 = t\;, \quad B_{-k}(t) = \overline{B_k(t)}\;,
\end{equation*}
and $B_{k}$ independent of $B_{\ell}$ if $k \neq -\ell$. Define
\begin{equation*}
    W_\eps[t] := \sum_{k \in \ZZ} B_k(t) e_k^\eps\;.
\end{equation*}
Then, $\dot{W}_\eps$ is a space-time white noise on the torus of side length $L_\eps$, and $a_\eps \dot{W} \stackrel{\text{law}}{=}  a_\eps \dot{W}_\eps$. Hence, $v_\eps$ satisfies the equation \eqref{e:v_eps_eqn} (in law) with $\dot{W}$ replaced by $\dot{W}_\eps$. Since only the law of $v_\eps$ is concerned, from now on, we \textit{define} $v_\eps$ as the solution to the equation
\begin{equation} \label{e:v_eps_localized}
    \d_t v_\eps = \eps^{-2\gamma-1} \big( \Delta v_\eps + f(v_\eps) \big) + \eps^{-\frac{1}{2}} a_\eps \dot{W}_\eps\;, \qquad v_\eps[0] \in \mM\;.
\end{equation}
Let $\chi \in \cC_c^\infty (\RR)$ be even and $\chi = 1$ on $[-1,1]$. Fix $N_\eps$ sufficiently large (depending on $\eps$), whose actual value will be specified later. Define $\widetilde{W}_\eps$ by
\begin{equation} \label{e:W_truncated}
    \widetilde{W}_\eps[t] := \sum_{k \in \ZZ} \chi (k / N_\eps) \, B_k(t) \, e_k^\eps\;.
\end{equation}
Then $\widetilde{W}_\eps$ is a regularization of $\widetilde{W}_\eps$. The correlation for its time derivative is
\begin{equation*}
    \EE \big( \dot{\widetilde{W}}_\eps (s, y_1)\dot{\widetilde{W}}_\eps(t, y_2) \big) = \delta (s-t) \rho_\eps (y_1, y_2)\;,
\end{equation*}
where
\begin{equation} \label{e:rho}
    \rho_\eps (y_1, y_2) = \sum_{k \in \ZZ} \chi^2 (k/N_\eps) \, e_k^\eps(y_1) \, e_{-k}^\eps(y_2) = \frac{1}{\sqrt{L_\eps}} \sum_{k \in \ZZ} \chi^2 (k/N_\eps) \, e_k^\eps(y_1 - y_2)\;.
\end{equation}
Define $\widetilde{v}_\eps$ as the solution to 
\begin{equation} \label{e:v_eps_truncated}
    \d_t \widetilde{v}_\eps = \eps^{-2\gamma-1} \big( \Delta \widetilde{v}_\eps + f(\widetilde{v}_\eps) \big) + \eps^{-\frac{1}{2}} a_\eps \dot{\widetilde{W}}_\eps\;, \qquad \widetilde{v}_\eps [0] = v_\eps[0]\;.
\end{equation}
The introduction of the cutoff $N_\eps$ is twofold. First, the solution $\widetilde{v}_\eps[t]$ is smooth for every fixed $\eps$ and $t$, which makes the verification of the identity \eqref{e:magical_cancellation} and its higher order analogues much easier. A general version of this identity is given in Theorem~\ref{thm:functional_equation} below (see also Remark~\ref{rmk:functional_equation_cutoff}). Second, the cutoff also turns the formal evaluation of second order Fr\'echet derivatives on the diagonal (such as the last term in \eqref{e:xi_formal_eq}) to a justified operation (see \eqref{e:D2psi_cor_eps_ell_bound} below). Note that both of these use the qualitative property of the finite dimensionality of the noise but \textit{not} its quantitative bounds. 

Let $\kappa'$ be the parameter in Theorem~\ref{thm:Linftycloseness}. Choose $N_\eps$ sufficiently large such that 
\begin{equation} \label{e:cutoff_close}
    \PP\Big( \sup_{t\in[0, \, |\log\eps|]} \|\widetilde{v}_\eps[t] -v_\eps[t] \|_{\wW^{\alpha^*,p^*}} > \eps^{\gamma-\kappa'}\Big)<e^{-\frac{1}{\eps}}\;,
\end{equation}
and that Lemma~\ref{lem:rho_convergence} for $\rho_\eps$ holds. This can always be achieved if $N_\eps$ is large enough (for example, $N_\eps \geq e^{e^{1/\eps}}$). 

We define the higher-order correctors $\psi_{cor}^{\eps,\ell}$ as follows. For $\ell \geq 0$, define $\bar{\psi}^{\eps,\ell}$ recursively by 
\begin{equation} \label{e:psi_bar_eps_ell}
    \bar{\psi}^{\eps,\ell}(v) := \int_{0}^{+\infty} \langle \, (D^2 \bar{\psi}^{\eps,\ell-1}) \big(F^t(v)\big), \, \rho_\eps \, \rangle \, \md t\;, \qquad \bar{\psi}^{\eps,0} := \zeta\;, 
\end{equation}
where $D^2$ is the second Fr\'echet derivative with respect to the differential strcture induced by $\lL^\infty$. We define the $\ell$-th corrector $\psi_{cor}^{\eps,\ell}$ by
\begin{equation} \label{e:psi_cor_eps_ell}
    \psi_{cor}^{\eps,\ell}(v) := \Big( \frac{1}{2} \eps^{2\gamma} a_\eps^2 \big( \zeta(v) \big) \Big)^\ell \; \bar{\psi}^{\eps,\ell}(v)\;, \quad \ell \geq 0\;.
\end{equation}
Since $\rho_\eps$ is a finite linear combination of products of functions of one variable, the pairing operation in \eqref{e:psi_bar_eps_ell} is valid if $\bar{\psi}^{\eps,\ell-1}$ is twice Fr\'echet differentiable. The following theorem says this is indeed the case together with more quantitative information on $D^{k} \psi^{\eps,\ell}_{cor}$ for $k=0,1,2$.

\begin{thm} \label{thm:derivativesOfPsicork}
For every $\ell \geq 0$, there exists $\beta_\ell > 0$ such that $\bar{\psi}^{\eps,\ell}$ and $\psi_{cor}^{\eps,\ell}$ given in \eqref{e:psi_bar_eps_ell} and \eqref{e:psi_cor_eps_ell} are well-defined in $\vV_{\beta_\ell}$, and satisfy the bound
\begin{equation} \label{e:psi_cor_eps_ell_bound}
    \big| \psi_{cor}^{\eps,\ell}(v) \big| \lesssim \eps^{2\ell \gamma}\;, \qquad \ell \geq 1\;.
\end{equation}
Moreover, $\bar{\psi}^{\eps,\ell}$ and $\psi_{cor}^{\eps,\ell}$ are twice Fr\'echet differentiable in $\vV_{\beta_\ell}$. Their Fr\'echet derivatives
\begin{equation*}
    D^k \bar{\psi}^{\eps,\ell}\,, \; D^k \psi_{cor}^{\eps,\ell} : \; \; \vV_{\beta_\ell} \; \rightarrow \; \emL^k \big( (\lL^\infty)^k, \RR \big)\;, \qquad k = 1,2
\end{equation*}
have kernels for every $v \in \vV_{\beta_\ell}$ in the sense of \eqref{e:kernel_defn}. Furthermore, for every $p\geq 0$, the kernels of $D^k \psi_{cor}^{\eps,\ell}(v)$ (still with the same notation) satisfy the bounds
\begin{equation}\label{e:Dpsi_cor_eps_ell_bound}
    \sum_{i=1}^{2} \int_{\RR^2} \big|y_i-\zeta(v) \big|^p \cdot \Big(\prod_{j=1}^{2} \big| D\psi_{cor}^{\eps,\ell_j}(y_j;v) \big| \Big) \cdot \big|\rho_\eps(y_1,y_2)\big|\,\md y_1\,\md y_2 \lesssim \eps^{2(\ell_1+\ell_2)\gamma}\;,
\end{equation}
\begin{equation} \label{e:D2psi_cor_eps_ell_bound}
    \sum_{i=1}^{2} \int_{\RR^2} \big|y_i-\zeta(v) \big|^p \cdot \big| D^2 \psi_{cor}^{\eps,\ell}(y_1,y_2;v) \big| \cdot \big|\rho_\eps(y_1,y_2)\big|\,\md y_1\,\md y_2 \lesssim \eps^{2\ell\gamma}\;.
\end{equation}
All the bounds are uniform over $v \in \vV_{\beta_\ell}$ for the relevant $\ell$ (and over $v \in \vV_{\beta_{\ell_1} \wedge \beta_{\ell_2}}$ for \eqref{e:Dpsi_cor_eps_ell_bound}). 
\end{thm}

It turns out that formulas can be derived for certain quantities involving the first two Fréchet derivatives of \( \psi^{\varepsilon,0}_{cor} = \zeta \) evaluated at \( v \in \mathcal{M} \). These quantities correspond to those appearing in the limiting SDE in Theorem~\ref{thm:main}. Recall that the cutoff $N_\eps$ is sufficiently large such that Lemma~\ref{lem:rho_convergence} holds.

\begin{thm} \label{thm:zeta_continuity_bound}
For every $\theta \in \RR$, the kernels of $D \zeta (m_\theta)$ and $D^2 \zeta (m_\theta)$ are continuous functions. The constants
\begin{equation} \label{e:alpha}
    \alpha_1:= \bigg( \int_{\RR} \big| D \zeta (y; m_\theta) \big|^2 \,\md y \bigg)^{\frac{1}{2}}\;, \quad \alpha_2 := \int_\RR \big( y-\theta \big) \, D^2 \zeta (y,y;m_\theta) \,\md y
\end{equation}
are well-defined and independent of $\theta$. They have the expressions 
\begin{equation*}
    \alpha_1 = \frac{1}{\|m'\|_{\lL^2}}\;, \quad \alpha_2= - \frac{1}{\Vert m'\Vert_{\lL^2}}\int_0^{+\infty} \int_{\RR^2}y p_{t}^2(y,z;m) f''(m(z))m'(z)\,\md y\,\md z\,\md t\;,
\end{equation*}
where $p_t(y,z;m)$ is defined in \eqref{e:p_t} later. Furthermore, we have 
\begin{equation}\label{e:Dzetacloseness}
    \Big| \int_{\RR^2}  \Big( \prod_{j=1}^{2} D \zeta (y_j; v) \Big) \cdot \rho_\eps (y_1, y_2) \,\md y_1 \,\md y_2 -\alpha_1^2 \Big| \lesssim \dist(v,\mM) + \eps\;,
\end{equation}
\begin{equation}\label{e:D2zetaclossness}
    \Big| \frac{1}{2} \int_{\RR^2} \Big(\sum_{i=1}^{2} \big(y_i- \zeta(v) \big) \Big) \cdot D^2\zeta(y_1,y_2;v) \cdot \rho_\eps (y_1,y_2) \, \md y_1 \,\md y_2 - \alpha_2\Big| \lesssim\dist(v,\mM)+\eps
\end{equation}
uniformly over $v\in\vV_{\beta_0}$. 
\end{thm}

The following theorem guarantees that \( \{\psi_{cor}^{\varepsilon,\ell}\}_{\ell \geq 0} \) indeed act as correctors with the desired cancellation effect. For notational simplicity, we also write $ \bar{\psi}^{\varepsilon,-1} = \psi^{\varepsilon,-1}_{cor} = 0 $.

\begin{thm} \label{thm:functional_equation}
    Let $\beta_\ell > 0$ be as in Theorem~\ref{thm:derivativesOfPsicork}. For $\ell\geq 0$ and $v\in\vV_{\beta_{\ell}}\cap\cC^{5/2}$, we have the identity
    \begin{equation} \label{e:functional_equation}
        \bracket{D \bar{\psi}^{\eps,\ell}(v), \Delta v + f(v)} + \bracket{  D^2 \bar{\psi}^{\eps,\ell-1} (v), \rho_\eps }= 0\;.
    \end{equation}
    As a consequence, one has
    \begin{equation} \label{e:corrector_cancel_general}
    \Big| \, \bracket{D\psi_{cor}^{\eps,\ell}(v), \Delta v + f(v)}+ \frac{\eps^{2\gamma}}{2} \, a^2_\eps\big(\zeta(v)\big) \cdot \bracket{D^2 \psi_{cor}^{\eps,\ell-1}(v), \rho_\eps} \, \Big| \lesssim \eps^{2 \ell \gamma + \frac{1}{2}} \mathbf{1}_{\ell \geq 2}
    \end{equation}
    uniformly over $v \in \vV_{\beta_\ell} \cap \cC^{5/2}$.
\end{thm}

\begin{rmk} \label{rmk:functional_equation_cutoff}
The original process $v_\eps[t]$ given in \eqref{e:v_eps_localized} belongs to $\cC^{\alpha}$ for every $\alpha<\frac{1}{2}$. With more technical efforts, it is possible to show that the kernel $D \bar{\psi}^{\eps,\ell}(v)$ belongs to $\wW^{\alpha',1}$ for all $\alpha' \in (\frac{3}{2},2)$. Thus, one can establish the identity \eqref{e:functional_equation} for $v \in \cC^{\alpha}$ for $\alpha \in (0,\frac{1}{2})$ and use it for $v = v_\eps[t]$. 

On the other hand, the $\cC^{5/2}$ regularity assumption on $v$ makes the justification of \eqref{e:functional_equation} much more straightforward, without the need to show the additional regularity of $D \bar{\psi}^{\eps,\ell}$. It only uses the qualitative fact that $v$ has more than two derivatives but not any quantitative information on its $\cC^{5/2}$ norm. We will use Theorem~\ref{thm:functional_equation} for $v = \widetilde{v}_\eps [t]$, which does belong to $\cC^{5/2}$ for every $\eps>0$.
\end{rmk}

\begin{rmk}
    The above construction leverages the first-order term in \eqref{e:xi_formal_eq} to cancel the potentially divergent first term in \eqref{e:natural_drift} and proceeds inductively. This gives the recursive identity \eqref{e:functional_equation}. One may wonder if it is possible to use together the first and second-order terms in \eqref{e:xi_formal_eq}. This leads to the second-order functional PDE (for $\psi_{cor}^\eps$)
    \begin{equation} \label{e:infinite_pde_2nd}
        \begin{split}
        &\frac{1}{2\sqrt{\eps}} \int_\RR a_\eps^2(y) \, D^2 \psi^\eps_{cor}(y,y;v) \,\md y + \eps^{-2\gamma-\frac{1}{2}}\bracket{D\psi^\eps_{cor}(v), \Laplace v + f(v)}\\
        &+\frac{a^2_\eps\big(\zeta(v)\big)}{2 \sqrt{\eps}}  \int_\RR D^2 \zeta(y,y;v)\,\md y = 0\;.
        \end{split}
    \end{equation}
    Let $\Phi_\eps^t(v)$ be the solution to \eqref{e:main_eqn} at time $t$ with initial data $v$. At least formally, one can check the expression 
    \begin{equation*}
        \psi_{cor}^\eps (v) := \frac{\eps^{2\gamma}}{2} \EE \int_{0}^{+\infty} \int_\RR a_\eps^2 \left( \zeta \big( \Phi_\eps^t (v) \big) \right) \cdot D^2 \zeta \big(y,y; \Phi_\eps^t(v) \big) \, \md y \, \md t
    \end{equation*}
    satisfies \eqref{e:infinite_pde_2nd}. But making this procedure rigorous requires understanding exponential mixing of the original SPDE \eqref{e:main_eqn} with noise acting on an interval of size $\eps^{-\frac{1}{2}}$. We leave it for further investigation. 
\end{rmk}

The statements in Theorems~\ref{thm:derivativesOfPsicork}, ~\ref{thm:zeta_continuity_bound} and~\ref{thm:functional_equation} seem heuristically clear from the expressions \eqref{e:psi_bar_eps_ell} and \eqref{e:psi_cor_eps_ell}. However, the expression \eqref{e:psi_bar_eps_ell} is only formal since it involves integration in whole time, and even the well-posedness of this expression is not a priori clear. This is part of Theorem~\ref{thm:derivativesOfPsicork}. The rigorous justification of these theorems turns out to be nontrivial, and involves detailed analysis of all high-order derivatives of the deterministic flow $(F^t)_{t \geq 0}$. Hence we defer their proofs to the next sections, and proceed by assuming them for now.

\subsection{Convergence of the interface process -- proof of Theorem~\ref{thm:main}} \label{sec:proof_of_main}

We now prove Theorem~\ref{thm:main} by assuming Theorem~\ref{thm:Linftycloseness}, as well as Theorems~\ref{thm:derivativesOfPsicork}, ~\ref{thm:zeta_continuity_bound} and~\ref{thm:functional_equation} on the correctors. With these theorems, the proof for the convergence follows the same procedure as in \cite[Section~8]{Fun95}. 

Recall the localized and truncated noise $\widetilde{W}_\eps$ from \eqref{e:W_truncated}, and the process $\widetilde{v}_\eps$ from \eqref{e:v_eps_truncated} with stochastic forcing $\widetilde{W}_\eps$. We first introduce a stopped version of $\xi^\eps$ based on $\widetilde{v}_\eps$ and the higher-order correctors. Let $\alpha^*$ and $p^*$ be as in Theorem~\ref{thm:Linftycloseness}. Fix $0<\kappa'<\kappa<\gamma$. Define
\begin{equation*}
    \tau_\eps:= \inf \Big\{t>0:\dist_{\wW^{\alpha^*,p^*}}(\widetilde{v}_\eps[t], \mM) > 2\eps^{\gamma-\kappa'}\Big\} \;.
\end{equation*}
Let $L$ be the smallest integer that is bigger than $\frac{1}{4\gamma}$. For $t > 0$, define
\begin{equation}\label{e:xiteps}
    \xi_t^\eps := \sqrt{\eps} \, \sum_{\ell=0}^{L} \psi_{cor}^{\eps,\ell} \big( \widetilde{v}_\eps [t \wedge \tau_\eps] \big)\;,
\end{equation}
where $\psi_{cor}^{\eps,\ell}$ is defined in \eqref{e:psi_cor_eps_ell} with $\psi_{cor}^{\eps,0} =\zeta$. 

We first show that the process $\xi^\eps$ satisfies \eqref{e:closeness_process}. Fix $T>0$ and $N>0$ arbitrary. By Theorem~\ref{thm:Linftycloseness} and \eqref{e:cutoff_close}, we have $\PP(\tau_\eps \leq T) \lesssim \eps^N$. Together with the bound \eqref{e:cutoff_close} and the embedding $\wW^{\alpha^*, p^*} \hookrightarrow \lL^\infty$, we have
\begin{equation*}
    \PP \Big( \sup_{t \in [0,T]} \|v_\eps[t] - \widetilde{v}_\eps (t \wedge \tau_\eps)\|_{\lL^\infty} > C \eps^{\gamma-\kappa'} \Big) \lesssim \eps^N\;.
\end{equation*}
Also, by the bound \eqref{e:deterministic_flow_convergence} with $t=0$, the bound \eqref{e:psi_cor_eps_ell_bound} and the embedding $\wW^{\alpha^*, p^*} \hookrightarrow \lL^\infty$, we have
\begin{equation*}
    \|\widetilde{v}_\eps[t \wedge \tau_\eps] - m_{\xi_t^\eps / \sqrt{\eps}} \|_{\lL^\infty} \lesssim \dist_{\wW^{\alpha^*, p^*}} \big( \widetilde{v}_\eps[t\wedge\tau_\eps], \mM \big) + \sum_{\ell=1}^{L} |\psi^{\eps,\ell} \big( \widetilde{v}_\eps[t \wedge \tau_\eps] \big)| \lesssim \eps^{\gamma-\kappa'}
\end{equation*}
for all $t \geq 0$. The claim \eqref{e:closeness_process} then follows. 

To prove weak convergence of $\xi^\eps$ to the solution of the SDE \eqref{e:limit_sde}, we need to apply It\^o's formula to $\xi^\eps$. The map $\psi_{cor}^{\eps,\ell}: \vV_{\beta_\ell} \rightarrow \RR$ originally defined on an open set of $\lL^\infty$ is not suitable for this purpose since $\lL^\infty$ is not separable. Hence we will use the fact that $\widetilde{v}_\eps$ is also in a small $\wW^{\alpha^*, p^*}$-neighborhood of $\mM$, which is a UMD space (see \cite[Example~4.2.18]{MR3617205}) and It\^o's formula (\cite[Theorem~2.4]{MR2422709}) can be applied. 

Recall $\beta_\ell>0$ from Theorem~\ref{thm:derivativesOfPsicork}. Let $\beta_{\min} := \min\limits_{0 \leq \ell \leq L} \beta_\ell$. For $r>0$, let $\widetilde{\vV}_r$ be the set of functions whose $\wW^{\alpha^*, p^*}$ distance from $\mM$ is less than $r$. By Sobolev inequality, $\widetilde{\vV}_r \subset \vV_{\beta_{\min}}$ for sufficiently small $r$. 

Recall from Theorem~\ref{thm:derivativesOfPsicork} the well-posedness and twice Fr\'echet differentiability of $\psi_{cor}^{\eps,\ell}$ in $\vV_{\beta_{\min}}$. By Lemma~\ref{lem:restriction_derivative}, for sufficiently small $r$, the restriction of $\psi_{cor}^{\eps,\ell}$ to $\widetilde{\vV}_r$ is also twice Fr\'echet differentiable with respect to the structure induced by $\wW^{\alpha^*, p^*}$. We denote the $\wW^{\alpha^*, p^*}$ induced Fr\'echet derivatives by $D_*^k \psi_{cor}^{\eps,\ell}$. Hence, applying It\^o's formula to $\xi_t^\eps$, we get
\begin{equation*}
    \xi_t^\eps = \xi_0 + \mu^\eps_{t\wedge\tau_\eps} +\frac{1}{2}\int_0^{t\wedge\tau_\eps} b^\eps_s \,\md s\;,
\end{equation*}
where
\begin{equation}\label{e:muteps}
    \mu^\eps_t = \sum_{\ell=0}^{L} \int_0^t \left\langle D_*\psi^{\eps,\ell}_{cor} (\widetilde{v}_{\eps}[s]), \, a_\eps \,\md \widetilde{W}_\eps[s] \right\rangle\;,
\end{equation}
and
\begin{equation} \label{e:bteps_W}
    b_t^\eps =  \sum_{\ell=0}^{L} \Big( 2 \eps^{-2\gamma-\frac{1}{2}} \left\langle D_*\psi_{cor}^{\eps,\ell} (\widetilde{v}_\eps), \, \Delta \widetilde{v}_\eps + f(\widetilde{v}_\eps) \right\rangle + \frac{1}{\sqrt{\eps}}  \langle D_*^2 \psi_{cor}^{\eps,\ell}(\widetilde{v}_\eps)\,, \, a_\eps^{\otimes 2} \rho_\eps \rangle \Big)\;.
\end{equation}
Here, $\mu_t^\eps$ and $b_t^\eps$ are both defined for $t \leq \tau_\eps$. Also, all appearances $\widetilde{v}_\eps$ in $b_t^\eps$ are evaluated at the time $t$, which we omit in notation for simplicity. The notation $a_\eps^{\otimes 2}$ standards for the function $a_\eps^{\otimes 2} (y_1, y_2) = a_\eps(y_1) a_\eps(y_2)$. The parings on the right-hand sides of \eqref{e:muteps} and \eqref{e:bteps_W} are between elements in $(\wW^{\alpha^*, p^*})^{k}$ and the linear functionals on it $(D_*^k \psi_{cor}^{\eps,\ell})$. 

By Theorem~\ref{thm:derivativesOfPsicork} and Lemma~\ref{lem:restriction_derivative}, the Fr\'echet differentiation $D_*^k$ with respect to the structure induced by $\wW^{\alpha^*, p^*}$ in both \eqref{e:muteps} and \eqref{e:bteps_W} can be replaced by the Fr\'echet differentiation $D^k$ with respect to the structure induced by $\lL^\infty$. In particular, we can re-write $b_t^\eps$ as
\begin{equation} \label{e:bteps}
    \begin{split}
    &\phantom{11111}b_t^\eps = 2 \, \eps^{-2\gamma-\frac{1}{2}} \sum_{\ell=0}^{L} \bigg( \left\langle D\psi_{cor}^{\eps,\ell} (\widetilde{v}_\eps), \, \Delta \widetilde{v}_\eps + f(\widetilde{v}_\eps) \right\rangle\\
    &\phantom{111111111111}+ \frac{\eps^{2\gamma}}{2} \, a_\eps^2 \big( \zeta(\widetilde{v}_\eps) \big) \left\langle D^2\psi_{cor}^{\eps,\ell-1} (\widetilde{v}_\eps), \, \rho_\eps \right\rangle \bigg)\\
    &+\frac{1}{\sqrt{\eps}} \sum_{\ell=0}^{L-1} \left\langle D^2 \psi_{cor}^{\eps,\ell} (\widetilde{v}_\eps), \, \left( a_\eps^{\otimes 2} - a_\eps^2 \big( \zeta (\widetilde{v}_\eps) \big) \right) \rho_\eps  \right\rangle + \frac{1}{\sqrt{\eps}} \left\langle D^2 \psi_{cor}^{\eps, L} (\widetilde{v}_\eps), \, a_\eps^{\otimes 2} \rho_\eps \right\rangle\;,
    \end{split}
\end{equation}
where we take the convention that $\psi_{cor}^{\eps,\ell-1} = 0$ for $\ell = 0$. The inputs in the bracket operations by $D \psi^{\eps,\ell}_{cor}$ and $D \psi^{\eps,\ell}_{cor}$ are now viewed as elements in $\lL^\infty$ or $(\lL^\infty)^2$. The space $\wW^{\alpha^*, p^*}$ is used only as an intermediate step to derive the SDE with It\^o's formula. Once we have derived the SDE, only $\lL^\infty$ bounds in Theorem~\ref{thm:derivativesOfPsicork} are needed to analyze the quantities \eqref{e:muteps} and \eqref{e:bteps}. For notational simplicity, we sometimes write $\widetilde{v}_\eps = \widetilde{v}_\eps[t]$, and all the bounds hold uniformly over $t \leq \tau_\eps$. 

We now show convergence of the terms on the right-hand sides of \eqref{e:muteps} and \eqref{e:bteps}. The choice of $L$ and construction of the correctors are such that the terms
\begin{equation*}
    \int_0^{t \wedge \tau_\eps} \bracket{D\psi^{\eps,0}_{cor} (\widetilde{v}_{\eps}[s]), \, a_\eps \,\md \widetilde{W}_\eps[s]} \quad \text{and} \quad \frac{1}{\sqrt{\eps}} \left\langle D^2 \psi_{cor}^{\eps,0} \big( \widetilde{v}_\eps \big), \, \left( a_\eps^{\otimes 2} - a_\eps^2 \big( \zeta (\widetilde{v}_\eps) \right) \rho_\eps  \right\rangle
\end{equation*}
in the diffusion and drift are the only ones that converge to a nontrivial limit as $\eps \rightarrow 0$, while all other terms in $\mu^\eps$ and $b^\eps$ vanish in the limit. 

Let $\alpha_1, \alpha_2$ be given by \eqref{e:alpha} in Theorem~\ref{thm:zeta_continuity_bound}. We first consider the drift term $b_t^\eps$ in \eqref{e:bteps}. By Theorem~\ref{thm:functional_equation}, we have
\begin{equation*}
    \eps^{-2\gamma-\frac{1}{2}} \sum_{\ell=0}^{L} \bigg| \left\langle D\psi_{cor}^{\eps,\ell} (\widetilde{v}_\eps), \, \Delta \widetilde{v}_\eps + f(\widetilde{v}_\eps) \right\rangle + \frac{\eps^{2\gamma}}{2} \, a_\eps^2 \big( \zeta(\widetilde{v}_\eps) \big) \left\langle D^2\psi_{cor}^{\eps,\ell-1} (\widetilde{v}_\eps), \, \rho_\eps \right\rangle \bigg| \lesssim \eps^{2\gamma}\;.
\end{equation*}
By \eqref{e:D2psi_cor_eps_ell_bound} with $p=0$, we have
\begin{equation*}
    \frac{1}{\sqrt{\eps}} \Big| \left\langle D^2 \psi_{cor}^{\eps, L} (\widetilde{v}_\eps), \, a_\eps^{\otimes 2} \rho_\eps \right\rangle \Big| \lesssim \eps^{2L\gamma - \frac{1}{2}}\;,
\end{equation*}
which vanishes at a positive power of $\eps$ since $L>\frac{1}{4\gamma}$. For the remaining terms in \eqref{e:bteps}, we first consider the case $\ell\geq 1$. By mean value inequality, we have
\begin{equation} \label{e:mean_value_a2}
    \big|a_\eps(y_1)a_\eps(y_2)-a^2_\eps\big(\zeta(\widetilde{v}_\eps)\big)\big| \lesssim \sqrt{\eps} \, \sum_{i=1}^{2} \big| y_i- \zeta(\widetilde{v}_\eps) \big| \;.
\end{equation}
Then by \eqref{e:D2psi_cor_eps_ell_bound} with $p=1$, we have
\begin{equation*}
    \frac{1}{\sqrt{\eps}} \sum_{\ell=1}^{L-1} \Big| \left\langle D^2 \psi_{cor}^{\eps,\ell} (\widetilde{v}_\eps), \, \left( a_\eps^{\otimes 2} - a_\eps^2 \big( \zeta (\widetilde{v}_\eps) \big) \right) \rho_\eps  \right\rangle \Big| \lesssim \eps^{2\gamma}\;.
\end{equation*}
It remains to consider the quantity with $\ell=0$, which is
\begin{equation} \label{e:drift_0}
    \frac{1}{\sqrt{\eps}}\int_{\RR^2} \Big( a_\eps(y_1)a_\eps(y_2) - a_\eps^2 \big( \zeta(\widetilde{v}_\eps) \big) \Big) \cdot D^2 \zeta(y_1,y_2;\widetilde{v}_\eps) \cdot\rho_\eps(y_1,y_2)\,\md y_1\,\md y_2\;.
\end{equation}
Taylor expanding $a_\eps$ at $\zeta(\widetilde{v}_\eps)$, we have
\begin{equation} \label{e:Taylor_a2_tensor}
    a_\eps(y_1)a_\eps(y_2)-a^2_\eps\big(\zeta(\widetilde{v}_\eps)\big) = \frac{\sqrt{\eps}}{2} \, \big(a^2\big)'\big( \sqrt{\eps} \zeta (\widetilde{v}_\eps) \big) \sum_{j=1}^{2} \big(y_j- \zeta(\widetilde{v}_\eps) \big) + r_t^\eps(y_1,y_2)\;,
\end{equation}
where $r_t^\eps(y_1,y_2)$ satisfies the pointwise bound
\begin{equation*}
    |r_t^\eps(y_1,y_2)|\lesssim \eps \sum_{i=1}^2\big|y_i-\zeta(\widetilde{v}_\eps)\big|^2\;.
\end{equation*}
By \eqref{e:D2psi_cor_eps_ell_bound} with $\ell = 0$ and $p =2$, the contribution from $r^\eps_t$ to \eqref{e:drift_0} is $\oO(\sqrt{\eps})$. As for the first term on the right-hand side of \eqref{e:Taylor_a2_tensor}, by \eqref{e:D2zetaclossness}, we have
\begin{align*}
    \Big| \frac{1}{2} \int_{\RR^2} \Big(\sum_{i=1}^{2} \big(y_i- \zeta(\widetilde{v}_\eps) \big) \Big)\cdot D^2 \zeta(y_1,y_2;\widetilde{v}_\eps) \cdot\rho_\eps(y_1,y_2)\,\md y_1\,\md y_2 -\alpha_2 \Big| \lesssim \dist(\widetilde{v}_\eps,\mM)+\eps\;.
\end{align*}
Combining all the above, we deduce that 
\begin{align*}
    \sup_{t\geq 0} \Big| b^\eps_{t\wedge\tau_\eps}- 2 \alpha_2 \, a\big(\sqrt{\eps}\zeta(\widetilde{v}_\eps[t \wedge \tau_\eps])\big) \, a'\big(\sqrt{\eps}\zeta(\widetilde{v}_\eps[t \wedge \tau_\eps])\big) \Big|
\lesssim \eps^{\frac{1}{2} \, \wedge \, (\gamma-\kappa') \, \wedge \, (2L\gamma-\frac{1}{2})}\;.
\end{align*}
Using \eqref{e:xiteps} and \eqref{e:psi_cor_eps_ell_bound}, we further deduce that
\begin{equation}\label{e:bteps_convergence}
    \sup_{t\geq 0} \vert b^\eps_{t\wedge\tau_\eps}- 2 \alpha_2 \, a(\xi_t^\eps) \, a'(\xi_t^\eps) \vert \lesssim \eps^{\frac{1}{2} \, \wedge\, (\gamma-\kappa') \, \wedge \, (2L\gamma-\frac{1}{2})}\;.
\end{equation}
For the martingale term $\mu_t^\eps$ in \eqref{e:muteps}, the time derivative of its quadratic variation process is given by
\begin{equation*}
   \frac{\md}{\md t} [\mu^\eps]_t = \sum_{\ell_1, \ell_2 = 0}^{L} \int_{\RR^2} \prod_{j=1}^{2} \Big( a_\eps(y_j) D \psi_{cor}^{\eps,\ell_j} \big( y_j; \widetilde{v}_\eps \big) \Big) \cdot \rho_\eps (y_1, y_2) \, \md y_1 \,\md y_2\;, \quad t \leq \tau_\eps\;.
\end{equation*}
We first consider the term with $\ell_1 = \ell_2 = 0$. Recall $\psi_{cor}^{\eps,0} = \zeta$. By the bound \eqref{e:Dpsi_cor_eps_ell_bound} for $\ell_1=\ell_2=0$ and $p=1$ and the mean value inequality \eqref{e:mean_value_a2}, we have
\begin{equation*}
    \Big| \int_{\RR^2} \Big( a_\eps(y_1)a_\eps(y_2) - a_\eps^2 \big( \zeta(\widetilde{v}_\eps) \big) \Big) \cdot\Big(\prod_{j=1}^2 D\zeta(y_j;\widetilde{v}_\eps)\Big)\cdot\rho_\eps(y_1,y_2)\,\md y_1\,\md y_2\Big| \lesssim \sqrt{\eps}\;.
\end{equation*}
Combining it with \eqref{e:Dzetacloseness}, we get
\begin{equation*}
    \Big|\int_{\RR^2} \prod_{j=1}^{2} \Big( a_\eps(y_j) D \zeta \big( y_j; \widetilde{v}_\eps \big) \Big) \cdot \rho_\eps(y_1,y_2)\,\md y_1\,\md y_2 - \alpha_1^2 \, a^2\big(\sqrt{\eps}\zeta(\widetilde{v}_\eps)\big)\Big| \lesssim \eps^{\frac{1}{2} \wedge (\gamma-\kappa')}\;.
\end{equation*}
Moreover, by \eqref{e:Dpsi_cor_eps_ell_bound} with $p=0$, we see that if $\ell_1 \geq 1$ or $\ell_2 \geq 1$, we have
\begin{equation*}
    \int_{\RR^2} \prod_{j=1}^{2} \Big| a_\eps(y_j) D \psi_{cor}^{\eps,\ell_j} \big( y_j; \widetilde{v}_\eps \big) \Big| \cdot |\rho_\eps (y_1, y_2)| \, \md y_1 \,\md y_2 \lesssim \eps^{2\gamma}\;.
\end{equation*}
Together with the bound \eqref{e:psi_cor_eps_ell_bound}, we deduce
\begin{equation}\label{e:muteps_convergence}
    \sup_{t\geq 0} \, \Big\vert \frac{\md}{\md t} [\mu^\eps]_{t \wedge \tau_\eps} - \alpha^2_1a^2(\xi_t^\eps) \Big\vert \lesssim \eps^{\frac{1}{2} \wedge(\gamma-\kappa')}\;.
\end{equation}
In particular, the bounds \eqref{e:bteps_convergence} and \eqref{e:muteps_convergence} give
\begin{equation*}
    \EE |\xi_t^\eps-\xi_s^\eps|^4\lesssim (t-s)^2
\end{equation*}
for all $s,t \in \RR^+$. This implies for every $T>0$, $\xi_t^\eps$ is tight on $\cC \big([0,T],\RR \big)$. The convergence in law of $\xi_t^\eps$ to the SDE \eqref{e:limit_sde} then follows from combining \eqref{e:bteps_convergence} and \eqref{e:muteps_convergence} together with tightness. This concludes the proof of Theorem~\ref{thm:main}.

\section{Analysis of the deterministic flow}
\label{sec:deterministic}

The key to the proof of Theorem~\ref{thm:derivativesOfPsicork} is controlling (the first two) Fr\'echet derivatives of $\bar{\psi}^{\eps,\ell}$. The recursive relation \eqref{e:psi_bar_eps_ell} indicates that it can be boiled down to the study of high-order derivatives of $\zeta$ and $F^t$. Since Fr\'echet derivatives of $\zeta$ are closely related to those of $m_\zeta$, and the latter is the limit of $F^t$ as $t \rightarrow +\infty$, this suggests that the essential object to study is the behavior of high-order derivatives of the deterministic flow $F^t$ for large $t$. 

The first two derivatives of $\zeta$ have been studied in \cite{Fun95}. Building on the arguments there, we introduce new ideas to analyze all high-order derivatives of $F^t$ in a systematic way, which lead to effective bounds on high-order derivatives of $\zeta$. It also simplifies some arguments in \cite{Fun95} at the same time. 

We use $[k]$ to denote the set $\{1,2,\dots, k\}$. For $S \subset [k]$, vector $\vec{y} = (y_1, \dots, y_k)$ and functions $(\varphi_j)_{j \in [k]}$, we write 
\begin{equation*}
    \vec{y}_{S} := (y_j)_{j \in S}\;, \qquad \vec{\varphi}_{S} := (\varphi_j)_{j \in S}\;.
\end{equation*}
These notations are used throughout the rest of the article. 

\subsection{Strategy of the proof} \label{sec:deterministic_overview}

To establish infinite Fr\'echet differentiability of $F^t$ as well as properties of the derivatives, for every $v \in \lL^\infty$, $t > 0$ and $k \geq 1$, we first identify a function $p_t^{(k)}(v): \RR^k \times \RR \rightarrow \RR$ as a candidate for the kernel of $D^k F^t(v)$ in the sense that 
\begin{equ}
    \langle D^k F^t(v), \, \vec{\varphi}_{[k]}  \rangle(z) = \int_{\RR^k} p_{t}^{(k)}(\vec{y},z;v) \, \prod_{j=1}^{k} \varphi_j(y_j) \,\md \vec{y}\;.
\end{equ}
For $v \in \lL^\infty$  and $y \in \RR$, let $p_t(y, \cdot\,; v) = p_t^{(1)}(y, \cdot\,; v)$ be the solution to
\begin{equation} \label{e:p_t}
    \d_t \, p_t (y, \cdot\,; v) = \Delta \, p_t (y, \cdot\,; v) + f'\big( F^t(v) \big) \, p_t (y, \cdot\,; v)\;, \qquad p_0(y, \cdot\,; v) = \delta_y\;.
\end{equation}
One can formally differentiate $F^t(v)$ with respect to $v$ and see that the solution $p_t(v)$ to \eqref{e:p_t} is indeed the natural candidate for $D F^t(v)$. 

Keep formally differentiating $p_t$ with respect to $v$, for every $k\geq 2$ and $\vec{y} \in \RR^k$, one obtains an equation for $u^{(k)}(t,\cdot) := p_t^{(k)}(\vec{y},\cdot \, ;v)$ of the form
\begin{equ}
\label{e:high_order_derivatives}
    \d_t u^{(k)} = \Laplace u^{(k)} + f'\big(F^t(v)\big) \, u^{(k)} + g^{(k)}_t(v)\;, \quad\quad u^{(k)}[0]=0\;,
\end{equ}
where $g^{(k)}_t(v)$ is a combination of products of $p_{t}^{(j)}(v)$ for $j \leq k-1$. In Proposition~\ref{prop:linearized_PDE} we show how to analyze such an equation. Roughly speaking, for short time we use the fact that the kernel of $\Laplace + f'\big(F^t(v)\big)$ behaves like a Gaussian. For the long-time behavior, we exploit the fact that $F^t(v)$ converges to $m_{\zeta(v)}$ exponentially fast if $v \in \vV_\beta$, so that $\Laplace + f'\big(F^t(v)\big)$ behaves like $\Laplace + f'(m_{\zeta(v)})$, whose spectrum is well-studied. 

In the rest of the section, we first give some essential facts on the deterministic flow $(F^t)_{t \geq 0}$ in Section~\ref{sec:deterministic_preliminaries}. In Section~\ref{sec:linear_PDE}, we prove Proposition~\ref{prop:linearized_PDE}, which gives a systematic analysis of the generalized form of \eqref{e:high_order_derivatives}. We then use it to analyze $p_t^{(k)}$ by induction in Section~\ref{sec:kernel}, where we will further assume $v\in\vV_\beta$. Finally in Proposition~\ref{prop:Ft_derivative}, we establish that $F^t$ is indeed infinitely Fr\'echet differentiable in the desired sense by checking that $p_t^{(k)}$ does act as the kernel of its $k$-th derivative. Note that the Fr\'echet differentiability of $F^t$ itself does not require the quantitative properties of $p_t^{(k)}$'s developed in Sections~\ref{sec:linear_PDE} or~\ref{sec:kernel}, but the presentation of its proof uses many of the notations developed in the systematic analysis in the two previous subsections. So we put it in Section~\ref{sec:Ft_derivative}.

\subsection{Preliminaries}
\label{sec:deterministic_preliminaries}
We make some preparations before we start. 
\begin{lem} \label{lem:statationary_exponential_decay}
    The function $m$ defined in \eqref{e:m} is smooth. There exist $c>0$ and $A_+, A_-\in \RR$ such that 
    \begin{equ}
    \label{e:m_infinite_behavior}
        \lim_{x\rightarrow + \infty} e^{cx}(1-m(x)) = A_+\;, \quad \quad \lim_{x\rightarrow - \infty} e^{-cx}(-1-m(x)) = A_-\;.
    \end{equ}
    Moreover, for the same $c$ above, for every $k\in\NN$, we have
    \begin{equ}
    \label{e:m_exponential_decay}
        |m^{(k)}(x)|\lesssim_k e^{-c |x|}\;.
    \end{equ}
\end{lem}
\begin{proof}
    Smoothness of $m$, \eqref{e:m_infinite_behavior} and the exponential decay \eqref{e:m_exponential_decay} for $k=1$ are proved in \cite{BCS20}. \eqref{e:m_exponential_decay} for $k = 2$ follows from \eqref{e:m_infinite_behavior} and that $m'' = -f(m)$. 
    
    Differentiating the above equation for $m$, we get 
    \begin{equ}
    \label{e:m_prime_equation}
        \Delta m' + f'(m)m' = 0\;.
    \end{equ}
    \eqref{e:m_exponential_decay} for $k=3$ then follows from the exponential decay for $k=1$ and \eqref{e:m_prime_equation}. For general $k\geq 3$, the result follows from differentiating \eqref{e:m_prime_equation} repeatedly and induction. 
\end{proof}

Recall the operators $\aA$ and $\pP$ defined in \eqref{e:op_A} and \eqref{e:projection}. Let $e^{-t \aA} g$ denote the solution $u$ to the linear PDE
\begin{equation*}
    \d_t u = - \aA u
\end{equation*}
with initial data $u[0] = g$. The following statement regarding the exponential decay of $e^{-t\aA}$ outside the image of $\pP$ is essential.

\begin{prop} \label{prop:Linfty_decay}
    There exist $c>0$ and $C>0$ such that for all $p \in [1,+\infty]$, we have
    \begin{equ}
        \|e^{-t\aA} - \pP\|_{\lL^p \to \lL^p} \leq C e^{-c t} \;.
    \end{equ} 
\end{prop}
\begin{proof}
    The case for $p=2$ is well-known; see for example \cite[Lemma~3.1]{Fun95}. For $p=+\infty$, we have
    \begin{equation} \label{e:Linfty_decay_BCS}
        \|(e^{-t \aA} - \pP) \phi\|_{\lL^\infty} \lesssim e^{-ct} \, \|(\id - \pP) \phi\|_{\lL^\infty} \lesssim e^{-ct} \, \|\phi\|_{\lL^\infty}\;.
    \end{equation}
    The first bound is the content of \cite[Section~2, Eq.(23)]{BCS20}, and the second inequality follows from the definition of $\pP$ in \eqref{e:projection} and properties of $m'$ in Lemma~\ref{lem:statationary_exponential_decay}. By Riesz-Thorin interpolation theorem, we get bound for all $p\geq 2$. A duality argument proves the case for $1\leq p \leq 2$. 

    For the sake of completeness, we briefly sketch the argument in \cite[Section~2]{BCS20} for the first bound in \eqref{e:Linfty_decay_BCS}. Let
    \begin{equation*}
        \lL := \frac{\md^2}{\md x^2} - \uU' \, \frac{\md}{\md x}\;, \qquad \uU = - 2 \log (m')\;.
    \end{equation*}
    Then we have $-\aA \phi = m' \lL \big( \frac{\phi}{m'} \big)$. The diffusion process generated by $\lL$ has the unique invariant measure
    \begin{equation*}
        \rho (\md x) = \frac{1}{\|m'\|_{\lL^2}^2} \, e^{-\uU} \, \md x = \frac{1}{\|m'\|_{\lL^2}^2} \cdot (m')^2 \, \md x\;.
    \end{equation*}
    By Lemma~\ref{lem:statationary_exponential_decay}, there exists $c>0$ such that the function $V = e^{\frac{\uU}{2}} = \frac{1}{m'}$ satisfies $V \geq c_1 > 0$, $V(x) \rightarrow +\infty$ as $|x| \rightarrow +\infty$, and
    \begin{equation*}
        \lL V \leq - c_2 V + K
    \end{equation*}
    for some $c_2, K>0$. It is also straightforward to verify the ``minorization" property (see \cite[Assumption~2]{Hairer-Mattingly}). Hence, by \cite[Theorem~1.2]{Hairer-Mattingly}, there exists $c>0$ such that
    \begin{equation*}
        \|e^{t \lL} \phi - \bracket{\rho, \phi}\|_{V,\infty} \lesssim e^{-ct} \|\phi - \bracket{\rho, \phi}\|_{V,\infty}\;,
    \end{equation*}
    where $\|\phi\|_{V,\infty} := \|\frac{\phi}{V}\|_{\lL^\infty} = \|m' \phi\|_{\lL^\infty}$. Since
    \begin{equation*}
        e^{-t \aA} \phi = m' \cdot e^{t \lL} \Big( \frac{\phi}{m'} \Big)\;, \qquad \pP \phi = \Big\langle \rho, \, \frac{\phi}{m'} \Big\rangle \cdot m'\;,
    \end{equation*}
    we have
    \begin{equation*}
        \begin{split}
        \|(e^{-t \aA} - \pP) \phi\|_{\lL^\infty} &= \Big\| e^{t \lL} \Big( \frac{\phi}{m'} \Big) - \Big\langle \rho, \, \frac{\phi}{m'} \Big\rangle \Big\|_{V, \infty}\\
        &\lesssim e^{-ct} \Big\| \frac{\phi}{m'} - \Big\langle \rho, \, \frac{\phi}{m'} \Big\rangle \Big\|_{V, \infty} = e^{-ct} \|(\id - \pP) \phi\|_{\lL^\infty}\;.
        \end{split}
    \end{equation*}
    This completes the proof of the first bound in \eqref{e:Linfty_decay_BCS}. 
\end{proof}

The following functional $\eta$ was introduced in \cite{Fun95} as the Fermi coordinate of $v$, and in \cite{BDMP95, BBDMP98} as linear center. It generalizes the concept of $\lL^2$ projection onto $\mM$. This functional plays an important role in the existence of the limiting functional $\zeta$ (Proposition~\ref{pr:LinftyEC}) and its differentiability properties.

\begin{prop} \label{prop:fermi} 
    There exists $\beta>0$ such that for every $v \in \vV_\beta$, there exists a unique $\eta = \eta(v) \in \RR$ such that
    \begin{equation} \label{e:fermi_defn}
        \int_\RR v(x) m_\eta'(x) \,\md x = 0\;.
    \end{equation}
    Furthermore, there exists $C>0$ such that if $v \in \vV_\beta$ and $\theta \in \RR$ are such that $\|v - m_\theta\|_{\lL^\infty} < \beta$, 
    then we have the bound
    \begin{equation*}
        |\eta(v) - \theta| \leq C \|v - m_\theta\|_{\lL^\infty}\;.
    \end{equation*}
\end{prop}
\begin{proof}
    Postponed to Appendix~\ref{sec:fermi}. 
\end{proof}

We are now ready to introduce the limiting functional $\zeta$ for the deterministic flow $F^t$. An important feature of $F^t$ is that if the initial data $v$ is sufficiently close to the stable manifold $\mM$, then there is a unique $\zeta(v) \in \RR$ such that $F^t(v)$ converges to $m_{\zeta(v)} \in \mM$ in certain sense as $t \rightarrow +\infty$. A precise version we need is the following proposition. 

\begin{prop} \label{pr:LinftyEC}
    There exist $\beta>0$, $c>0$ and $C_{\ex} > 1$ such that for all $v\in\vV_{\beta}$, there exists a $\zeta(v)\in\RR$ such that
    \begin{equation} \label{e:deterministic_flow_convergence}
        \Vert F^t(v) - m_{\zeta(v)}\Vert_{\lL^\infty} \leq C_{\ex} e^{-c t}\dist(v,\mM)\;.
    \end{equation}
    Moreover, we have the bound
    \begin{equation*}
        |\zeta(v)-\eta(v)|\lesssim \dist(v,\mM)
    \end{equation*}
    uniformly over all $v \in \vV_\beta$. 
\end{prop}
\begin{proof}
    We first consider the convergence \eqref{e:deterministic_flow_convergence}. A version of $\lL^2$ convergence for $v$ in an $\lL^2$ neighborhood of $\mM$ is proven in \cite[Theorem~7.1]{Fun95}. For the above $\lL^\infty$ version, the new ingredients are the exponential decay of $\|e^{-t\aA}-\pP\|_{\lL^\infty \rightarrow \lL^\infty}$ in Proposition~\ref{prop:Linfty_decay} and the properties of $\eta$ in Proposition~\ref{prop:fermi}. Based on these two ingredients, the proof follows in the same way as \cite[Theorem~5]{BBDMP98}, so we omit the details. 

    As for the difference $|\eta(v) - \zeta(v)|$, we first note that by \eqref{e:deterministic_flow_convergence}, we have
    \begin{equation*}
        \|v - m_{\zeta(v)}\| \leq C_{\ex} \beta\;.
    \end{equation*}
    Let $\beta$ be sufficiently small so that $C_{\ex} \beta$ satisfies the assumption for ``$\beta$" in Proposition~\ref{prop:fermi}. The conclusion then follows.
\end{proof}

Proposition~\ref{prop:Linfty_decay} gives a good description for the first-order kernel $p_{t}(y,z;m)$ for large $t$. For small $t$ we have the following lemma from \cite{Fun95}, which we omit the proof. 

\begin{lem} \label{lem:Gaussian_small_time}
(\cite[Lemma~9.3]{Fun95})
    For every $T>0$, there exists $C=C_T>0$ such that
    \begin{equ}
        0<p_{t}(y,z;v)\leq Cq_{t}(y-z)
    \end{equ}
    for all $t\in(0,T]$ and $v \in \lL^\infty$. Here $q_t(x) = \frac{1}{\sqrt{4\pi t}}e^{-\frac{x^2}{4t}}$ is the heat kernel. 
\end{lem}

Recall the weighted norm defined in \eqref{e:Llambda}. We have the following $\lL^p_{-\lambda}$ version for Proposition~\ref{prop:Linfty_decay}. It will play a key role in subsequent analysis in this section. 

\begin{prop}
\label{prop:spectral_decom_lambda}
    There exist $\lambda^*,c, C > 0$ such that for every $\lambda\in[0,\lambda^*]$, $t>0$ and $1 \leq p \leq 2 \leq q \leq +\infty$, we have 
    \begin{equation} \label{e:spectral_decom_lambda}
        \|e^{-t\aA} - \pP\|_{\lL^p_{-\lambda}\to\lL^q_{-\lambda}} \leq C t^{-\frac{1}{2p} + \frac{1}{2q}} \, e^{-ct}\;,
    \end{equation}
    and hence
    \begin{equ}
    \label{e:semigroup_bound}
        \|e^{-t\aA}\|_{\lL^p_{-\lambda}\to\lL^q_{-\lambda}} \leq C \big( 1+t^{-\frac{1}{2p} + \frac{1}{2q}} \big)\;.
    \end{equ}
\end{prop}
\begin{proof}
    By the exponential decay of $m'$ from Lemma~\ref{lem:statationary_exponential_decay}, there exists $\lambda^*>0$ such that $\pP$ is a bounded operator from $\lL^p_{-\lambda}$ to $\lL^q_{-\lambda}$, with the operator norm independent of $\lambda \in [0, \lambda^*]$. Hence it suffices to prove \eqref{e:spectral_decom_lambda}. For $t\leq 2$, by the Gaussian bound in Lemma~\ref{lem:Gaussian_small_time} and Young's convolution inequality, we get 
    \begin{equation} \label{e:heat_flow_smoothing}
        \|e^{-t\aA}\phi\|_{\lL^q_{-\lambda}} \leq C\|(q_t e^{\lambda|\cdot|})*(|\phi| e^{-\lambda|\cdot|})\|_{\lL^q} \leq C \, t^{-\frac{1}{2p}+\frac{1}{2q}}\| \, |\phi| \, e^{-\lambda|\cdot|} \, \|_{\lL^p}\;,
    \end{equation}
    We now turn to $t\geq 2$. Note that $e^{-t \aA} \pP = \pP$ and that $e^{-t\aA}$ commutes with $\pP$. Combining this with the semigroup property, we have
    \begin{equation*}
        e^{-t\aA} - \pP = e^{-t\aA}\pP^\perp = e^{-\aA} e^{-(t-2)\aA}\pP^\perp e^{-\aA}\;.
    \end{equation*}
    Now using \cite[Lemma~9.6(iii)]{Fun95} and \eqref{e:semigroup_bound} for $t=1$, we get 
    \begin{equation*}
        \Vert e^{-t\aA} - \pP\Vert_{\lL^p_{-\lambda}\to\lL^q_{-\lambda}} \leq \Vert e^{-\aA} \Vert_{\lL^2_{-\lambda}\to\lL^q_{-\lambda}}\Vert e^{-(t-2)\aA}\pP^\perp\Vert_{\lL^2_{-\lambda}\to\lL^2_{-\lambda}} \Vert e^{-\aA}\Vert_{\lL^p_{-\lambda}\to\lL^2_{-\lambda}}
        \lesssim e^{-ct}\;.
    \end{equation*}
    This completes the proof. 
\end{proof}

\subsection{Linearized PDE}
\label{sec:linear_PDE}

From now on, we fix $c^*>0$ such that Propositions~\ref{prop:Linfty_decay}, ~\ref{pr:LinftyEC} and~\ref{prop:spectral_decom_lambda} hold with $c= c^*$. We also fix $\lambda^*>0$ from Proposition~\ref{prop:spectral_decom_lambda}. The main goal of this subsection is to prove 
Proposition~\ref{prop:linearized_PDE}, which is at the heart of our analysis for $D^k \zeta$. We begin with the following lemma.

\begin{lem}\label{lem:A_equation}
    Fix $p \in [1,2]$, $q \in [2,+\infty]$ and $\lambda \in [0, \lambda^*)$. For every $g \in \lL_{-\lambda}^p$ such that $\pP g = 0$, there exists a unique solution $h\in \lL^q_{-\lambda}$ of the equation $\aA h = g$ such that $\pP h=0$, which we denote by $\aA^{-1} g$. Moreover, we have 
    \begin{equ}
        \Big\| \int_0^t e^{-s\aA}g \,\md s - \aA^{-1} g \, \Big\|_{\lL^q_{-\lambda}} \lesssim e^{-c^* t} \|g\|_{\lL^p_{-\lambda}} \;.
    \end{equ}
\end{lem}
\begin{proof}
    Since $\pP g = 0$, by Proposition~\ref{prop:spectral_decom_lambda}, we have
    \begin{equation*}
        \Vert e^{-t\aA}g\Vert_{\lL^q_{-\lambda}} = \Vert (e^{-t\aA}-\pP)g\Vert_{\lL^q_{-\lambda}} \lesssim t^{-\frac{1}{2p} + \frac{1}{2q}} \,e^{-c^*t} \Vert g\Vert_{\lL^p_{-\lambda}}\;,
    \end{equation*}
    so that $h := \int_0^{+\infty}e^{-s\aA}g \, \md s$ is well-defined in $\lL_{-\lambda}^q$, and satisfies $\pP h = 0$. Furthermore, we have
    \begin{equation*}
        \aA h = \int_0^{+\infty} \aA e^{-s\aA}g \,\md s = \int_0^{+\infty} -\d_s (e^{-s\aA}g) \,\md s=g\;.
    \end{equation*}
    It remains to show uniqueness. Suppose not, then there exists $\tilde{h} \in \lL_{-\lambda}^q$ with $\pP \tilde{h} = 0$ and $\aA \tilde{h} = 0$. Since $\lambda < \lambda^*$, we have $\tilde{h} \in \lL_{-\lambda^*}^2$ and $e^{-t \aA} \tilde{h} = \tilde{h}$ in $\lL_{-\lambda^*}^2$. On the other hand, by the assumption $\pP \tilde{h} = 0$ and Proposition~\ref{prop:spectral_decom_lambda}, we have
    \begin{equation*}
        \|e^{-t \aA} \tilde{h}\|_{\lL_{-\lambda^*}^2} = \|(e^{-t \aA} - \pP) \tilde{h}\|_{\lL_{-\lambda^*}^2} \lesssim e^{- c^* t}\;,
    \end{equation*}
    which implies that $\tilde{h}$ has to be $0$. This completes the proof of the lemma.   
\end{proof}

\begin{prop}
\label{prop:linearized_PDE}
    Fix $p \in [1,2]$ and $\lambda \in [0, \lambda^*)$. Suppose $(g_t)_{t > 0}$ is a family of functions such that for some $\sigma \in (0,1]$ and $M>0$, one has
    \begin{equation} \label{e:ass_gt_uniform}
        \|g_t\|_{\lL_{-\lambda}^p} \leq M (1 + t^{-1+\sigma})
    \end{equation}
    for all $t > 0$. Suppose furthermore there exist $g \in \lL_{-\lambda}^p$ with $\pP g = 0$ and $c_0 >0$ such that
    \begin{equation} \label{e:ass_gt_ec}
        \Vert g_t-g\Vert_{\lL^p_{-\lambda}} \leq Mt^{-1+\sigma} e^{-c_0 t}
    \end{equation}
    for all $t>0$. Moreover, let $(r_t)_{t\geq 0}$ be another family of functions such that for some $c_0', C_0 > 0$, one has
    \begin{equation} \label{e:ass_r_decay}
        \|r_t\|_{\lL^\infty}\leq C_0 e^{-c'_0 t}
    \end{equation}
    for all $t \geq 0$. Let $u$ be the solution to the PDE
    \begin{equ} \label{e:linearized_PDE}
        \d_t u = \Laplace u + \big(f'(m) + r_t\big)u + g_t\;, \qquad 
        u[0] = 0\;.
    \end{equ}
    Then there exists $C>0$ independent of $M$ and $t$ such that the followings hold for all $t>0$:
    \begin{enumerate}
        \item Bound for small and large times: for every $q \in [2, +\infty]$, we have
        \begin{equ}
        \label{e:linearized_equation_solution}
            \Vert u[t]\Vert_{\lL^q_{-\lambda}} \leq CM \big(1 + t^{(\frac{1}{2q}-\frac{1}{2p}+\sigma) \wedge 0} \big)\;.
        \end{equ}
        \item Convergence in long time: let
        \begin{equation*}
            \tilde{g} := \int_{0}^{+\infty} \pP g_s \, \md s + \aA^{-1} g + \int_{0}^{+\infty} \pP \big( r_s u[s] \big) \, \md s\;,
        \end{equation*}
        where $\aA^{-1}g$ is defined in Lemma~\ref{lem:A_equation}. Then $\tilde{g} \in \lL_{-\lambda}^q$ is well-defined with
        \begin{equation} \label{e:g_tilde_bound}
            \|\tilde{g}\|_{\lL_{-\lambda}^q} \leq C \, M\;.
        \end{equation}
        Furthermore, we have
        \begin{equ}
        \label{e:linearized_equation_convergence}
            \Vert u[t] - \tilde{g}\Vert_{\lL^q_{-\lambda}} \leq C \, M \big(1 + t^{(\frac{1}{2q}-\frac{1}{2p}+\sigma) \wedge 0} \big) \; e^{-c t}
        \end{equ}
        for some $c>0$ depending on $c_0$, $c_0'$ and $c^*$ only. 
    \end{enumerate}
\end{prop}

\begin{rmk}
    As mentioned in Section~\ref{sec:deterministic_overview}, we will use this proposition for
    \begin{equation*}
        u[t] = p_t^{(k)}(\vec{y}, \cdot\,;v)
    \end{equation*}
    for $k \geq 2$, $\vec{y} \in \RR^k$ and $v \in \vV_{\beta,0}$. It satisfies \eqref{e:linearized_PDE} for some $g_t$ depending on $\vec{y}$ and $v$, and for
    \begin{equation*}
        r_t = f' \big( F^t(v) \big) - f'(m)\;.
    \end{equation*}
    The quantity $M$ in the assumption of the proposition depends on $\vec{y}$ but will be uniform over $v \in \vV_{\beta,0}$. The constant $C$ is independent of $\vec{y}$ and $v$. 
    
    We only use the above proposition with $p=1$ and $q=+\infty$, but the proof for general $p,q$ does not cost more, so we choose to state the general version. 
\end{rmk}

\begin{proof}[Proof of Proposition~\ref{prop:linearized_PDE}]
    In what follows, the value of the constant $C>0$ may vary from line to line, but they are all independent of $M$ and $t>0$. 

    We first note that the assumptions \eqref{e:ass_gt_uniform} and \eqref{e:ass_gt_ec} imply $\|g\|_{\lL_{-\lambda}^p} \leq 2M$. For $t>0$, define
    \begin{equation} \label{e:defn_g_tilde}
        \tilde{g}_t := \int_{0}^{t/2} \pP g_s \, \md s + \int_{0}^{t/2} e^{-s \aA} g \md s + \int_{0}^{t} \pP \big( r_s u[s] \big) \, \md s\;.
    \end{equation}
    We will show that
    \begin{equation} \label{e:g_tilde_converge}
        \|\tilde{g}_t - \tilde{g}\|_{\lL_{-\lambda}^q} \leq C M e^{-ct}\;,
    \end{equation}
    and
    \begin{equation} \label{e:u_converge_g_tilde}
        \|u[t] - \tilde{g}_t\|_{\lL_{-\lambda}^q} \leq C \, M \big(1 + t^{(\frac{1}{2q}-\frac{1}{2p}+\sigma) \wedge 0} \big) \; e^{-c t}\;.
    \end{equation}
    Since \eqref{e:g_tilde_converge}, \eqref{e:u_converge_g_tilde} and \eqref{e:linearized_equation_solution} directly imply the well-posedness of $\tilde{g}$ in $\lL_{-\lambda}^q$ as well as the bounds \eqref{e:g_tilde_bound} and \eqref{e:linearized_equation_convergence}, it suffices to prove \eqref{e:g_tilde_converge}, \eqref{e:u_converge_g_tilde} and \eqref{e:linearized_equation_solution}. 

    The first two integrals defining $\tilde{g}_t$ in \eqref{e:defn_g_tilde} do not depend on the solution $u$. By the assumption \eqref{e:ass_gt_ec}, $\pP g = 0$ and Lemma~\ref{lem:A_equation}, we have
    \begin{equation} \label{e:g_tilde_converge_1_2}
        \int_{t/2}^{+\infty} \big( \|\pP g_s\|_{\lL_{-\lambda}^q} + \|e^{-s \aA} g\|_{\lL_{-\lambda}^q} \big) \, \md s \leq C M e^{-ct}
    \end{equation}
    for some $c>0$. This in particular implies the first two integrals on the right-hand side of \eqref{e:defn_g_tilde} are bounded by $C \, M$ in $\lL_{-\lambda}^q$ for all $t \geq 0$. 

    By Duhamel's principle for \eqref{e:linearized_PDE}, we have
    \begin{equation} \label{e:u_duhamel}
        u[t] = \int_0^t e^{-(t-s)\aA}g_s \,\md s + \int_0^t e^{-(t-s)\aA} \big(r_s u[s] \big)\,\md s\;.
    \end{equation}
    We want to compare it with $\tilde{g}_t$ defined in \eqref{e:defn_g_tilde}. For the first integral on the right-hand side above, we break it into integrations on $[0, \frac{t}{2}]$ and $[\frac{t}{2}, t]$, and compare them with the first two terms on the right-hand side of \eqref{e:defn_g_tilde} respectively. More precisely, by Proposition~\ref{prop:spectral_decom_lambda} and the assumption \eqref{e:ass_gt_uniform}, we have the bounds
    \begin{equation} \label{e:u_converge_11}
        \int_{0}^{t/2} \| ( e^{-(t-s) \aA} - \pP ) g_s \|_{\lL_{-\lambda}^q} \md s \leq C \, M \, t^{\frac{1}{2q} - \frac{1}{2p} + \sigma} e^{-ct}\;,
    \end{equation}
    and
    \begin{equation} \label{e:u_converge_12}
        \int_{t/2}^{t} \| e^{-(t-s) \aA} (g_s - g) \|_{\lL_{-\lambda}^q} \md s \leq C \, M \, t^{\frac{1}{2q}-\frac{1}{2p}+\sigma} e^{-ct}\;.  
    \end{equation}
    The bounds \eqref{e:u_converge_11} and \eqref{e:u_converge_12} together with \eqref{e:g_tilde_converge_1_2} in particular imply
    \begin{equation} \label{e:u_uniform_1}
        \int_{0}^{t} \| e^{-(t-s) \aA} g_s \|_{\lL_{-\lambda}^q} \md s \leq C \, M \big( 1 + t^{(\frac{1}{2q} - \frac{1}{2p} + \sigma) \wedge 0} \big)
    \end{equation}
    for all $t>0$. 

    We now turn to the second term on the right-hand side of the Duhamel formula \eqref{e:u_duhamel}. We first consider the special case $q=2$. By \eqref{e:u_uniform_1} and Proposition~\ref{prop:spectral_decom_lambda} applied to $\|e^{-(t-s) \aA}\|_{\lL_{-\lambda}^2 \rightarrow \lL_{-\lambda}^2}$, we get
    \begin{equation*}
        \Vert u[t] \Vert_{\lL^2_{-\lambda}} \leq C \, M \, \big(1 + t^{(\frac{1}{4}-\frac{1}{2p}+\sigma) \wedge 0} \big) + C \int_0^t \Vert r_s \Vert _{\lL^\infty}\Vert u[s]\Vert_{\lL^2_{-\lambda}} \,\md s\;.
    \end{equation*}
    Then, with the assumption \eqref{e:ass_r_decay} on $r_s$ and using Gr\"onwall's inequality, we get
    \begin{equation} \label{e:linearized_solution_gronwall}
        \| u[t] \|_{\lL_{-\lambda}^2} \leq C \, M \, \big(1 + t^{(\frac{1}{4}-\frac{1}{2p}+\sigma) \wedge 0} \big)\;.
    \end{equation}
    This bound, together with $\|\pP\|_{\lL_{-\lambda}^2 \rightarrow \lL_{-\lambda}^q} \lesssim 1$, immediately imply (for $q \in [2,+\infty]$)
    \begin{equation} \label{e:g_tilde_converge_3}
        \int_{t}^{+\infty} \left\| \pP \big( r_s u[s] \big) \right\|_{\lL_{-\lambda}^q} \, \md s \leq C \int_{t}^{+\infty} \|r_s\|_{\lL^\infty} \|u[s]\|_{\lL_{-\lambda}^2} \md s \leq C M e^{-ct}\;.
    \end{equation}
    Combining \eqref{e:g_tilde_converge_1_2}, \eqref{e:g_tilde_converge_3} and Lemma~\ref{lem:A_equation}, we deduce the exponential convergence \eqref{e:g_tilde_converge}.
    
    We now consider the $\lL_{-\lambda}^q$ norm of $u[t]$ for general $q \geq 2$. By \eqref{e:u_uniform_1}, we have
    \begin{equation*}
        \Vert u[t] \Vert_{\lL^q_{-\lambda}} \leq C \, M \big(1 + t^{(\frac{1}{2q}-\frac{1}{2p}+\sigma) \wedge 0} \big) + \int_0^t \Vert e^{-(t-s)\aA}\Vert_{\lL^2_{-\lambda}\to\lL^q_{-\lambda}} \Vert r_s u[s]\Vert_{\lL^2_{-\lambda}} \,\md s\;.
    \end{equation*}
    Applying Proposition~\ref{prop:spectral_decom_lambda}, \eqref{e:ass_r_decay} and \eqref{e:linearized_solution_gronwall} to the second integral on the right-hand side above, we then deduce the bound \eqref{e:linearized_equation_solution} for $\| u[t] \|_{\lL_{-\lambda}^q}$. 

    Finally, a similar argument (with Proposition~\ref{prop:spectral_decom_lambda} applied to $e^{-(t-s) \aA} - \pP$ this time) gives
    \begin{equation} \label{e:u_converge_2}
        \int_{0}^{t} \left\| \big( e^{-(t-s) \aA} - \pP \big) \big( r_s u[s] \big) \right\|_{\lL_{-\lambda}^q} \md s \leq C \, M e^{-ct}\;.
    \end{equation}
    We then deduce \eqref{e:u_converge_g_tilde} by combining \eqref{e:u_converge_11}, \eqref{e:u_converge_12}, \eqref{e:u_converge_2}. This completes the proof of the theorem. 
\end{proof}

\subsection{Kernel decompositions}\label{sec:kernel}

In this subsection, we prove a systematic decomposition as well as quantitative bounds for $p_t^{(k)}$. Before giving precise descriptions of the equations for $p_t^{(k)}$, we first introduce some notations for partitions and blocks. For a finite set $S$, we use $|S|$ to denote the cardinality of $S$, and $\emP(S)$ to denote the set of all partitions of $S$. Two typical sets $S$ we will encounter below are $S$ being a subset of $[n]$, or $S \in \emP([n])$ for some $n$. Elements in a partition $P \in \emP(S)$ are called blocks. 

For example, if $n=4$ and $S = \big\{ \{1,2\}, \{3\}, \{4\} \big\} \in \emP([4])$, then $|S|=3$ and $\emP(S)$ includes the five partitions of the three-element set $S$. The partition
\begin{equation*}
    P = \Big\{ \{ \{1,2\}, \{3\} \}, \; \{ \{4\} \} \Big\} \in \emP(S)
\end{equation*}
has cardinality $|P|=2$. It has two blocks with cardinalities $2$ and $1$ respectively. 

We say a partition $P'$ is \textit{finer} than $P$, denoted by $P' \leq P$, if for every block $B \in P$, there exist $k \geq 1$ and blocks $B_{1}', \dots B_{k}' \in P'$ such that 
\begin{equation*}
    B = \bigcup_{j=1}^{k} B_j'\;.
\end{equation*}
Define $p_t^{(k)}$ for $k \geq 1$ recursively as follows. For $k=1$, recall $p_t = p_t^{(1)}$ from \eqref{e:p_t}. For $k \geq 2$, $\vec{y} \in \RR^k$ and $v\in\lL^\infty$, let $p_t^{(k)}(\vec{y}, \cdot\,; v)$ be the solution $u^{(k)}$ to the equation
\begin{equation*} 
    \d_t u^{(k)} = \Delta u^{(k)} + f' \big( F^t(v) \big) u^{(k)} + g_t^{(k)}(\vec{y}; v)\;, \qquad u^{(k)}[0] = 0\;,
\end{equation*}
where
\begin{equation} \label{e:g_t-k_defn}
    g_t^{(k)}(\vec{y}, \cdot\,; v) = \sum_{P\in\emP([k]),|P|\neq 1}f^{(|P|)}\big(F^t(v)\big)\Big(\prod_{B\in P}p_t^{(|B|)}(\vec{y}_B,\cdot\,;v)\Big)\;.
\end{equation}
Since the sum is taken over partitions $P$ with $|P| \geq 2$, all blocks $B$ appearing in the product above satisfy $|B| \leq k-1$, and hence the right-hand side above involves $p_t^{(j)}$ for $j \leq k-1$ only. 

For large time, we study the operator $\Delta + f' \big( F^t(v) \big)$ as a perturbation of $\Delta + f'(m_{\zeta(v)})$. It is then natural to introduce
\begin{equation} \label{e:rtv}
    \r_t(v):= f'\big(F^t(v)\big) - f'(m_{\zeta(v)})\;.
\end{equation}
Since $f\in \cC_b^\infty$, by mean value theorem and the exponential convergence \eqref{e:deterministic_flow_convergence}, we deduce
\begin{equation} \label{e:r_decay}
    \Vert \r_t(v) \Vert_{\lL^\infty}\lesssim e^{-c^*t}
\end{equation}
uniformly over $v \in \vV_\beta$ and $t \geq 0$. For $v \in \vV_{\beta,0}$, $k \geq 2$ and $\vec{y} \in \RR^k$, the above equation for $u^{(k)}$ has the same form as \eqref{e:linearized_PDE} with $g_t = g_t^{(k)}(\vec{y}; v)$ and $r_t = \r_t(v)$, satisfying the decay assumption \eqref{e:ass_r_decay} uniformly over $v \in \vV_{\beta,0}$. 

We now start to describe our decomposition and bounds for $p_t^{(k)}$. The main results are Lemma~\ref{lem:decomposition1} for $p_t = p_t^{(1)}$ and Proposition~\ref{pr:decomposition} for general $k \geq 2$. We prove them by induction, starting with $k=1$. 

\begin{lem}\label{lem:bound1}
    Recall $\lambda^* > 0$ from Proposition~\ref{prop:spectral_decom_lambda}. There exists $\overline{c}>0$ such that 
    \begin{equation*}
        \Vert p_t(y,\cdot\,;v)\Vert_{\lL^q_{-\lambda}} \lesssim (e^{\overline{c} \lambda t} + t^{-\frac{1}{2}+\frac{1}{2q}})e^{-\lambda |y|}
    \end{equation*}
    for all $t> 0$, $\lambda \in [0,\lambda^*]$ and $q \in [2,+\infty]$.
\end{lem}
\begin{proof}
    For $t\leq 2$, the assertion follows directly from the Gaussian bound in Lemma~\ref{lem:Gaussian_small_time}. For $t\geq 2$, we use the semigroup property of $p$ to write
    \begin{equation*}
        p_{t}(y,z;v) = \int_\RR p_1 \big( z_2, z; F^{t-1}(v) \big) \bigg( \int_\RR p_1 \big( y, z_1; v \big) \, p_{t-2} \big( z_1, z_2; F^1(v) \big) \, \md z_1 \bigg) \, \md z_2\;.
    \end{equation*}
    By Lemma~\ref{lem:Gaussian_small_time} and the estimate \eqref{e:heat_flow_smoothing}, we have
    \begin{equation*}
        \| p_1\big(z_2, z; F^{t-1}(v)\big) \|_{\lL_{-\lambda}^2(z_2) \rightarrow \lL_{-\lambda}^q(z)} \lesssim 1\;, \quad \|p_1(y,z_1;v)\|_{\lL^2_{-\lambda}(z_1)} \lesssim e^{-\lambda |y|}\;.
    \end{equation*}
    Also, by \cite[Lemma~9.6(ii)]{Fun95}, we have
    \begin{equation*}
        \| p_{t-2}\big(z_1, z_2; F^{1}(v)\big) \|_{\lL_{-\lambda}^2(z_1) \rightarrow \lL_{-\lambda}^2(z_2)} \lesssim e^{\overline{c} \lambda t}\;.
    \end{equation*}
    The bound in \cite[Lemma~9.6(ii)]{Fun95} was proved for $v$ in $\lL^2$ neighborhood of $\mM$ instead of $\lL^\infty$ neighborhood, but the proof works for the $\lL^\infty$ setting in our situation in the same way. Hence, combining the above gives
    \begin{equation*}
        \|p_t(y, z; v)\|_{\lL_{-\lambda}^q(z)} \lesssim e^{\overline{c} \lambda t} e^{-\lambda |y|}
    \end{equation*}
    for $t \geq 2$. This completes the proof. 
\end{proof}

We use Lemma~\ref{lem:bound1} as an intermediate step to prove the following strictly stronger version.

\begin{lem}\label{lem:decomposition1}
Recall $\lambda^*>0$ from Proposition~\ref{prop:spectral_decom_lambda} and $c^*>0$ from the beginning of Section~\ref{sec:linear_PDE}. There exist $c_1\in(0,c^*]$ and $\overline{\lambda}\in(0,\lambda^*)$ such that for every $v\in \vV_{\beta,0}$, the kernel $p_t(y,z;v)$ can be decomposed as
    \begin{equation}\label{e:decomposition1}
        p_t(y,z;v) = \Lambda^{(1)}_v(y) \, m'(z) + \hat{p}_{t}(y,z;v)\;,
    \end{equation}
    where $\Lambda_v^{(1)}$ satisfies the bound
    \begin{equation}\label{e:decomposition1-1}
        \vert \Lambda^{(1)}_v(y)\vert \lesssim e^{- \frac{1}{2} \overline{\lambda} |y|}\;,
    \end{equation}
    and for every $\lambda \in (0, \overline{\lambda}]$, $\hat{p}_{t}$ satisfies the bounds
    \begin{equation} \label{e:decomposition1-2}
        \Vert \hat{p}_{t}(y,\cdot\,;v)\Vert_{\lL^\infty_{-\lambda}} \lesssim_\lambda t^{-1/2}e^{-c_1 t}e^{-\frac{1}{2}\lambda |y|}\;,
    \end{equation}
    \begin{equation} \label{e:decomposition1-3}
        \Vert \hat{p}_{t}(y,\cdot\,;v)\Vert_{\lL^2_{-\lambda}} \lesssim_\lambda t^{-1/4}e^{-c_1 t}e^{-\frac{1}{2}\lambda |y|}
    \end{equation}
    uniformly over all $t>0$ and $v\in\vV_{\beta,0}$. As a consequence, for every $\lambda\in(0,\overline{\lambda}]$, we have
    \begin{equation}\label{e:bound1-2}
        \Vert p_t(y,\cdot\,;v)\Vert_{\lL^2_{-\lambda}}\lesssim (1+t^{-1/4})e^{-\frac{1}{2}\lambda |y|}
    \end{equation}
    uniformly over all $t>0$ and $v\in\vV_{\beta,0}$. Moreover, $\Lambda_v^{(1)}$ has the explicit expression
    \begin{equation}\label{e:lambda_v-1_expression}
        \Lambda^{(1)}_v(y)=\frac{1}{\|m'\|_{\lL^2}^2}\Big(m'(y)+\int_0^{+\infty}\int_{\RR} m'(z)\r_t(z;v)p_t(y,z;v) \,\md z \,\md t\Big)\;,
    \end{equation}
    where we recall from \eqref{e:rtv} that $\r_t(v) = f' \big( F^t(v) \big) - f'(m)$ for $v \in \vV_{\beta,0}$. 
\end{lem}
\begin{proof}
    We first deal with the case $v=m$. Define
    \begin{equation*}
        \Lambda_m^{(1)}(y) := \frac{1}{\|m'\|_{\lL^2}^2} \bracket{m', p_t(y,\cdot \,; m)} = \frac{m'(y)}{\|m'\|_{\lL^2}^2}\;,
    \end{equation*}
    which is independent of $t$. Here we use the fact that $m'$ satisfies \eqref{e:m_prime_equation}, so it is invariant under the flow $e^{-t\aA}$. By Lemma~\ref{lem:statationary_exponential_decay}, we have
    \begin{equation*}
        |\Lambda_m^{(1)}(y)| \lesssim e^{-\overline{\lambda} |y|}\;.
    \end{equation*}
    We now turn to the bounds \eqref{e:decomposition1-2} and \eqref{e:decomposition1-3} for
    \begin{equ}
        \hat{p}_{t}(y,\cdot\,;m) = p_t(y,\cdot \,;m) - \Lambda_m^{(1)}(y) m'\;,
    \end{equ}
    where we can prove better bounds at $v=m$ than \eqref{e:decomposition1-2} and \eqref{e:decomposition1-3} for general $v \in \vV_{\beta,0}$. More precisely, we will show that
    \begin{equation} \label{e:bound-1-m-remainder}
        \begin{split}
        \|\hat{p}_{t}(y,\cdot\,; m)\|_{\lL_{-\lambda}^\infty} &\lesssim \big( 1 + t^{-1/2} \big) e^{- c^* t} e^{-\lambda |y|}\;,\\
        \|\hat{p}_{t}(y,\cdot\,; m)\|_{\lL_{-\lambda}^2} &\lesssim \big( 1 + t^{-1/4} \big) e^{- c^* t} e^{-\lambda |y|}\;.
        \end{split}
    \end{equation}
    For $t \leq 1$, the bounds in \eqref{e:bound-1-m-remainder} follow from Lemma~\ref{lem:bound1} and the exponential decay of $m'$. For $t \geq 1$, by definition of $\Lambda_m^{(1)}$, we can write
    \begin{equation*}
        \hat{p}_{t}(y,\cdot\,;m) = \pP^{\perp} e^{-(t-1) \aA} p_{1}(y,\cdot\,;m) = (e^{-(t-1) \aA} - \pP)p_{1}(y,\cdot\,;m)\;,
    \end{equation*}
    and the desired bounds follow from Proposition~\ref{prop:spectral_decom_lambda} and bounds for $p_{1}$.
    
    For general $v \in \vV_{\beta,0}$, by Duhamel's principle for \eqref{e:p_t}, we have
    \begin{equation*}
        p_t(y,z;v) = p_t(y,z;m) + \int_0^t\int_{\RR} p_{t-s}(z',z;m) \r_s(z';v) p_s(y,z';v) \,\md z' \,\md s\;,
    \end{equation*}
    where we used the fact that $p$ is translation invariant in time at $v=m$. Then using the decomposition for $p_t(y,z;m)$ established above, we get the decomposition \eqref{e:decomposition1} for $p_t(y,z;v)$ with
    \begin{equation*}
        \Lambda^{(1)}_v(y):=\Lambda^{(1)}_m(y)+\int_0^{+\infty}\int_{\RR} \Lambda^{(1)}_m(z')\r_s(z';v)p_s(y,z';v) \,\md z' \,\md s\;,
    \end{equation*}
    and
    \begin{equation} \label{e:p_t_hat}
    \begin{aligned}
        \hat{p}_{t}(y,z;v):=&\hat{p}_{t}(y,z;m)+ \int_0^t\int_{\RR} \hat{p}_{t-s}(z',z;m)\r_s(z';v)p_s(y,z';v) \,\md z' \,\md s \\&-\Big(\int_t^\infty\int_{\RR} \Lambda^{(1)}_m(z')\r_s(z';v)p_s(y,z';v) \,\md z' \,\md s\Big)m'(z)\;.
    \end{aligned}
    \end{equation}
    We now prove the bound \eqref{e:decomposition1-1} for $\Lambda_v^{(1)}$ and \eqref{e:decomposition1-2} and \eqref{e:decomposition1-3} for $\hat{p}_t(v)$. The two terms with $v=m$ on the right-hand sides above have already been treated, so it remains to control the other three terms with space-time integrations. 

    For every $\lambda > 0$, by H\"older inequality and Lemma~\ref{lem:bound1}, we have
    \begin{equation} \label{e:bound-1-ps}
        \|p_s(y,\cdot\,;v)\|_{\lL^{1}_{-\lambda}} \lesssim_{\lambda} \|p_s(y,\cdot\,;v)\|_{\lL^{2}_{-\lambda/2}} \lesssim \big(e^{\frac{\overline{c} \lambda s}{2}} + s^{-\frac{1}{4}} \big) \, e^{-\frac{\lambda}{2} |y|} \;.
    \end{equation}
    This enables us to control the $\lL_{-\lambda}^2$ and $\lL_{-\lambda}^\infty$ norms of the three space-time integrals above. For example, using Minkowski inequality, we can control the $\lL_{-\lambda}^q(z)$ norm (for $q=2$ or $+\infty$) of the second term on the right-hand side of \eqref{e:p_t_hat} by
    \begin{equation*}
        \begin{split}
        &\phantom{111}\int_{0}^{t} \int_\RR |p_s(y,z';v)| \cdot |\r_s(z';v)| \cdot \|\hat{p}_{t-s}(z', \cdot\,; m)\|_{\lL_{-\lambda}^q} \, \md z' \, \md s\\
        &\lesssim \int_{0}^{t} \big( 1 + (t-s)^{-\frac{1}{2}} \big) e^{-c^* (t-s)} e^{-c^* s} \bigg( \int_\RR e^{-\lambda |z'|} |p_s (y, z'; m)| \, \md z' \bigg) \, \md s\\
        &\lesssim e^{-c^* t} e^{-\frac{\lambda}{2} |y|} \int_{0}^{t} \big( 1 + (t-s)^{-\frac{1}{2}} \big) \cdot \big(e^{\frac{\overline{c} \lambda s}{2}} + s^{-\frac{1}{4}} \big) \, \md s\;,
        \end{split}
    \end{equation*}
    where we used the decay of $\|\r_s\|_{\lL^\infty}$ and the bound \eqref{e:bound-1-m-remainder} in the first inequality above, and \eqref{e:bound-1-ps} in the second inequality. Choose $\overline{\lambda}$ such that $\overline{c} \overline{\lambda} < c^*$. Then the above quantity can be controlled by $e^{-c_1 t} e^{- \frac{\lambda}{2} |y|}$ for some $c_1>0$. 
    
    Bounds for the other terms follow from similar but simpler calculations. This completes the proof. 
\end{proof}

\begin{rmk} \label{rmk:translation}
    The above lemma extends to $v \in \vV_\beta$ as follows. For $v \in \vV_\beta$, let
    \begin{equation*}
        \widetilde{v} = \emS_{-\zeta(v)} v \in \vV_{\beta,0}\;.
    \end{equation*}
    Since $\emS_\theta F^t(v) = F^t(\emS_\theta v)$, we have
    \begin{equation*}
        p_t(y,z;v) = p_t \big(y- \zeta(v), z - \zeta(v); \, \widetilde{v} \big)\;.
    \end{equation*}
    Then all the statements in Lemma~\ref{lem:decomposition1} holds with $m'$ replaced by $m_{\zeta(v)}'$, $\lL_{-\lambda}^q$-norm replaced by $\lL_{-\lambda, \zeta(v)}^q$-norm given by $\Vert \phi\Vert_{\lL^q_{-\lambda, \zeta(v)}}:=\Vert \emS_{-\zeta(v)}\phi\Vert_{\lL^q_{-\lambda}}$, and the exponential decay in $|y|$ replaced by $|y-\zeta(v)|$. This also applies to the properties of $p_t^{(k)}(\cdot\,, \cdot\,; v)$ for $k \geq 2$ in Lemmas~\ref{lem:decomposition2} and~\ref{pr:decomposition} below. 
\end{rmk}

With the decomposition for $p_t$ at hand, together with Proposition~\ref{prop:linearized_PDE}, we can prove the decompositions for $p_t^{(k)}$ inductively. Since the case for $k=2$ is still relatively simple, to illustrate the general mechanism in the main proof, we choose to present it here separately from the general induction. 

\begin{lem}\label{lem:decomposition2}
   Let $\overline{\lambda}$ be same as that in Lemma~\ref{lem:decomposition1}. For every $\lambda\in(0,\overline{\lambda}]$, we have
    \begin{equation}\label{e:bound2}
        \Vert p_t^{(2)}(y_1,y_2,\cdot\,;v)\Vert_{\lL^\infty_{-\lambda}} \lesssim e^{-\frac{1}{4}\lambda|y_1|-\frac{1}{4}\lambda|y_2|}
    \end{equation}
    uniformly over all $t>0$ and  $v\in\vV_{\beta,0}$. Moreover, for every $v\in\vV_{\beta,0}$, $p_t^{(2)}(y_1,y_2,z;v)$ can be decomposed as 
    \begin{equation}\label{e:decomposition2}
        p_t^{(2)}(y_1,y_2,z;v) = \Lambda^{(2)}_v(y_1,y_2)m'(z) + \Lambda^{(1)}_v(y_1)\Lambda^{(1)}_v(y_2)m''(z)+\hat{p}_{t}^{(2)}(y_1,y_2,z;v)\;,
    \end{equation}
    where $\Lambda^{(1)}_v$ is defined in Lemma~\ref{lem:decomposition1}, and $\Lambda_v^{(2)}$ is symmetric in its two variables. Furthermore, there exists $c_2> 0$ such that for every $\lambda\in(0,\overline{\lambda}]$, $\Lambda_v^{(2)}$ and $\hat{p}_{t}^{(2)}$ satisfy the bounds
    \begin{equation}\label{e:decomposition2-1}
        \vert \Lambda^{(2)}_v(y_1,y_2)\vert \lesssim e^{-\frac{1}{4}\overline{\lambda} |y_1|-\frac{1}{4}\overline{\lambda}|y_2|}\;,
    \end{equation}
    \begin{equation}\label{e:decomposition2-2}
        \Vert \hat{p}_{t}^{(2)}(y_1,y_2,\cdot\,;v)\Vert_{\lL^\infty_{-\lambda}}\lesssim e^{-c_2t}e^{-\frac{1}{4}\lambda |y_1|-\frac{1}{4}\lambda |y_2|}
    \end{equation}
    uniformly over all $t>0$ and $v\in\vV_{\beta,0}$. Finally, $\Lambda_v^{(2)}$ has the explicit expression
    \begin{equation} \label{e:lambda_v-2_expression}
        \begin{split}
        \Lambda_v^{(2)}(y_1,y_2) = \frac{1}{\|m'\|_{\lL^2}^2} \int_{0}^{+\infty} \int_\RR\Big( & f'' \big(F^t(v) \big)(z) \prod_{j=1}^{2} p_t(y_j, z; v) \\
        &+  \r_t(z;v) \, p_t^{(2)}(y_1, y_2, z; v)\Big)m'(z)\,\md z\,\md t\;.
        \end{split}
    \end{equation}
\end{lem}
\begin{proof}
    The kernel $p_t^{(2)}(y_1, y_2, \cdot\,; v)$ satisfies \eqref{e:linearized_PDE} with $r_t = \r_t(v)$ and
    \begin{equation*}
        g_t = g_t^{(2)}(y_1, y_2, \cdot \,; \, v) = f''\big(F^t(v)\big) \, p_t(y_1,\cdot\,;v) \, p_t(y_2,\cdot\,;v)\;.
    \end{equation*}
    We first check that this $g_t$ satisfies the assumptions of Proposition~\ref{prop:linearized_PDE}. By H\"older inequality, boundedness of $f''$ and \eqref{e:bound1-2}, we have
    \begin{equation*}
        \Vert g_t\Vert_{\lL^1_{-\lambda}}\lesssim \Vert p_t(y_1,\cdot\,;v)\Vert_{\lL^2_{-\lambda/2}}\Vert p_t(y_2,\cdot\,;v)\Vert_{\lL^2_{-\lambda/2}} \lesssim \big( 1+t^{-\frac{1}{2}} \big) \, e^{-\frac{1}{4}\lambda (|y_1| + |y_2|)}\;.
    \end{equation*}
    The candidate function for the limit of $g_t = g_t^{(2)}$ as $t \rightarrow +\infty$ is
    \begin{equation*}
        g = g^{(2)}(y_1, y_2, \cdot \,; \, v) = \Lambda_v^{(1)}(y_1) \, \Lambda_v^{(1)}(y_2) \, f''(m) \, \big( m' \big)^2\;.
    \end{equation*}
    Since $f''(m)$ and $m'$ are both odd, $\pP g = 0$ by symmetry. For the difference $g_t - g$, with the decomposition \eqref{e:decomposition1} for $p_t$, we have
    \begin{equs}
        g_t - g =& f''\big(F^t(v)\big) \left( \hat{p}_t(y_1, \cdot\,; v) \, \Lambda_v^{(1)}(y_2) \, m' + p_t(y_1, \cdot\,; v) \, \hat{p}_t(y_2, \cdot\,; v) \right)\\
        &+ \Lambda^{(1)}_v(y_1) \, \Lambda^{(1)}_v(y_2) \, \left(f''\big(F^t(v)\big)-f''(m)\right) \big(m'\big)^2\;.
    \end{equs}
    Then, using $f \in \cC_b^\infty$ and H\"older inequality, we have
    \begin{equation*}
        \begin{split}
        \|g_t - g\|_{\lL_{-\lambda}^1} \lesssim \, & \|\hat{p}_t(y_1, \cdot)\|_{\lL_{-\lambda}^2} |\Lambda_v^{(1)}(y_2)| \,\|m'\|_{\lL^2} + \|p_t (y_1, \cdot)\|_{\lL_{-\lambda/2}^2} \|\hat{p}_t(y_2, \cdot)\|_{\lL_{-\lambda/2}^2}\\
        &+ |\Lambda_v^{(1)}(y_1)| \, |\Lambda_v^{(1)}(y_2)| \, \|f''\big( F^t(v) \big) - f''(m)\|_{\lL^\infty} \|\big( m' \big)^2\|_{\lL_{-\lambda}^1}\;,
        \end{split}
    \end{equation*}
    where we omitted $v$ in $p_t$ or $\hat{p}_t$ for notational simplicity. Applying Lemma~\ref{lem:statationary_exponential_decay}, Proposition~\ref{pr:LinftyEC} and Lemma~\ref{lem:decomposition1} to various terms on the right-hand side above, we get
    \begin{equation*}
        \|g_t - g\|_{\lL_{-\lambda}^1} \lesssim t^{-\frac{1}{2}} \, e^{-c_1 t} e^{-\frac{\lambda}{4} (|y_1| + |y_2| )}\;.
    \end{equation*}
    Hence, the family $(g_t)_{t \geq 0}$ satisfies the assumptions of Proposition~\ref{prop:linearized_PDE} with $\sigma = \frac{1}{2}$ and
    \begin{equation*}
        M = M(y_1, y_2) = C_0 \, e^{-\frac{\lambda}{4} (|y_1| + |y_2|)}
    \end{equation*}
    for some constant $C_0$ independent of $y_1, y_2$ and $v$. The bound \eqref{e:bound2} follows directly from Assertion 1 in Proposition~\ref{prop:linearized_PDE} with $p=1$ and $q=+\infty$. 
    
    Now we turn to the second part of the lemma, namely the decomposition of $p_t$ and their bounds. Define
    \begin{equation} \label{e:g_tilde-2}
        \tilde{g} = \tilde{g}^{(2)} = \Lambda^{(1)}_v(y_1) \, \Lambda^{(1)}_v(y_2) \, m'' + \int_0^{+\infty} \pP \big( \, g_s + \r_s(v) \, p_s^{(2)}(y_1,y_2,\cdot\,;v) \big) \,\md s\;.
    \end{equation}
    Since $\pP m'' = 0$ and $\Lambda_v^{(1)}(y_1) \Lambda_v^{(2)}(y_2) \aA m'' = g$, this $\tilde{g}$ does correspond to the one in Proposition~\ref{prop:linearized_PDE}. Hence, by Assertion 2 in that proposition, we have
    \begin{equation} \label{e:second_order_g_tilde}
        \Vert \tilde{g}\Vert_{\lL^\infty_{-\lambda}} \lesssim e^{-\frac{\lambda}{4} (|y_1|+|y_2|)}\;, \quad
        \Vert p_t^{(2)}(y_1,y_2,\cdot\,;v) - \tilde{g}\Vert_{\lL^\infty_{-\lambda}} \lesssim e^{-c_2 t} \, e^{-\frac{\lambda}{4} (|y_1| + |y_2|)}
    \end{equation}
    for some $c_2>0$. Define
    \begin{equation*}
        \hat{p}_t^{(2)}(y_1, y_2, z; v) := p_t^{(2)}(y_1, y_2, z; v) - \tilde{g}(z)\;.
    \end{equation*}
    Then \eqref{e:second_order_g_tilde} implies that the bound \eqref{e:decomposition2-2} holds for $\hat{p}_t^{(2)}$. It remains to check $\tilde{g}$ has the same form as the sum of the first two terms on the right-hand side of \eqref{e:decomposition2} and to get corresponding bounds for them. 

    Comparing \eqref{e:g_tilde-2} and the first two terms on the right-hand side of \eqref{e:decomposition2}, since $\pP$ ``projects" a function onto the direction spanned by $m'$, we see the decomposition \eqref{e:decomposition2} indeed holds with $\Lambda_v^{(2)}$ being such that
    \begin{equation*}
        \Lambda_v^{(2)}(y_1, y_2) \, m' = \int_0^{+\infty} \pP \big( \, g_s + \r_s(v) \, p_s^{(2)}(y_1,y_2,\cdot\,;v) \big) \,\md s\;.
    \end{equation*}
    Testing both sides above with $m'$ gives the expression \eqref{e:lambda_v-2_expression}. The symmetry of $\Lambda_v^{(2)}$ in its two variables follows from that of $p_s^{(2)}$ and $g_s$. Also, since
    \begin{equation*}
        \Lambda_v^{(2)}(y_1, y_2) \, m' = \tilde{g} - \Lambda_v^{(1)}(y_1) \, \Lambda_v^{(1)}(y_2) \, m''\;.
    \end{equation*}
    testing both sides above with $m'$ and using the bound \eqref{e:second_order_g_tilde} and that $\bracket{m'', m'} = 0$, we deduce the bound \eqref{e:decomposition2-1}. This completes the proof. 
\end{proof}

Before proceeding to $p_t^{(k)}$ for $k \geq 3$, we first note from \eqref{e:g_t-k_defn} that the forcing term $g_t^{(k)}$ involves $k$ products of $p_t$. A naive bound via \eqref{e:decomposition1-2} and \eqref{e:decomposition1-3} will then give a non-integrable singularity at $t=0$ as soon as $k \geq 3$. The next lemma says one can reduce the singularity in time by trading it with some singularity in $\vec{y}$.

\begin{lem}\label{lem:technical}
Let $\lambda \geq 0$. For every $\sigma \in (0, \frac{1}{2})$ and $k \geq 3$, we have
\begin{equation} \label{e:technical1}
    \int_\RR e^{-\lambda|z|} \prod_{i=1}^k |p_t(y_i,z;v)| \,\md z \lesssim t^{-\frac{1}{2}-\sigma}e^{-\frac{\lambda}{k} \sum|y_i|}\Big(\sum\limits_{i,j\in[k]}(y_i-y_j)^2\Big)^{-\frac{k-2}{2}+\sigma}
\end{equation}
for all $t \in (0,1)$. 
\end{lem}
\begin{proof}
Denote $\overline{y} = \frac{\sum y_i}{k}$. Since $t \in (0,1)$, the Gaussian bound in Lemma~\ref{lem:Gaussian_small_time} gives
\begin{equation*}
    \prod_{i=1}^{k} |p_t (y_i, z; v)| \lesssim t^{-\frac{k}{2}} e^{-\frac{1}{4t} \sum (z-y_i)^2} = t^{-\frac{k}{2}} e^{-\frac{\sum\limits_{1\leq i<j\leq k}(y_i-y_j)^2}{4kt}} e^{-\frac{k}{4t} (z - \overline{y})^2}\;.
\end{equation*}
Multiplying both sides by $e^{-\lambda |z|}$ and integrating over $z \in \RR$ with a change of variable $z \mapsto z + \overline{y}$, we get
\begin{equation*}
    \int_\RR e^{-\lambda|z|} \prod_{i=1}^k |p_t(y_i,z;v)| \,\md z \lesssim t^{-\frac{k-1}{2}} \, e^{-\frac{\sum\limits_{1\leq i<j\leq k}(y_i-y_j)^2}{4kt}} \, e^{-\lambda|\overline{y}|}
\end{equation*}
for all $t\in(0,1)$. We claim that
\begin{equation*}
    |\overline{y}|\geq \frac{1}{k}\sum|y_i|-\sqrt{\sum\limits_{1\leq i<j\leq k}(y_i-y_j)^2}\geq\frac{1}{k}\sum|y_i|-\frac{\sum\limits_{1\leq i<j\leq k }(y_i-y_j)^2}{8\lambda kt} - 2\lambda kt\;.
\end{equation*}
The second inequality is trivial. The first one follows from the fact that
\begin{equ}
    |y_i|-|y|\leq |y_i - \Bar{y}| \leq \sup_{1\leq j\leq k}|y_i-y_j| \leq \sqrt{\sum_{1\leq \ell<j\leq k}|y_\ell-y_j|^2}\;.
\end{equ}
With this claim, we can get
\begin{equation*}
    \int_\RR e^{-\lambda|z|} \prod_{i=1}^k |p_t(y_i,z;v)| \,\md z \lesssim t^{-\frac{k-1}{2}}e^{2\lambda^2 kt}e^{-\frac{\sum\limits_{1\leq i<j\leq k}(y_i-y_j)^2}{8  k t}}e^{-\frac{\lambda}{k}\sum|y_i|}\;.
\end{equation*}
Using $t \leq 1$ and $e^{-x} \lesssim x^{-\frac{k-2}{2}+\sigma}$ for $x > 0$, we get the desired bound \eqref{e:technical1}. 
\end{proof}

For $\sigma>0$, $k \geq 1$ and $\vec{y} \in \RR^k$, define $\sS_\sigma^{(1)}(y):=1$, $\sS_\sigma^{(2)}(y_1,y_2):=1$, and
\begin{equation}\label{e:Ssigmak}
    \sS^{(k)}_\sigma(\vec{y}):=\Big(\sum\limits_{i,j\in[k]}(y_i-y_j)^2\Big)^{-\frac{k-2}{2}+\sigma}\;, \qquad k \geq 3\;.
\end{equation}
Recall that for a set $S$, we use $\emP(S)$ to denote the set of all partitions of $S$. The following lemma is frequently used. 

\begin{lem} \label{lem:prod_sum_prod_S}
    For every $n\in \NN$ and every partition $P \in \emP([n])$, we have
    \begin{equation*}
        \prod_{B \in P} \Big( \sum_{P' \in \emP(B)} \prod_{B' \in P'} \sS_\sigma^{(|B'|)}(\vec{y}_{B'}) \Big) \lesssim \sum_{P' \in \emP([n])} \prod_{B' \in P'} \sS_\sigma^{(|B'|)}(\vec{y}_{B'})\;.
    \end{equation*}
\end{lem}
\begin{proof}
    The left-hand side is a sum of products of $\sS_\sigma$'s. Each product in the sum is precisely $\prod_{B' \in P'} \sS_\sigma^{(|B'|)}(\vec{y}_{B'})$ for some $P' \in \emP([n])$. 
\end{proof}

We give one more lemma before stating the general decomposition. 

\begin{lem} \label{lem:prod_sum_prod}
    Let $\{G^{(k)}\}_k$ be a sequence of functions that are symmetric in $k$ variables, and $\{\Theta(k)\}_k$, $\{\Gamma(k)\}_k$ be two sequences of real numbers. For every $n\in\NN$, we have the identity
    \begin{equation*}
        \begin{split}
        &\sum_{P\in \emP([n])} \Gamma(|P|) \prod_{B \in P} \Big( \sum_{Q_B \in \emP(B)} \Theta \big(|Q_B| \big) \prod_{J \in Q_B} G^{(|J|)}(\vec{y}_J) \Big)\\
        =& \sum_{P^*\in \emP([n])} \Big( \prod_{B^* \in P^*} G^{(|B^*|)}(\vec{y}_{B^*}) \Big) \sum_{P\in\emP( P^*)} \Gamma(|P|) \Big( \prod_{P' \in P} \Theta \big( |P'| \big) \Big)\;.
        \end{split}
    \end{equation*}
\end{lem}
\begin{proof}
    For any two partitions $P^* \leq P$ (recall that it means $P^*$ is finer than $P$) and any block $B \in P$, write $P^*(B)$ as the partition $P^*$ restricted to $B$. Note that
    \begin{equation*}
        P^* = \bigcup_{B \in P} P^*(B)
    \end{equation*}
    if $P^* \leq P$. We then have the identity
    \begin{equation} \label{e:prod_sum_prod}
        \begin{split}
        &\prod_{B \in P} \Big( \sum_{Q_B \in \emP(B)} \Theta \big(|Q_B| \big) \prod_{J \in Q_B} G^{(|J|)}(\vec{y}_J) \Big)\\
        = &\sum_{P^* \leq P} \bigg[ \Big( \prod_{B \in P} \Theta \big( |P^*(B)| \big) \Big) \cdot \Big( \prod_{B^* \in P^*} G^{(|B^*|)}(\vec{y}_{B^*}) \Big) \bigg]
        \end{split}
    \end{equation}
    for every $P \in \emP([n])$, where the sum is taken over all partitions $P^* \in \emP([n])$ which are finer than $P$ (including $P$). 

    Multiplying $\Gamma(|P|)$ on both sides of \eqref{e:prod_sum_prod}, summing over $P \in \emP([n])$ and applying Fubini to change the order of summation, we get
    \begin{equation*} 
        \begin{split}
        &\sum_{P\in \emP([n])} \Gamma(|P|) \prod_{B \in P} \Big( \sum_{Q_B \in \emP(B)} \Theta \big(|Q_B| \big) \prod_{J \in Q_B} G^{(|J|)}(\vec{y}_J) \Big)\\
        =&\sum_{P^*\in \emP([n])} \Big( \prod_{B^* \in P^*} G^{(|B^*|)}(\vec{y}_{B^*}) \Big) \sum_{P\geq P^*} \Gamma(|P|) \Big( \prod_{B \in P} \Theta \big( |P^*(B)| \big) \Big)\;.
        \end{split}
    \end{equation*}
    Now note that any $P\geq P^*$ can be interpreted as a partition in $\emP(P^*)$, so we have
    \begin{equ}
        \sum_{P\geq P^*} \Gamma(|P|) \Big( \prod_{B \in P} \Theta \big( |P^*(B)| \big) \Big) = \sum_{P\in\emP( P^*)} \Gamma(|P|) \Big( \prod_{P' \in P} \Theta \big( |P'| \big) \Big)\;.
    \end{equ}
    The proof is then complete. 
\end{proof}

\begin{prop}\label{pr:decomposition}
    Fix $0<\sigma \ll 1$, and let $K_\sigma = |\log_2 \sigma|-10$. Let $\overline{\lambda}$ be same as that in Lemma~\ref{lem:decomposition1}. Then for every $2 \leq k\leq K_\sigma$, there exists $c_k>0$ such that for every $\lambda \in(0,\overline{\lambda}]$, we have
    \begin{equation}\label{e:bound}
        \Vert p_t^{(k)}(\vec{y},\cdot\,;v)\Vert_{\lL^\infty_{-\lambda}} \lesssim (1+t^{-2^k\sigma})e^{-c_k\lambda |\vec{y}|}\sum_{P\in\emP([k])}\Big(\prod_{B\in P}\sS^{(|B|)}_\sigma(\vec{y}_B)\Big)
    \end{equation}
    uniformly over all $t>0$ and $v\in\vV_{\beta,0}$. Moreover, for every $v\in\vV_{\beta,0}$, $p_t^{(k)}(\vec{y},z;v)$ can be decomposed as 
    \begin{equation}\label{e:decomposition}
        p_t^{(k)}(\vec{y},z;v)=\sum_{P\in\emP([k])}\Big(\prod_{B\in P}\Lambda^{(|B|)}_v(\vec{y}_B)\Big)m^{(|P|)}(z)+\hat{p}_{t}^{(k)}(\vec{y},z;v)\;,
    \end{equation}
    where $\Lambda_v^{(j)}$ for $j<k$ are the same functions appearing in the decomposition of $p_t^{(j)}$ (in particular, $\Lambda_v^{(1)}$ and $\Lambda_v^{(2)}$ are given by \eqref{e:lambda_v-1_expression} and \eqref{e:lambda_v-2_expression}), and $\Lambda_v^{(k)}$ is symmetric in its variables. This decomposition does not depend on $\sigma$. Furthermore, for every $\lambda\in(0,\overline{\lambda}]$, $\Lambda_v^{(k)}$ and $\hat{p}_t^{(k)}$ satisfy the bounds
    \begin{equation}\label{e:decomposition-1}
        \vert \Lambda^{(k)}_v(\vec{y})\vert \lesssim e^{-c_k\overline{\lambda} |\vec{y}|}\sum_{P\in\emP([k])}\Big(\prod_{B\in P}\sS^{(|B|)}_\sigma(\vec{y}_B)\Big)\;,
    \end{equation}
    \begin{equation}\label{e:decomposition-2}
        \Vert \hat{p}_{t}^{(k)}(\vec{y},\cdot\,;v)\Vert_{\lL^\infty_{-\lambda}}\lesssim t^{-2^k\sigma}e^{-c_kt}e^{-c_k\lambda|\vec{y}|}\sum_{P\in\emP([k])}\Big(\prod_{B\in P}\sS^{(|B|)}_\sigma(\vec{y}_B)\Big)
    \end{equation}
    uniformly over all $t>0$ and $v\in\vV_{\beta,0}$. 
\end{prop}
\begin{proof}
We prove by induction on $k$. The case $k=2$ is proved in Lemma~\ref{lem:decomposition2}. Now suppose the conclusion is proved for all $2\leq j<k$, where $k\leq K_\sigma$, and consider the decomposition for $p_t^{(k)}$.

The kernel $p_t^{(k)}(\Vec{y}, \cdot\,; v)$ satisfies \eqref{e:linearized_PDE} with $r_t = \r_t(v)$ given in \eqref{e:rtv} and $g_t = g_t^{(k)}(\vec{y};v)$ given in \eqref{e:g_t-k_defn}. By \eqref{e:r_decay}, the family $\big(\r_t(v) \big)_{t \geq 0}$ satisfies the assumption \eqref{e:ass_r_decay} in Proposition~\ref{prop:linearized_PDE}. We now check that this $g_t$ also satisfies the assumptions \eqref{e:ass_gt_uniform} and \eqref{e:ass_gt_ec} in Proposition~\ref{prop:linearized_PDE}. 

For $t \geq 1$, we apply H\"older inequality to the right-hand side of \eqref{e:g_t-k_defn} to get
\begin{equation*}
    \| g_t \|_{\lL_{-\lambda}^1} \lesssim \sum_{P \in \emP[k], |P| \neq 1} \prod_{B \in P} \big\| p_t^{(|B|)} (\vec{y}_B, \cdot\,; \, v) \big\|_{\lL_{-\lambda/2k}^\infty}\;.
\end{equation*}
Note that for $P \in \emP[k]$ with $|P| \neq 1$, each block $B \in P$ satisfies $|B| \leq k-1$. Hence, the induction hypothesis can be used for each $p_t^{(|B|)}$ on the right-hand side above. Together with Lemma~\ref{lem:prod_sum_prod_S}, we get
    \begin{equation*}
        \|g_t\|_{\lL^1_{-\lambda}} \lesssim e^{-c\lambda |\vec{y}|}\sum_{P\in\emP([k])}\Big(\prod_{B\in P}\sS^{(|B|)}_\sigma(\vec{y}_B)\Big)\;, \quad t \geq 1\;.
    \end{equation*}
    For $t \leq 1$, we need to be careful about the singularities from those $p_t^{(|B|)}$ with $|B|=1$. In this case, we first use H\"older inequality and $f\in \cC_b^{\infty}$ to get
    \begin{equ}
        \Vert g_t\Vert_{\lL^1_{-\lambda}} \lesssim \sum_{\substack{P\in\emP([k])\\|P|\neq 1}} \bigg( \Big\Vert\prod_{\substack{B \in P \\ |B|=1}}p_t(\vec{y}_B,\cdot\,;v)\Big\Vert_{\lL^1_{-\lambda/k}} \prod_{\substack{B\in P \\ |B|\neq 1}}\Vert p_t^{(|B|)}(\vec{y}_B,\cdot\,;v)\Vert_{\lL^\infty_{-\lambda/k}} \bigg)\;.
    \end{equ}
    We bound each term in the sum as follows. Let
    \begin{equation*}
        S(P) := \big\{ j \in [k]: \; \{j\}\in P \big\}
    \end{equation*}
    be the set of indices which are singleton sets in $P$. For the first term, if $|S(P)| \geq 3$, we use \eqref{e:technical1}; if $|S(P)| = 2$, we use H\"older inequality and \eqref{e:bound1-2}; if $|S(P)| = 1$, we use the inclusion $\lL^{2}_{-\lambda/2k}\hookrightarrow\lL^{1}_{-\lambda/k}$ and \eqref{e:bound1-2}. For the second term, we use the induction hypothesis on $p_t^{(j)}$ for $j <k$. Hence, by Lemma~\ref{lem:prod_sum_prod_S}, we get
    \begin{equation*}
    \begin{aligned}
        \Vert g_t\Vert_{\lL^1_{-\lambda}} &\lesssim (1+t^{-\frac{1}{2}-\sigma}) \, \sum_{P\in\emP([k]),|P|\neq 1}  \bigg[ e^{-c\lambda|\vec{y}_{S(P)}|} \, \sS_\sigma^{(|S(P)|)}(\vec{y}_{S(P)})\\
        &\qquad \prod_{B\in P, |B|\neq 1}\bigg((1+t^{-2^{|B|}\sigma})e^{-c\lambda |\vec{y}_B|}\sum_{P'\in\emP(B)}\Big(\prod_{B'\in P'}\sS^{(|B'|)}_\sigma(\vec{y}_{B'})\Big)\bigg) \bigg]\\
        &\lesssim \big(1+t^{-\frac{1}{2}-2^k\sigma} \big) \, e^{-c\lambda |\vec{y}|}\sum_{P\in\emP([k])}\Big(\prod_{B\in P}\sS^{(|B|)}_\sigma(\vec{y}_B)\Big)\;,
    \end{aligned}
    \end{equation*}
    where the constant $c$ on the exponential depends on $k$. This verifies the assumption \eqref{e:ass_gt_uniform}. Now, define 
    \begin{equ}
        g := \sum_{P\in\emP([k]),|P|\neq 1}f^{(|P|)}(m)\bigg(\prod_{B\in P}\Big(\sum_{P'\in\emP(B)}\big(\prod_{B'\in P'}\Lambda^{(|B'|)}_v(\vec{y}_{B'})\big)m^{(|P'|)}\Big)\bigg)\;.
    \end{equ}
    We need to check $\pP g = 0$ and to bound $\|g_t-g\|_{\lL^1_{-\lambda}}$. For $t\leq 1$, we bound $\|g_t\|_{\lL^1_{-\lambda}}$ and $\|g\|_{\lL^1_{-\lambda}}$ separately. For $t\geq 1$, similar arguments as in Lemma~\ref{lem:decomposition2} imply that we only need to control the quantity
    \begin{equ}
        \Big \Vert \prod_{B\in P} p_t^{(|B|)}(\vec{y}_B, \cdot\,; v) - \prod_{B\in P} \Big(\sum_{P'\in \emP(B)}\big(\prod_{B'\in P'}\Lambda_v^{(|B'|)}(\vec{y}_{B'})\big)m^{(|P'|)}\Big) \Big \Vert_{\lL^1_{-\lambda}}
    \end{equ}
    for each partition $P\in \emP([k])$ with $|P|\neq 1$. Note that here we are treating products of arbitrary numbers, so we choose to first use $\lL^{\infty}_{-\lambda/2} \hookrightarrow \lL^1_{-\lambda}$, so that we will not lose integrability when using H\"older inequality. Then the bound follows by using the decompositions \eqref{e:decomposition1} and \eqref{e:decomposition} as well as bounds \eqref{e:decomposition1-2} and \eqref{e:decomposition-2} for $p^{(j)}_{t}$ with $j< k$. In conclusion, we get
    \begin{equation*}
        \Vert g_t-g\Vert_{\lL^1_{-\lambda}}\lesssim t^{-\frac{1}{2}-2^k\sigma}e^{-ct}e^{-c\lambda |\vec{y}|}\sum_{P\in\emP([k])}\Big(\prod_{B\in P}\sS^{(|B|)}_\sigma(\vec{y}_B)\Big)\;.
    \end{equation*}
    This verifies the bound \eqref{e:ass_gt_ec}. It remains to check that $\pP g = 0$. By Lemma~\ref{lem:prod_sum_prod} (with $\Gamma(1)=0$ and $\Gamma(k)=1$ for $k \geq 2$), we can rewrite $g$ as
    \begin{equation} \label{e:gv}
    g = \sum_{\substack{P^*\in\emP([k])\\|P^*|\neq 1}} \Big(\prod_{B^*\in P^*}\Lambda_v^{(|B^*|)}(\vec{y}_{B^*})\Big) \Big(\sum_{\substack{P \in\emP(P^*)\\|P|\neq 1}} f^{(|P|)}(m)\prod_{P'\in P}m^{(|P'|)}\Big)\;.
    \end{equation}
    Keeping differentiating $\aA m' = 0$, we see that for every $j \geq 1$, we have
    \begin{equ}
        \label{e:m_derivatives}
        \aA m^{(j)} = \sum_{P\in \emP([j]), |P|\neq 1} f^{(|P|)}(m)\prod_{P'\in P}m^{(|P'|)}\;.
    \end{equ}
    Therefore, the quantity in the second parentheses on the right-hand side of \eqref{e:gv} is precisely $\aA m^{(|P^*|)}$. Since $\bracket{\aA h, m'} = \bracket{h, \aA m'} = 0$, this implies $\pP g = 0$. Hence, we deduce that $g_t$ satisfies the assumptions of Proposition~\ref{prop:linearized_PDE} with
    \begin{equation*}
        M = M(\vec{y}) = C \, e^{-c\lambda |\vec{y}|}\sum_{P\in\emP([k])}\Big(\prod_{B\in P}\sS^{(|B|)}_\sigma(\vec{y}_B)\Big)
    \end{equation*}
    for some $c$ depending on $k$ and $C>0$ depending on $k$ and $\lambda$, and for the exponent ``$\sigma$" being $\frac{1}{2} - 2^k \sigma$. 
    
    The bound \eqref{e:bound} follows directly from Proposition~\ref{prop:linearized_PDE} with $p=1$ and $q=+\infty$. Also, from the identity \eqref{e:gv} and \eqref{e:m_derivatives}, one also sees that
    \begin{equation*}
        \aA^{-1} g = \sum_{P\in\emP([k]),|P|\neq 1}\Big(\prod_{B\in P}\Lambda^{(|B|)}_v(\vec{y}_B)\Big)\pP^\perp m^{(|P|)}\;,
    \end{equation*}
    where the meaning of $\aA^{-1}g$ is given in Lemma~\ref{lem:A_equation}. Let
    \begin{equation*}
        \tilde{g} =\int_0^{+\infty} \pP \Big( g_s + \r_s(v) \, p_s^{(k)}(\vec{y},\cdot\,;v) \Big) \,\md s + \sum_{\substack{P \in \emP([k])\\|P| \neq 1}} \Big(\prod_{B\in P}\Lambda^{(|B|)}_v(\vec{y}_B)\Big) \pP^\perp m^{(|P|)}\;.
    \end{equation*}
    Then by Proposition~\ref{prop:linearized_PDE}, we have the bounds 
    \begin{equation*}
        \begin{split}
        \Vert \tilde{g}\Vert_{\lL^\infty_{-\lambda}} &\lesssim e^{-c\lambda |\vec{y}|} \sum_{P\in\emP([k])}\Big(\prod_{B\in P}\sS^{(|B|)}_\sigma(\vec{y}_B)\Big)\;,\\
        \Vert p_t^{(k)}(\vec{y},\cdot\,;v) - \tilde{g}\Vert_{\lL^\infty_{-\lambda}} &\lesssim \big( 1 + t^{-2^k \sigma} \big) \, e^{-c't} \, e^{-c\lambda |\vec{y}|}\sum_{P\in\emP([k])}\Big(\prod_{B\in P}\sS^{(|B|)}_\sigma(\vec{y}_B)\Big)
        \end{split}
    \end{equation*}
    for some $c'>0$. Now set $\hat{p}_{t}^{(k)}(\Vec{y},z;v) = p_t^{(k)}(\Vec{y},z;v) - \tilde{g}(z)$, then \eqref{e:decomposition-2} holds. Since $\pP$ is ``projection" onto $\spann\{m'\}$, we see \eqref{e:decomposition} holds with $\Lambda_v^{(k)}$ which is uniquely determined by the relation
    \begin{equation*}
        \Lambda_{v}^{(k)}(\Vec{y}) \, m' = \int_0^{+\infty} \pP \Big( g_s + \r_s(v) \, p_s^{(k)}(\Vec{y},\cdot\,;v) \Big) \,\md s - \sum_{\substack{P \in \emP([k])\\|P| \neq 1}} \Big(\prod_{B\in P}\Lambda^{(|B|)}_v(\vec{y}_B)\Big)\pP m^{(|P|)}\;.
    \end{equation*}
    The symmetry of $\Lambda_v^{(k)}$ under permutations of its variables follows from that of $p_s^{(k)}$, $g_s$ and $\Lambda_v^{(j)}, j<k$. Furthermore, by the inductive hypothesis, 
    \begin{equation*}
        \vert \Lambda^{(j)}_v(\vec{y})\vert \lesssim e^{-c\overline{\lambda} |\vec{y}|}\sum_{P\in\emP([j])}\Big(\prod_{B\in P}\sS^{(|B|)}_\sigma(\vec{y}_B)\Big)
    \end{equation*}
    for $j<k$. Combining it with Lemma~\ref{lem:prod_sum_prod_S}, we deduce the bound \eqref{e:decomposition-1} follows. 
\end{proof}

\begin{rmk}
    The bounds \eqref{e:bound} and \eqref{e:decomposition-2} are \textit{not} true for $k=1$, in which case there is a more serious singularity in time given in Lemma~\ref{lem:decomposition1}. 
\end{rmk}

\subsection{Fr\'echet derivatives of the deterministic flow}
\label{sec:Ft_derivative}

We now show that $F^t$ is indeed infinitely Fr\'echet differentiable with $p_t^{(k)}$'s as kernels. This is the content of the following proposition. 

\begin{prop} \label{prop:Ft_derivative}
    For every $t>0$, the map $F^t: \lL^\infty \rightarrow \lL^\infty$ is infinitely Fr\'echet differentiable. Moreover, for every $k$ and $v \in \lL^\infty$, the kernel of $D^k F^t(v)$ is $p_t^{(k)}(v)$ in the sense that
    \begin{equation*}
        \langle D^k F^t(v), \, \vec{\varphi}_{[k]}  \rangle(z) = \int_{\RR^k} p_t^{(k)}(\vec{y}, z; v) \prod_{j=1}^{k} \varphi_j(y_j) \, \md \vec{y}
    \end{equation*}
    for every $\varphi_1, \dots, \varphi_k \in \lL^\infty$. 
\end{prop}
\begin{proof}
We prove by induction, starting with $k=1$. Fix $v \in \lL^\infty$. For every $h \in \lL^\infty$, let
\begin{equation*}
    u_h[t] := F^t(v+h) - F^t(v) - \int_\RR p_t(y, \cdot\,; v) h(y) \,\md y\;.
\end{equation*}
We need to show that for every $t>0$, $\|u_h[t]\|_{\lL^\infty} / \|h\|_{\lL^\infty} \rightarrow 0$ as $\|h\|_{\lL^\infty} \rightarrow 0$. $u_h$ satisfies the equation
\begin{equation*}
    \d_t u_h = \Delta u_h + f'\big( F^t(v) \big) u_h + R_h\;, \qquad u_h[0] = 0\;,
\end{equation*}
where
\begin{equation*}
    R_h[t] = f \big( F^t(v+h) \big) - f \big( F^t(v) \big) - f'\big( F^t(v) \big) \, \big( F^t(v+h) - F^t(v) \big)
\end{equation*}
satisfies
\begin{equation*}
    \|R_h[t]\|_{\lL^\infty} \lesssim \|F^t(v+h) - F^t(v)\|_{\lL^\infty}^2 \lesssim_t \|h\|_{\lL^\infty}^2\;.
\end{equation*}
Since $\|e^{(t-s)\Delta}\|_{\lL^\infty \rightarrow \lL^\infty} \leq 1$ and $f \in \cC_b^\infty$, we have
\begin{equation*}
    \|u_h[t]\|_{\lL^\infty} \lesssim \int_{0}^{t} \|R_h[s]\|_{\lL^\infty} \, \md s + \int_{0}^{t} \|u_h[s]\|_{\lL^\infty} \, \md s\;.
\end{equation*}
It then follows from Gr\"onwall's inequality and the bound for $R_h$ that
\begin{equation*}
    \|u_h[t]\|_{\lL^\infty} \lesssim_t \|h\|_{\lL^\infty}^2\;.
\end{equation*}
This finishes the case $k=1$. Now suppose the claim holds for $j=1, \dots, k$, and we wish to establish the case $k+1$. Fix $v \in \lL^\infty$ and $\varphi_1, \dots, \varphi_k \in \lL^\infty$. For every $h \in \lL^\infty$, let
\begin{equation*}
    w_h[t] := \langle D^k F^t(v+h) - D^k F^t(v)\,, \; \vec{\varphi}_{[k]} \rangle\;,
\end{equation*}
and
\begin{equation*}
    u_h[t] := w_h[t] - \int_{\RR^{k+1}} p_{t}^{(k+1)}(\vec{y}, \cdot\,; v) \, h(y_{k+1}) \, \prod_{j=1}^{k} \varphi_j(y_j) \, \md \vec{y}\;.
\end{equation*}
Both $w_h$ and $u_h$ depends on $\varphi_1, \dots, \varphi_k$. We need to show that for every $t>0$, $\|u_h[t]\|_{\lL^\infty} / \|h\|_{\lL^\infty} \rightarrow 0$ as $\|h\|_{\lL^\infty} \rightarrow 0$, uniformly over the choices $\varphi_1, \dots, \varphi_k$ with $\|\varphi_j\|_{\lL^\infty} \leq 1$. Write 
\begin{equation*}
    \begin{split}
    G_t^{(k)}(v) &:= \int_{\RR^k} g_t^{(k)}(\vec{y}; v) \prod_{j=1}^{k} \varphi_j (y_j) \, \md \vec{y}\;,\\
    \mathscr{G}_t^{(k+1)}(v,h) &:= \int_{\RR^{k+1}} g_{t}^{(k+1)} (\vec{y}; v) \, h(y_{k+1}) \, \prod_{j=1}^{k} \varphi_j(y_j) \, \md \vec{y}\;.
    \end{split}
\end{equation*}
Note that $u_h$ satisfies the equation
\begin{equation*}
    \d_t u_h = \Delta u_h + f'\big( F^t(v) \big) u_h + \Big( f' \big( F^t(v+h) \big) - f' \big( F^t(v) \big) \Big) \, w_h + R_h\;, \quad u_h[0] = 0\;,
\end{equation*}
where
\begin{equation*}
    \begin{split}
    R_h[t] = &\; G_t^{(k)}(v+h) - G_t^{(k)}(v) - \mathscr{G}_t^{(k+1)}(v,h)\\
    &+ \Big( f' \big( F^t(v+h) \big) - f' \big( F^t(v) \big) \Big) \, \langle D^k F^t(v)\,, \; \vec{\varphi}_{[k]} \rangle\;.
    \end{split}
\end{equation*}
Since $\|e^{(t-s)\Delta}\|_{\lL^\infty \rightarrow \lL^\infty} \leq 1$ and $f \in \cC_b^\infty$, we have
\begin{equation*}
    \|u_h[t]\|_{\lL^\infty} \lesssim_t \|h\|_{\lL^\infty} \int_{0}^{t} \|w_h[s]\|_{\lL^\infty} \md s + \int_0^t \|R_h[s]\|_{\lL^\infty} \md s + \int_0^t \|u_h[s]\|_{\lL^\infty} \md s\;.
\end{equation*}
For $w_h$, it satisfies
\begin{equation*}
    \begin{split}
    \d_t w_h = &\Delta w_h + f'\big( F^t(v+h) \big) w_h + G_t^{(k)}(v+h) - G_t^{(k)}(v)\\
    &+ \Big( f'\big( F^t(v+h) \big) - f'\big( F^t(v) \big) \Big) \, \langle D^k F^t(v)\,, \; \vec{\varphi}_{[k]} \rangle
    \end{split}
\end{equation*}
with $w_h[0] = 0$. Since $g_t^{(k)}$ are multi-linear functions of $p_t^{(j)}$ for $j \leq k-1$, it follows from Gr\"onwall's inequality and the induction hypothesis that 
\begin{equation*}
    \sup_{s \in [0,t]} \|w_h[s]\|_{\lL^\infty} =_t\oO(\|h\|_{\lL^\infty})\;.
\end{equation*}
Furthermore, the expression of $g_t^{(k)}$ and the induction hypothesis on Fr\'echet differentiability up to $k-1$ imply that
\begin{equation*}
    \mathscr{G}_{t}^{(k+1)}(v,h) = \langle D G_t^{(k)}(v)\,, \; h \rangle + f''\big( F^t(v) \big) \, \langle D F^t(v)\,, \; h \rangle \, \langle D^k F^t\,, \vec{\varphi}_{[k]} \rangle\;.
\end{equation*}
Hence we deduce
\begin{equation*}
    \sup_{s \in [0,t]}  \| R_h[s] \|_{\lL^\infty} =_t o\big( \|h\|_{\lL^\infty} \big)\;.
\end{equation*}
All bounds and convergences above are uniform over $\varphi_1, \dots, \varphi_k$ with $\|\varphi_j\|_{\lL^\infty} \leq 1$. The conclusion then follows from Gr\"owall's inequality. 
\end{proof}

\section{Proof of Theorems~\ref{thm:derivativesOfPsicork}, ~\ref{thm:zeta_continuity_bound} and~\ref{thm:functional_equation}}
\label{sec:proof_of_Theorems}

In this section, we prove differentiability properties of $\bar{\psi}^{\eps,\ell}$ defined in \eqref{e:psi_bar_eps_ell}, and use them to prove Theorems~\ref{thm:derivativesOfPsicork}, ~\ref{thm:zeta_continuity_bound} and~\ref{thm:functional_equation}. We start with $\bar{\psi}^{\eps,0} = \zeta$.

\subsection{Fr\'echet derivatives of the limiting functional $\zeta$}

The functional $\zeta$ and the deterministic flow $F^t$ are related by
\begin{equation*}
    m_{\zeta(v)} = \lim_{t \rightarrow +\infty} F^t(v)\;.
\end{equation*}
One would like to analyze derivatives of $\zeta$ via derivatives of $F^t$ for large $t$ by differentiating both sides above. One issue with this approach is that differentiability of $m_{\zeta(v)}$ with respect to $v$ does not directly give differentiability of $\zeta$. 

On the other hand, differentiability properties of the functional $\eta$ introduced in Proposition~\ref{prop:fermi} are much easier to obtain, and one has the more direct relation
\begin{equation*}
    \zeta(v) = \lim_{t \rightarrow +\infty} \eta \big( F^t(v) \big) =: \lim_{t \rightarrow +\infty} \eta_t(v) \;,
\end{equation*}
where we write $\eta_t = \eta \circ F^t$ for convenience. In what follows, we will analyze the differentiability properties of $\zeta$ from those of $\eta$ (given in Lemma~\ref{lem:derivativesOfEta}) via the above relation. We first give a preliminary lemma.

\begin{lem} \label{lem:mean_value_path}
    There exists $C > 0$ such that for every $v \in \vV_{\beta,0}$ with $v\neq m$, there exists $\theta \in \RR$ such that
    \begin{equation*}
        |\theta| \leq C \dist(v, \mM) \quad \text{and} \quad \|v-m_\theta\| < \min \left\{ \beta, \, 2 \dist(v, \mM) \right\}\;.
    \end{equation*}
    For the above $\theta$, the family $(u_s)_{s \in [0,1]}$ given by
    \begin{equation} \label{e:path_us}
        u_s := m_{s \theta} + s (v - m_\theta)
    \end{equation}
    satisfies $\{u_s:s \in [0,1]\} \subset \vV_\beta$ with $u_0=m$ and $u_1 = v$, and that $|\zeta(u_s)| \lesssim 1$ uniformly over $s \in [0,1]$ and $v \in \vV_{\beta,0}$. Furthermore, $\dot{u}_s = v - m_\theta - \theta m_{s\theta}'$ satisfies
    \begin{equation*}
        \sup_{s \in [0,1]} \|\dot{u}_s\|_{\lL^\infty} \lesssim \dist (v, \mM)\;.
    \end{equation*}
    All the proportionality constants above are independent of $v \in \vV_{\beta,0}$. 
\end{lem}
\begin{proof}
    Fix $v \in \vV_{\beta,0}$. The existence of a $\theta$ with the required bound on $\|v-m_\theta\|_{\lL^\infty}$ is immediate. For such a $\theta$, we have $|\eta(v)-\theta|\lesssim \dist(v,\mM)$ from Proposition~\ref{prop:fermi}. Furthermore, we have $|\eta(v)| \lesssim \dist(v,\mM)$ by Proposition~\ref{pr:LinftyEC} (recalling that $\zeta(v) = 0$). Hence, we have $|\theta| \lesssim \dist (v, \mM)$. 
    
    We now turn to the family $(u_s)_{s \in [0,1]}$. That $u_0=m$ and $u_1=v$ is immediate. Also, by the above property of $\theta$, we have
    \begin{equation*}
        \|u_s - m_{s\theta}\|_{\lL^\infty} = s \|v-m_\theta\|_{\lL^\infty} < \beta
    \end{equation*}
    for $s \in[0,1]$, and hence $u_s \in \vV_\beta$ for all $s \in [0,1]$. Moreover, by Propositions~\ref{prop:fermi} and~\ref{pr:LinftyEC}, we have
    \begin{equation*}
        |\eta(u_s) - s \theta| \lesssim 1\;, \qquad |\zeta(u_s) - \eta(u_s)| \lesssim \dist (u_s, \mM)\;.
    \end{equation*}
    Together with $|\theta| \lesssim 1$, this implies $|\zeta(u_s)| \lesssim 1$ for all $s\in[0,1]$. 

    Finally, it follows from the properties of $\theta$ that
    \begin{equation*}
        \sup_{s \in [0,1]} \|\dot{u}_s\|_{\lL^\infty} \leq \|v-m_\theta\|_{\lL^\infty} + |\theta| \cdot \|m'\|_{\lL^\infty} \lesssim \dist (v, \mM)\;.
    \end{equation*}
    All the bounds above are uniform over $v \in \vV_{\beta,0}$. This completes the proof. 
\end{proof}

\begin{prop} \label{prop:derivativesOfZeta}
    The functional $\zeta$ satisfies the following properties:
    \begin{enumerate}
        \item For every $n\in \NN$, $D^n\eta_t(v)$ converges to $-\Lambda^{(n)}_v$ in $\emL^n\big((\lL^\infty)^n, \RR\big)$ as $t \rightarrow +\infty$ uniformly in $v\in \vV_\beta$, where $\Lambda^{(n)}_v$ is viewed as a multi-linear map from $(\lL^\infty)^n$ to $\RR$. As a result, $\zeta$ is infinitely Fr\'echet differentiable in $\vV_\beta$, and $D^n\zeta(v)$ has a kernel which equals
        \begin{equation} \label{e:Dnzeta_identity}
            D^n\zeta(\Vec{y};v) = - \Lambda^{(n)}_v(\Vec{y})
        \end{equation}
        for all $n \in \NN$ and $v \in \vV_\beta$. 
        
        \item The derivatives of $\zeta$ satisfy the translation and reflection properties
        \begin{equation*}
            \begin{split}
            D^n\zeta(y_1,\cdots,y_n;\emS_z v) &= D^n\zeta(y_1-z,\cdots,y_n-z;v)\;,\\
            D^n\zeta(y_1,\cdots,y_n;\emR v) &= (-1)^{n-1} \, D^n\zeta(-y_1,\cdots,-y_n;v)
            \end{split}
        \end{equation*}
        for every $n \in \NN$ and $v \in \vV_\beta$. 

        \item There exists $\lambda>0$ such that
        \begin{equation} \label{e:Dzeta_uniform}
            \begin{split}
            |D\zeta(y;v)| &\lesssim e^{-\lambda|y-\zeta(v)|}\;,\\
            |D^2\zeta(y_1,y_2;v)| &\lesssim e^{-\lambda(|y_1-\zeta(v)|+|y_2-\zeta(v)|)}\;,
            \end{split}
        \end{equation}
        and
        \begin{equation*}
            \begin{split}
            |D\zeta(y;v) - D\zeta(y;m_{\zeta(v)})| &\lesssim e^{-\lambda|y-\zeta(v)|} \dist(v,\mM)\;,\\
            |D^2\zeta(y_1,y_2;v)-D^2\zeta(y_1,y_2;m_{\zeta(v)})| &\lesssim e^{-\lambda(|y_1-\zeta(v)|+|y_2-\zeta(v)|)}\dist(v,\mM)
            \end{split}
        \end{equation*}
        uniformly over $v \in \vV_{\beta}$ and $y, y_1, y_2 \in \RR$. 
    \end{enumerate}
\end{prop}
\begin{proof}
We first verify Assertions 2 and 3 by assuming Assertion 1. The translation and reflection properties of $D^n \zeta$ in Assertion 2 follow from the relations
\begin{equation*}
    \zeta (\emS_z v) = \zeta(v) + z\;, \qquad \zeta (\emR v) = - \zeta(v)\;,
\end{equation*}
and taking Fr\'echet derivatives in $v$ on both sides. We now turn to the first two uniform bounds in Assertion 3. For $v \in \vV_{\beta,0}$, they follow from $D^k \zeta(v) = - \Lambda_v^{(k)}$ in Assertion 1 and the bounds for $\Lambda_v^{(k)}$ in \eqref{e:decomposition1-1} and \eqref{e:decomposition2-1} for $k=1,2$. The bounds for general $v \in \vV_\beta$ then follow from the translation property in Assertion 2. 

For the bounds on differences between $D^k \zeta(v)$ and $D^k \zeta (m_{\zeta(v)})$, we again first fix $v \in \vV_{\beta,0}$ with $v \neq m$. We have
\begin{equation*}
    D^k \zeta (\vec{y}; v) - D^k \zeta (\vec{y}; m) = \int_{0}^{1} \int_\RR D^{k+1} \zeta (\vec{y}, z; u_s) \, \dot{u}_s (z) \,\md z \,\md s\;,
\end{equation*}
where $(u_s)_{s \in [0,1]}$ is as in \eqref{e:path_us}. By the bound on $\|\dot{u}_s\|_{\lL^\infty}$ in Lemma~\ref{lem:mean_value_path}, we have
\begin{equation*}
    |D^k \zeta(\vec{y}; v) - D^k \zeta(y;m)| \lesssim \Big( \sup_{s \in [0,1]} \|D^{k+1} \zeta(\vec{y}, \cdot\,; u_s)\|_{\lL^1} \Big) \cdot \dist (v, \mM)\;,
\end{equation*}
where $u_s \in \vV_\beta$ for all $s \in [0,1]$ and $|\zeta(u_s)| \lesssim 1$ uniformly over $s \in [0,1]$. For $k=1$, by the uniform bound for $D^2 \zeta$ in \eqref{e:Dzeta_uniform}, we have
\begin{equation*}
    \|D^2 \zeta (y, \cdot\,; u_s)\|_{\lL^1} = \int_\RR |D^2 \zeta (y,z;u_s)| \, \md z \lesssim e^{-\lambda |y-\zeta(u_s)|} \lesssim e^{-\lambda |y|}\;.
\end{equation*}
The desired bound for $D \zeta (y;v) - D \zeta(y;m)$ then follows. For $k=2$, we use \eqref{e:decomposition-1} and Remark~\ref{rmk:translation} so that
\begin{equation*}
    \begin{split}
    \big\| D^3\zeta(y_1,y_2,\cdot\,;u_s) \big\|_{\lL^1} &\lesssim e^{-\lambda(|y_1 - \zeta(u_s)| + |y_2-\zeta(u_s)|)} \int_\RR e^{-\lambda |z-\zeta(u_s)|} \big( 1 + \sS_\sigma^{(3)}(y_1,y_2,z) \big) \, \md z\\
    &\lesssim e^{-\lambda(|y_1| + |y_2|)}\;.
    \end{split}
\end{equation*}
This finishes the desired bounds for $D^k\zeta(\vec{y};v) - D^k \zeta(\vec{y};m)$ for $k=1,2$ and $v \in \vV_{\beta,0}$. The bounds for general $v \in \vV_\beta$ follows immediately from the translation property. 

Hence, it remains to prove Assertion 1. We will show that for every $n \in \NN$, there exists $c>0$ such that
\begin{equation} \label{e:EC_derivatives_eta}
    \big| \langle D^n\eta_t(v) + \Lambda_v^{(n)}, \;  \vec{\varphi}_{[n]} \rangle \big| \lesssim e^{-ct}\prod_{i=1}^n\Vert \varphi_i\Vert_{\lL^\infty} 
\end{equation}
uniformly over sufficiently large $t$, $v \in \vV_\beta$ and $\varphi_1, \dots, \varphi_n \in \lL^\infty$. The infinite differentiability of $\zeta$ and the identity \eqref{e:Dnzeta_identity} follow from \eqref{e:EC_derivatives_eta}, Lemma~\ref{lem:UCProperty} and the convergence $\eta_t(v) \rightarrow \zeta(v)$ as $t \rightarrow +\infty$. 

We again note that it suffices to prove \eqref{e:EC_derivatives_eta} for $v \in \vV_{\beta,0}$, and the general case follows then from the translation properties of $D^n \eta_t$ (since $\eta_t(\emS_z v) = \eta_t(v)+z$) and $\Lambda_v^{(n)}$; see Remark~\ref{rmk:translation} for details. So we focus on proving \eqref{e:EC_derivatives_eta} for $v \in \vV_{\beta,0}$. We give details for $n=1$, and give a sketch for the general $n$ situation, whose arguments are very similar but just heavier in notation. 

For $n=1$, we need to show
\begin{equation}\label{e:quantitativeConvergenceOfZeta}
    \vert \langle D\eta_t(v) + \Lambda_v^{(1)}, \varphi \rangle \vert\lesssim e^{-ct}\Vert \varphi\Vert_{\lL^\infty}
\end{equation}
uniformly over sufficiently large $t$, $v \in \vV_{\beta,0}$ and $\varphi \in \lL^\infty$. By chain rule and that $p_t$ is the kernel for $DF^t$ (see Proposition~\ref{prop:Ft_derivative} for the precise meaning), we have 
\begin{equation} \label{e:Detav_chain_rule}
    \bracket{D\eta_t (v), \, \varphi} = \big\langle (D \eta) \big( F^t(v) \big), \, \bracket{ DF^t(v), \varphi} \big\rangle\;,
\end{equation}
where
\begin{equation*}
    \bracket{DF^t(v),\varphi}(z) = \int_\RR p_t(y,z;v)\varphi(y)\,\md y\;.
\end{equation*}
By the decomposition of $p_t$ in Lemma~\ref{lem:decomposition1} and the bound \eqref{e:decomposition1-2}, we have 
\begin{equation*}
    \Vert \bracket{DF^t(v),\varphi}- \bracket{\Lambda^{(1)}_v,\varphi}m'\Vert_{\lL^\infty_{-\lambda}} \lesssim e^{-c_1t}\Vert \varphi\Vert_{\lL^\infty}\;,\quad t\geq 1\;.
\end{equation*}
As for $(D\eta) \big( F^t(v) \big)$, we have
\begin{equation*}
    (D\eta)\big(F^t(v)\big)- D\eta(m) = \int_{0}^{1} \left\langle (D^2 \eta) \big(m + \theta (F^t(v) - m) \big)\,, \; F^t(v) - m \right\rangle  \, \md \theta\;.
\end{equation*}
By Proposition~\ref{pr:LinftyEC}, there exists $T_0>0$ and $K>0$ such that $m + \theta \big( F^t(v) - m \big) \in \vV_{\beta,K}$ for all $\theta \in [0,1]$ and $t \geq T_0$. Hence by Assertion 1 in Lemma~\ref{lem:derivativesOfEta} and the exponential convergence \eqref{e:deterministic_flow_convergence}, we have
\begin{equation*}
    \Vert (D\eta)\big(F^t(v)\big)- D\eta(m)\Vert_{\lL^1_\lambda} \lesssim \|F^t(v)-m\|_{\lL^\infty} \cdot \sup_{u \in \vV_{\beta,K}} \|D^2 \eta(u)\|_{\lL^1_\lambda(\RR^2)} \lesssim e^{-ct}
\end{equation*}
for sufficiently large $t$. Plugging the above two exponential bounds back into \eqref{e:Detav_chain_rule}, and noting $\bracket{D \eta(m), m'} = -1$ (Assertion 2 in Lemma~\ref{lem:derivativesOfEta}), we conclude \eqref{e:quantitativeConvergenceOfZeta}. 

For $n \geq 2$ and $v \in \vV_{\beta,0}$, by the chain rule (Lemma~\ref{lem:chainRule_n}), we have
\begin{equation} \label{e:Dn_etat_chain_rule}
    \langle  D^n \eta_t(v), \;  \vec{\varphi}_{[n]} \rangle = \sum_{P\in\emP([n])} \Big\langle  (D^{|P|}\eta)\big(F^t(v)\big), \; \Big(\big\langle D^{|B|}F^t(v), \vec{\varphi}_B \big\rangle\Big)_{B\in P} \Big\rangle\;.
\end{equation}
We first identify the $t \rightarrow +\infty$ limit of the above expression, and then discuss how to quantify the convergence. By Assertion 1 in Lemma~\ref{lem:derivativesOfEta}, we have
\begin{equation} \label{e:EC_derivative_eta_1}
    (D^{|P|}\eta) \big( F^t(v) \big) \rightarrow D^{|P|}\eta(m)
\end{equation}
in $\lL^1_\lambda(\RR^n)$ as $t \rightarrow +\infty$. It is understood as convergence of their corresponding kernels. Also, by Proposition~\ref{prop:Ft_derivative}, the decompositions \eqref{e:decomposition1} and \eqref{e:decomposition}, their corresponding error bounds \eqref{e:decomposition1-2} and \eqref{e:decomposition-2} for $p_t^{(|B|)}$, and that
\begin{equation*}
    \int_{\RR^{|B|}} e^{-\lambda |\vec{y}_B|} \sS_{\sigma}^{(|B|)}(\vec{y}_B) \, \md \vec{y}_B < +\infty
\end{equation*}
for every block $B$ (see \eqref{e:partition_0} below for a more general version of this bound), we have
\begin{equation} \label{e:EC_derivative_eta_2}
    \langle D^{|B|}F^t(v), \, \vec{\varphi}_B \rangle \rightarrow \sum_{P'\in\emP(B)}  \bigg[ \Big( \prod_{B'\in P'}\bracket{\Lambda_v^{(|B'|)}, \,\vec{\varphi}_{B'}} \Big) \, m^{(|P'|)} \bigg]\;
\end{equation}
in $\lL_{-\lambda}^\infty(\RR)$ as $t \rightarrow +\infty$, and uniformly over $(\varphi_i)_{i \in B}$ with $\|\varphi_i\|_{\lL^\infty} \leq 1$. Plugging the above convergences back into the expression \eqref{e:Dn_etat_chain_rule} and re-organizing the sums according to Lemma~\ref{lem:prod_sum_prod}, we get
\begin{equation*}
    \begin{split}
    &\phantom{111}\langle  D^n \eta_t(v), \; \vec{\varphi}_{[n]} \rangle\\
    &\rightarrow \sum_{P\in\emP([n])} \bigg\langle  D^{|P|}\eta(m), \;  \Big( \sum_{Q \in\emP(B)}  \Big( \prod_{J\in Q}\bracket{\Lambda_v^{(|J|)},\vec{\varphi}_J} \Big) \, m^{(|Q|)} \Big)_{B\in P} \bigg\rangle\\
    &=\sum_{P^*\in\emP([n])} \bigg[\Big(\prod_{B^*\in P^*}\bracket{\Lambda_v^{(|B^*|)}, \; \vec{\varphi}_{B^*}}\Big)\Big(\sum_{P\in\emP(P^*)}\bracket{D^{|P|}\eta(m), \; \big(m^{(|P'|)}\big)_{P'\in P}}\Big)\bigg]\;.
    \end{split}
\end{equation*}
By Assertion 2 in Lemma~\ref{lem:derivativesOfEta}, the only non-zero term in the sum in the last line above is $P^* = \{[n]\}$ the trivial partition, which equals
\begin{equation*}
   \langle \Lambda_v^{(n)}, \; \vec{\varphi}_{[n]} \rangle \cdot \bracket{D \eta(m), m'} = - \langle \Lambda_v^{(n)}, \; \vec{\varphi}_{[n]} \rangle\;.
\end{equation*}
Hence, we deduce
\begin{equation*}
    \langle  D^n \eta_t(v) + \Lambda_v^{(n)}, \; \vec{\varphi}_{[n]} \rangle \rightarrow 0
\end{equation*}
as $t \rightarrow +\infty$. To quantify the above convergence to the level \eqref{e:EC_derivatives_eta}, we need to quantify the convergences in \eqref{e:EC_derivative_eta_1} and \eqref{e:EC_derivative_eta_2}. For \eqref{e:EC_derivative_eta_1}, the exponential rate of convergence can be obtained from Assertion 1 in Lemma~\ref{lem:derivativesOfEta}, the mean value inequality, and the exponential convergence of $F^t(v)$ to $m$ in Proposition~\ref{pr:LinftyEC}. The exponential convergence in \eqref{e:EC_derivative_eta_2} follows from the error bounds of $\hat{p}_{t}$ in \eqref{e:decomposition1-2} and \eqref{e:decomposition-2}. One then proceeds similarly as in the case for $n=1$ to obtain the quantitative bound \eqref{e:EC_derivatives_eta}. The proof of the proposition is thus complete. 
\end{proof}

\subsection{Behavior of the correctors}
\label{sec:correctors}

Now we finally turn back to $\bar{\psi}^{\eps,\ell}$ given in \eqref{e:psi_bar_eps_ell}. Recall the definition of $\rho_\eps$ from \eqref{e:rho}. We define $\psi^{\eps,\ell}$ and $\psi_t^{\eps,\ell}$ by
\begin{equation}\label{e:psiellt}
    \psi^{\eps,\ell}(v) := \int_{\RR^2} D^2 \bar{\psi}^{\eps,\ell-1} (y_1,y_2;v) \rho_\eps(y_1,y_2)\,\md y_1\,\md y_2\;, \; \; \psi^{\eps,\ell}_t (v) := \psi^{\eps,\ell} \big( F^t (v) \big)\;.
\end{equation}
Then $\Bar{\psi}^{\eps,\ell} = \int_{0}^{+\infty} \psi_{t}^{\eps,\ell} \,\md t$. Recall $\beta>0$ is the value such that both Propositions~\ref{prop:fermi} and~\ref{pr:LinftyEC} hold. Let $C_{\ex}$ be the proportionality constant in Proposition~\ref{pr:LinftyEC}. For $\ell \geq 0$, define
\begin{equation} \label{e:beta_ell}
    \beta_{\ell} := \frac{\beta}{C_{\ex}^{\ell}}\;.
\end{equation}
In particular, $\beta_\ell \leq \beta$ for each $\ell$, and $v \in \vV_{\beta_{\ell}}$ implies $F^t(v) \in \vV_{\beta_{\ell-1}}$ for all $t \geq 0$. 

To prove Theorem~\ref{thm:derivativesOfPsicork}, we need to control the first two Fr\'echet derivatives of $\bar{\psi}^{\eps,\ell}$ for every $\ell$. This can be attributed to control the time integral of $D^2 \psi^{\eps,\ell}_t$ over $t \geq 0$. We have the formal expression
\begin{equation}
\begin{aligned} \label{e:second_derivative_psi}
    D^2\psi^{\eps,\ell}_t(y_1,y_2;v) = &\int_{\RR^2} D^2\psi^{\eps,\ell}\big(z_1,z_2;F^t(v)\big) \, p_{t}(y_1,z_1;v) \, p_{t}(y_2,z_2;v) \,\md z_1\,\md z_2 \\
    &+ \int_\RR  D\psi^{\eps,\ell}\big(z;F^t(v)\big) \, p_t^{(2)}(y_1,y_2,z;v) \,\md z\;.
\end{aligned}
\end{equation}
Note that we have shown $p_t^{(k)}$ converges to a non-zero limit as $t \rightarrow +\infty$, so they cannot give any decay in $t$ to control the right-hand side of \eqref{e:second_derivative_psi} factor by factor. One key ingredient in the proof below is to explore cancellations in the integration on the right-hand side of \eqref{e:second_derivative_psi}, which give decay in time. Furthermore, we also obtain uniform decay in $\vec{y}$ so that our objectives in \eqref{e:psiellt} can be shown to be well-behaved. We will justify them in Theorem~\ref{thm:derivativesOfPsik}. 

We give two lemmas before stating the essential properties of $\psi^{\eps,\ell}$.

\begin{lem} \label{lem:Lp1_inter_bound}
    Recall $\sS_{\sigma}$ defined in \eqref{e:Ssigmak}. For every $n \geq 0$, $\lambda>0$ and $\sigma \in (0,1)$, there exists $\delta > 0$ such that for every $P \in \emP([n+2])$ and $p \in [1,1+\delta)$, we have
    \begin{equation} \label{e:Lp1_inter_bound}
        \sup_{\eps\in(0,1)}\int_{\RR^{n+2}} e^{-p \lambda |\vec{y}|} \Big( \prod_{B \in P} \sS_{\sigma}^{(|B|)}(\vec{y}_B) \Big)^p \big|\rho_\eps (y_{n+1} , y_{n+2})\big| \,\md \vec{y} < +\infty\;.
    \end{equation}
\end{lem}
\begin{proof}
    Fix $P \in \emP([n+2])$. We first claim that there exists $\delta > 0$ such that for every $J \subset [n]$, every $\lambda>0$ and every $p \in [1,1+\delta)$, we have the bounds
\begin{equation} \label{e:partition_0}
    \int_{\RR^{|J|}} e^{-p\lambda |\vec{y}_J|} \big(\sS^{(|J|)}_\sigma(\vec{y}_J)\big)^p\,\md \vec{y}_J<\infty\;,
\end{equation}
\begin{equation} \label{e:partition_1}
    \sup_{z\in\RR}\int_{\RR^{|J|}} e^{-p\lambda |\vec{y}_J|}\big(\sS^{(|J|+1)}_\sigma(\vec{y}_J, z)\big)^p\,\md \vec{y}_J < \infty\;,
\end{equation}
\begin{equation} \label{e:partition_2}
    \sup_{z_1,z_2\in\RR}\int_{\RR^{|J|}} e^{-p\lambda |\vec{y}_J|}\big(\sS^{(|J|+2)}_\sigma(\vec{y}_J, z_1, z_2)\big)^p\,\md \vec{y}_J < \infty\;.
\end{equation}
The desired bound \eqref{e:Lp1_inter_bound} then follows from the above three bounds (using \eqref{e:partition_0} for $B \subset [n]$, \eqref{e:partition_1} for $B$ containing one of $n+1$ or $n+2$, and \eqref{e:partition_2} for $B$ containing both) and Lemma~\ref{lem:rho_boundedness}. 

Turning back to the bounds \eqref{e:partition_0} to \eqref{e:partition_2}, we first note from the definition of $\sS_\sigma^{(k)}$ in \eqref{e:Ssigmak} that
\begin{equation*}
    \begin{split}
    \sS_\sigma^{(|J|+1)} (\vec{y}_J, z) &\lesssim \Big( \sum_{j \in J} |y_j - z| \Big)^{-(|J|-1) + 2\sigma}\;,\\
    \sS_\sigma^{(|J|+2)} (\vec{y}_J, z_1, z_2) &\lesssim \Big( \sum_{j \in J} |y_j - z_1| \Big)^{-|J| + 2\sigma}\;,
    \end{split}
\end{equation*}
which imply \eqref{e:partition_1} and \eqref{e:partition_2} directly. The choice of $\delta$ is to make the above two singularities in $\vec{y}_J$ integrable. These are also valid for $J = \emptyset$ or $|J|=1$. The bound \eqref{e:partition_0} follows from \eqref{e:partition_1} by taking out any single element from $J$ and applying \eqref{e:partition_1} to integrate separately. 
\end{proof}

\begin{lem} \label{lem:Holder}
    For every $\lambda > 0$ and $p \in[1,+\infty)$ , we have
    \begin{equation*}
        \left\| \int_{\RR^2} F(\cdot\,, \vec{z}) \rho_\eps (z_1, z_2) \, \md \vec{z} \right\|_{\lL_\lambda^p (\RR^n)}^p \lesssim \int_{\RR^{n+2}} e^{p \lambda |\vec{y}|} |F(\vec{y})|^p |\rho_\eps (y_{n+1}, y_{n+2})| \, \md \vec{y}\;.
    \end{equation*}
\end{lem}
\begin{proof}
    By Lemma~\ref{lem:rho_boundedness}, for $\lambda>0$, $e^{-\lambda |\vec{z}|} |\rho_\eps(\vec{z})| \in \lL^1(\RR^2)$. Hence, Jensen's inequality applied to the finite measure $e^{-\lambda |\vec{z}|} |\rho_\eps(\vec{z})| \,\md \vec{z}$ implies
    \begin{equation*}
        \Big| \int_{\RR^2} F(\vec{x}, \vec{z}) \rho_\eps (\vec{z}) \, \md \vec{z} \, \Big| \lesssim \int_{\RR^2} |F(\vec{x}, \vec{z})|^{p} e^{p \lambda |\vec{z}|} |\rho_\eps(\vec{z})| \, \md \vec{z}\;.
    \end{equation*}
    The conclusion then follows by writing $\vec{y} = (\vec{x}, \vec{z}) \in \RR^{n+2}$. 
\end{proof}

We have the following theorem regarding derivatives of $\psi^{\eps,\ell}$.

\begin{thm} \label{thm:derivativesOfPsik}
Recall the definition of $\beta_\ell$ from \eqref{e:beta_ell}. For every $\ell\geq 1$, the recursive definition \eqref{e:psiellt} is well-posed in $\vV_{\beta_{\ell-1}}$. Furthermore, each $\psi^{\eps,\ell}: \vV_{\beta_{\ell-1}} \rightarrow \RR$ is infinitely Fr\'echet differentiable, and $D^n\psi^{\eps,\ell}(v)$ has a kernel for every $n\in\NN$ and $v\in\vV_{\beta_{\ell-1}}$ (still denoted by $D^n\psi^{\eps,\ell}(v)$). Moreover, the following properties hold for all $\ell \geq 1$ and $n \geq 0$. 
    \begin{enumerate}
        \item There exist $\lambda = \lambda_{\ell,n} > 0$ and $\delta = \delta_{\ell,n}>0$ such that for every $p\in[1,1+\delta)$, we have
        \begin{equation}\label{e:boundednessOfDpsik}
            \sup_{\eps\in(0,1)}\sup_{v\in\vV_{\beta_{\ell-1},0}} \Vert D^n\psi^{\eps,\ell}(v)\Vert_{\lL^p_{\lambda}(\RR^n)} < \infty\;.
        \end{equation}
        
        \item We have the bound
        \begin{equation*}
            \Vert D^n\psi^{\eps,\ell}(v)-D^n\psi^{\eps,\ell}(m)\Vert _{\lL^1_{\lambda}(\RR^n)} \lesssim \dist(v,\mM)
        \end{equation*}
        uniformly over all $v\in \vV_{\beta_{\ell-1},0}$.
        
        \item $\psi^{\eps,\ell}$ is translation invariant and ``odd"; that is, $\psi^{\eps,\ell}(\emS_\theta v) = \psi^{\eps,\ell}(v)$ and $\psi^{\eps,\ell}(\emR v) = -\psi^{\eps,\ell}(v)$. 
        
        \item For every $\theta \in \RR$, we have the identities $\psi^{\eps,\ell}(m_\theta) \equiv 0$, and
        \begin{equation}\label{e:cancellation}  
        \sum_{P\in\emP([n])} \Big\langle D^{|P|}\psi^{\eps,\ell}(m_{\theta})\,, \; \big(m^{(|B|)}_{\theta} \big)_{B\in P}\Big\rangle = 0\;, \quad n \geq 1\;.
        \end{equation}
    \end{enumerate}
\end{thm}
\begin{proof}
We prove by induction on $\ell$, starting with $\ell=1$. By Proposition~\ref{prop:derivativesOfZeta} and the expression \eqref{e:rho}, $\psi^{\eps,1}(v) = \langle D^2 \zeta(v), \rho_\eps \rangle$ is well-defined and infinitely Fr\'echet differentiable in $\vV_\beta$. Furthermore, its $n$-th Fr\'echet derivative $D^n \psi^{\eps,1}$ has a kernel which equals
\begin{equation*}
    D^n \psi^{\eps,1} (\vec{y}; v) = \int_{\RR^2} D^{n+2} \zeta (\vec{y}, y_{n+1}, y_{n+2}; v) \rho_\eps (y_{n+1}, y_{n+2}) \, \md y_{n+1}\,\md y_{n+2}\;, \quad \vec{y} \in \RR^n\;.
\end{equation*}
For Assertion 1, for fixed $n$, take $\sigma>0$ such that $K_\sigma>n+2$. By Lemma~\ref{lem:Holder} and Proposition~\ref{prop:derivativesOfZeta}, we have
\begin{equation*}
    \|D^n \psi^{\eps,1}(v)\|_{\lL_\lambda^p(\RR^n)}^p \lesssim \int_{\RR^{n+2}} e^{p \lambda |\vec{y}|} \, |\Lambda_v^{(n+2)}(\vec{y})|^p \, |\rho_\eps(y_{n+1}, y_{n+2})| \,\md \vec{y}\;,
\end{equation*}
where $\vec{y}$ now denotes the integration variable in $\RR^{n+2}$. Using the bound for $\Lambda_v^{(n+2)}$ in \eqref{e:decomposition-1}, we deduce that for every $p \geq 1$ and every sufficiently small $\lambda$ depending on $n$, there exists $\lambda'>0$ such that
\begin{equation*}
    \|D^n \psi^{\eps,1}(v)\|_{\lL_\lambda^p(\RR^n)}^p
    \lesssim \sum_{P \in \emP([n+2])} \int_{\RR^{n+2}} e^{-p \lambda' |\vec{y}|} \Big( \prod_{B \in P} \sS_\sigma^{(|B|)} (\vec{y}_B) \Big)^p |\rho_\eps(y_{n+1}, y_{n+2})| \,\md \vec{y}\;.
\end{equation*}
By Lemma~\ref{lem:Lp1_inter_bound}, we conclude there exist $\lambda$ and $\delta$ depending on $n$ such that \eqref{e:boundednessOfDpsik} holds for $p \in [1, 1+\delta)$ and $\ell=1$. 

Assertion 3 then follows from Assertion 2 in Proposition~\ref{prop:derivativesOfZeta}. For Assertion 4, the case $n=0$ (that is, $\psi^{\eps,1}(m_\theta) \equiv 0$) follows from Assertion 3 and that $\emR m = m$, and the identity \eqref{e:cancellation} for $n \geq 1$ follows from differentiating $\psi^{\eps,1}(m_\theta) = 0$ in $\theta$. For Assertion 2, we proceed similarly as the proof in Proposition~\ref{prop:derivativesOfZeta} to get
\begin{equation*}
    |D^n \psi^{\eps,1}(\vec{y}; v) - D^n \psi^{\eps,1}(\vec{y}; m)| \lesssim \Big( \sup_{u \in \vV_{\beta,K}} \|D^{n+1} \psi^{\eps,1} (\vec{y},\cdot\,;u)\|_{\lL^1(\RR)} \Big) \cdot \dist(v, \mM)\;.
\end{equation*}
By the bound in Assertion 1 with $p=1$ and translation invariance, we get
\begin{equation*}
    \|D^n \psi^{\eps,1}(v) - D^n \psi^{\eps,1}(m)\|_{\lL_\lambda^1(\RR^n)} \lesssim \dist(v, \mM)\;.
\end{equation*}
This finishes the proof for $\ell=1$.

Now suppose all the assertions hold for $\psi^{\eps,\ell}$, and we show that they also hold for $\psi^{\eps,\ell+1}$. The induction hypothesis in particular implies that for every $n$ and every $t>0$, $D^{n} \psi^{\eps,\ell}_t$ is well-defined and has a kernel in $\lL^1(\RR^n)$. Assuming one can interchange the Fr\'echet differentiation and integration in time, then $\psi^{\eps,\ell+1}$ is indeed well-defined and infinitely Fr\'echet differentiable with
\begin{equation} \label{e:derivative_psi_expression}
    D^{n} \psi^{\eps,\ell+1} (v) = \int_{0}^{+\infty} \langle D^{n+2} \psi_t^{\eps,\ell}(v), \, \rho_\eps \rangle \, \md t\;.
\end{equation}
We will discuss in Remark~\ref{rmk:derivative_psi_max_integrable} later why the interchange of differentiation and integration is valid. For the moment, we assume the expression \eqref{e:derivative_psi_expression}, and we need to control the integrand $\|\langle D^{n+2} \psi^{\eps,\ell}_t(v), \, \rho_\eps \rangle\|_{\lL_\lambda^p(\RR^n)}$. Using Lemma~\ref{lem:Holder}, we have 
\begin{equation} \label{e:derivative_psi_Lp_bound}
    \|\langle D^{n+2} \psi^{\eps,\ell}_t(v), \, \rho_\eps \rangle\|_{\lL_\lambda^p(\RR^n)}^p \lesssim \int_{\RR^{n+2}} e^{p \lambda |\vec{y}|} \, \big| D^{n+2} \psi_t^{\eps,\ell}(\vec{y}; v) \big|^p \, \big|\rho_\eps (y_{n+1} , y_{n+2})\big| \,\md \vec{y}\;.
\end{equation}
By chain rule in Lemma~\ref{lem:chainRule_n}, the kernel of $D^{n+2} \psi^{\eps,\ell}_t = D^{n+2} \big( \psi^{\eps,\ell} \circ F^t \big)$ has the expression
\begin{equation} \label{e:derivative_psi_kernel_sum}
    D^{n+2} \psi_t^{\eps,\ell}(\vec{y}; v) = \sum_{P \in \emP([n+2])} \iI_{P}(t,\vec{y}; v)\;,
\end{equation}
where
\begin{equation} \label{e:I_P_expression}
    \iI_{P}(t,\vec{y}; v) = \int_{\RR^{|P|}} D^{|P|}\psi^{\eps,\ell}\big(\vec{z};F^t(v)\big)\prod_{B\in P} p_t^{(|B|)}(\vec{y}_B,z_B;v) \,\md \vec{z}\;.
\end{equation}
Here, we write $\vec{z} = (z_B)_{B \in P} \in \RR^{|P|}$. We control \eqref{e:derivative_psi_Lp_bound} separately according to whether $t \geq 1$ or $t<1$. 

We start with $t \geq 1$. Note that neither $D^{|P|} \psi^{\eps,\ell}\big(F^t(v)\big)$ nor $p_t^{(k)}(v)$ has decay in time. Hence one does not expect to get decay in time for \eqref{e:derivative_psi_Lp_bound} if one naively controls $\iI_P(t, \vec{y}; v)$ by putting absolute value in the integrand in \eqref{e:I_P_expression}. Fortunately, it turns out that the integrand in \eqref{e:I_P_expression} can be decomposed into a constant-in-time part and an exponential-decay-in-time part. The first part turns out to be $0$ after integrating out the $\vec{z}$ variable and summing over $P\in\emP([n+2])$, while the other part has sufficient decay in time to make \eqref{e:derivative_psi_Lp_bound} well-behaved. 

We now start to describe the decomposition. For every $P \in \emP([n+2])$ and every $B \in P$, write
\begin{equation*}
    \begin{split}
    D^{|P|} \psi^{\eps,\ell} \big(\vec{z}; F^t(v)\big) &= D^{|P|} \psi^{\eps,\ell}(\vec{z}; m) + \Big( D^{|P|} \psi^{\eps,\ell} \big(\vec{z}; F^t(v)\big) - D^{|P|} \psi^{\eps,\ell}(\vec{z}; m) \Big)\;,\\
    p_t^{(|B|)} (\vec{y}_B, z_B; v) &= \Big( p_t^{(|B|)} (\vec{y}_B, z_B; v) - \hat{p}_t^{(|B|)} (\vec{y}_B, z_B; v) \Big) + \hat{p}_t^{(|B|)} (\vec{y}_B, z_B; v)\;,
    \end{split}
\end{equation*}
where we recall $\hat{p}_t^{(k)}$ from Section~\ref{sec:kernel} and that $\hat{p}_t^{(1)} = \hat{p}_t$. We split $\iI_P(t,\vec{y};v)$ by
\begin{equation*}
    \iI_P (t,\vec{y}; v) = \iI_P^{(1)} (\vec{y}; v) + \iI_P^{(2)} (t,\vec{y}; v)\;,
\end{equation*}
where
\begin{equation*} 
    \iI_P^{(1)} (\vec{y};v) = \int_{\RR^{|P|}} D^{|P|} \psi^{\eps,\ell}(\vec{z}; m) \prod_{B \in P} \Big( p_t^{(|B|)} (\vec{y}_B, z_B; v) - \hat{p}_t^{(|B|)} (\vec{y}_B, z_B; v) \Big) \,\md \vec{z}\;,
\end{equation*}
and $\iI_P^{(2)}(t,\vec{y};v)$ is the sum of the terms where the integrand consists of products of the rest. Note that $\iI_P^{(1)}$ does not depend on $t$ since $p_t^{(|B|)} - \hat{p}_t^{(|B|)}$ does not, and hence does not have any decay in time. By the decompositions for $p^{(|B|)}_t$ in \eqref{e:decomposition1} and \eqref{e:decomposition}, for every $P \in \emP([n+2])$, we have
\begin{equation*}
    \iI_{P}^{(1)} (\vec{y};v) = \bigg\langle D^{|P|}\psi^{\eps,\ell}(m), \; \Big(\sum_{Q_B\in\emP(B)}\Big(\prod_{J\in Q_B}\Lambda_v^{(|J|)}(\vec{y}_{J})\Big)m^{(|Q_B|)}\Big)_{B\in P}\bigg\rangle\;.
\end{equation*}
Summing over $P \in \emP([n+2])$ and applying Lemma~\ref{lem:prod_sum_prod} and Assertion 4 for $\ell$ (from induction hypothesis), we get
\begin{equation} \label{e:Lp_cancellation}
\begin{aligned}
    &\sum_{P \in \emP([n+2])} \iI_{P}^{(1)} (\vec{y};v)\\
    =&\sum_{P^*\in \emP([n+2])}\bigg[\Big(\prod_{B^*\in P^*}\Lambda_v^{(|B^*|)}(\vec{y}_{B^*})\Big)\Big( \sum_{P\in\emP(P^*)}\bracket{D^{|P|}\psi^{\eps,\ell}(m), \big(m^{(|P'|)}\big)_{P'\in P}}\Big)\bigg]\\
    =&\;0\;.
\end{aligned}
\end{equation}
We now turn to the part involving $\iI_P^{(2)}$, for which we do not need to explore cancellations. The integrand for $\iI_P^{(2)}$ is a linear combination of products of $|P|+1$ terms, each term containing at least a factor $\hat{p}_t^{(|B|)}$ for some $B \in P$, or the factor $D^{|P|} \psi^{\eps,\ell} \big(\vec{z}; F^t(v)\big) - D^{|P|} \psi^{\eps,\ell}(\vec{z}; m)$. By Assertions 1 and 2 for $\ell$, there exists $\lambda'$ depending on $n$ and $\ell$ such that 
\begin{equation*}
    \begin{split}
    \| (D^{|P|} \psi^{\eps,\ell})(m) \|_{\lL_{\lambda'}^{1}(\RR^{|P|})} &\lesssim 1\;,\\
    \| (D^{|P|} \psi^{\eps,\ell})\big(F^t(v)\big) - D^{|P|} \psi^{\eps,\ell}(m) \|_{\lL_{\lambda'}^1(\RR^{|P|})} &\lesssim \dist \big( F^t(v), \mM \big) \lesssim e^{-c^*t}\;,\\
    \end{split}
\end{equation*}
where the last exponential decay follows from Proposition~\ref{pr:LinftyEC}. We can further assume $\lambda' \leq \bar{\lambda}$ for the $\bar{\lambda}$ in Lemma~\ref{lem:decomposition1}. Hence, by Lemma~\ref{lem:decomposition1} and Proposition~\ref{pr:decomposition}, for the same $\lambda'$, there exists $c \in (0,c^*)$ depending on $n$ such that
\begin{equation*}
    \begin{split}
    \|p_{t}^{(B)}(\vec{y}_B, \cdot\,; v) - \hat{p}_{t}^{(B)}(\vec{y}_B, \cdot\,; v)\|_{\lL^\infty_{-\lambda'}} &\lesssim e^{-c\lambda'|\vec{y}_B|} \sum_{Q_B \in \emP(B)} \prod_{J \in Q_B} \sS_\sigma^{(|J|)}(\vec{y}_J)\;,\\
    \|\hat{p}_{t}^{(B)}(\vec{y}_B, \cdot\,; v)  \|_{\lL_{-\lambda'}^\infty} &\lesssim e^{-ct} e^{-c\lambda'|\vec{y}_B|} \sum_{Q_B \in \emP(B)} \prod_{J \in Q_B} \sS_\sigma^{(|J|)}(\vec{y}_J)\;,
    \end{split}
\end{equation*}
where we used that all derivatives of $m$ are uniformly bounded and that $t\geq 1$. Taking products over $B \in P$, applying Lemma~\ref{lem:prod_sum_prod_S}, and noticing that each product contains at least one exponential decay factor in time, we get the pointwise bound
\begin{equation} \label{e:IP_2_bound}
    |\iI_P^{(2)}(t,\vec{y};v)| \lesssim e^{-ct} e^{-c \lambda' |\vec{y}|} \sum_{P' \in \emP([n+2])} \prod_{B'\in P'}\sS_\sigma^{(|B'|)}(\vec{y}_{B'})
\end{equation}
for every $P \in \emP([n+2])$. Combining \eqref{e:derivative_psi_kernel_sum}, \eqref{e:Lp_cancellation} and \eqref{e:IP_2_bound}, we get
\begin{equation} \label{e:derivative_psi_large_time_pointwise}
    \big| D^{n+2} \psi^{\eps,\ell}_t (v) \big| \lesssim e^{-ct} e^{-c \lambda' |\vec{y}|} \sum_{P' \in \emP([n+2])} \prod_{B'\in P'}\sS_\sigma^{(|B'|)}(\vec{y}_{B'})
\end{equation}
uniformly over $\eps\in(0,1)$, $v \in \vV_{\beta_\ell,0}$ and $t \geq 1$. Plugging it back to \eqref{e:derivative_psi_Lp_bound} and applying Lemma~\ref{lem:Lp1_inter_bound}, we see that if $\lambda, \delta>0$ are sufficiently small (both depending on $n$ and $\ell$) and $p \in [1,1+\delta)$, we have 
\begin{equation} \label{e:Lp_large_time}
    \sup_{\eps \in (0,1)} \sup_{v \in \vV_{\beta_\ell,0}} \| \langle D^{n+2} \psi_t^{\eps,\ell}(v), \, \rho_\eps \rangle \|_{\lL_\lambda^p(\RR^n)} \lesssim e^{-ct}\;, \qquad t \geq 1\;.
\end{equation}

We now turn to the range $t \in [0,1]$. Here the problem is the singularity near $t \approx 0$. In particular, the factor $p_t^{(1)} = p_t$ has a singularity of order $t^{-\frac{1}{2}}$, and a product of two of them will create a non-integrable singularity in time. Hence, instead of exploring cancellation as in \eqref{e:Lp_cancellation}, we shall utilize the exponential decay in $\vec{y}$ from $\iI_P(t, \vec{y}; v)$ to exchange with singularity in time. More precisely, using Minkowski inequality and Lemma~\ref{lem:Holder}, for every $\lambda>0$ and $p \geq 1$, we have
\begin{equation*}
    \|\langle D^{n+2} \psi_t^{\eps,\ell}(v), \rho_\eps \rangle \|_{\lL_\lambda^p(\RR^n)} \lesssim \sum_{P \in \emP([n+2])} \int_{\RR^{|P|}} \big| D^{|P|} \psi^{\eps,\ell} \big( \vec{z}; F^t(v) \big) \big| \cdot  \big| \Gamma_{P}(t,\vec{z};v) \big|^{\frac{1}{p}} \,\md \vec{z}\;,
\end{equation*}
where
\begin{equation} \label{e:Gamma}
    \Gamma_{P}(t,\vec{z};v) = \int_{\RR^{n+2}} e^{p \lambda |\vec{y}|} \Big( \prod_{B \in P} \big| p_t^{(|B|)}(\vec{y}_B, z_B; v) \big|^p \Big) \, \big|\rho_\eps (y_{n+1} , y_{n+2})\big| \,\md \vec{y}\;.
\end{equation}
Here, $\vec{z} = (z_B)_{B \in P} \in \RR^{|P|}$. $\Gamma_{P}$ also depends on $\eps$, $\lambda$ and $p$ but we omit them in the notation for simplicity. To control $\Gamma_{P}(t, \vec{z}; v)$, we distinguish blocks $B \in P$ by whether it contains $n+1$ or $n+2$. 

We first consider blocks $B \in P$ that contain neither $n+1$ nor $n+2$. If $|B|=1$, we apply the Gaussian bound in Lemma~\ref{lem:Gaussian_small_time} to $p_t$. If $|B| \geq 2$, we apply \eqref{e:bound} to $p_t^{(|B|)}$ and then \eqref{e:partition_0} for integration over $\vec{y}_B$. In this way, we see there exists $C>0$ such that for all sufficiently small $\lambda, \delta>0$ and $p \in [1,1+\delta)$, we have
\begin{equation} \label{e:bounds_pure_blocks}
    \begin{split}
    \int_{\RR} e^{p \lambda |y_B|} \big| p_t(y_B, z_B; v) \big|^p \,\md y_B &\lesssim t^{-\frac{p-1}{2}} e^{C \lambda |z_B|}\;, \qquad |B|=1\;,\\
    \int_{\RR^{|B|}} e^{p \lambda |\vec{y}_B|} \big| p_t^{(|B|)}(\vec{y}_B, z_B; v) \big|^p \,\md \vec{y}_B &\lesssim t^{-2^{|B|} p \sigma} e^{C \lambda |z_B|}\;, \qquad |B| \geq 2\;.
    \end{split}
\end{equation}
We now turn to the blocks that contain $n+1$ or $n+2$. If $B$ contains both $n+1$ and $n+2$, we again use \eqref{e:bound} to control $p_t^{(|B|)}$ and then Lemma~\ref{lem:Lp1_inter_bound} to control the corresponding integral in $\vec{y}_B$. Then, we deduce that there exists $C>0$ such that for all sufficiently small $\lambda, \delta>0$ and $p \in [1,1+\delta)$, we have
\begin{equation*}
    \int_{\RR^{|B|}} e^{p \lambda |\vec{y}_B|} \big| p_t^{(|B|)}(\vec{y}_B, z_B; v) \big|^p \, \big|\rho_\eps (y_{n+1} , y_{n+2})\big| \,\md \vec{y}_B \lesssim t^{-2^{|B|} p \sigma} e^{C\lambda |z_B|}\;. 
\end{equation*}
We finally turn to the situation of two different blocks $B_1$ and $B_2$ containing $n+1$ and $n+2$ respectively. Their contribution to $|\Gamma_P(t, \vec{z}; v)|$ in \eqref{e:Gamma} is
\begin{equation} \label{e:annoying_blocks}
    \int_{\RR^{|B_1 \cup B_2|}} e^{p \lambda (|\vec{y}_{B_1}| + |\vec{y}_{B_2}|)} \Big( \prod_{i=1}^{2} \big| p_t^{(|B_i|)}(\vec{y}_{B_i}, z_{B_i}; v) \big|^p \Big) \, \big|\rho_\eps(y_{n+1} , y_{n+2})\big| \,\md \vec{y}_{B_1} \,\md \vec{y}_{B_2}\;.
\end{equation}
If $|B_1| \geq 2$ and $|B_2| \geq 2$, then similar as above, using \eqref{e:bound}, the expression \eqref{e:annoying_blocks} can be bounded by $t^{-2^{(|B_1| + |B_2|)} p \sigma} e^{C\lambda (|z_{B_1}| + |z_{B_2}|)}$ if $\lambda$ is sufficiently small. 

The situation is subtler if at least one of the two blocks have cardinality $1$. We assume without loss of generality that $|B_1|=1$. Using the Gaussian bound in Lemma~\ref{lem:Gaussian_small_time}, we first integrate out the variable $\vec{y}_{B_1} = y_{n+1}$ in \eqref{e:annoying_blocks} to get
\begin{equation*}
    \begin{split}
    &\phantom{111}\int_\RR e^{p \lambda |y_{n+1}|} |p_t (y_{n+1}, z_{B_1}; v)|^p |\rho_\eps(y_{n+1}, y_{n+2})| \, \md y_{n+1}\\
    &\lesssim t^{-\frac{p}{2}} e^{C\lambda |z_{B_1}|} \int_{\RR} e^{-p\lambda |y_{n+1}|} |\rho_\eps (y_{n+1}, y_{n+2})| \, \md y_{n+1} \lesssim t^{-\frac{p}{2}} e^{C\lambda(|z_{B_1}| + |y_{n+2}| )}\;,
    \end{split}
\end{equation*}
where the last bound follows from that $\rho_\eps$ is a function of $y_{n+1}-y_{n+2}$ and Lemma~\ref{lem:theta_boundedness}. Integrating out $y_{n+2}$ and using \eqref{e:bounds_pure_blocks}, we see that if $|B_1|=1$, \eqref{e:annoying_blocks} can be controlled by
\begin{equation*}
    t^{-\frac{p}{2}} \big( t^{-\frac{p-1}{2}} \mathbf{1}_{|B_2|=1} + t^{- 2^{|B_2|} p \sigma} \mathbf{1}_{|B_2| \geq 2} \big) \, e^{C\lambda |z_{B_2}|}\;.
\end{equation*}
Combining the above, we conclude that
\begin{equation*}
    \big| \Gamma_P(t,\vec{z}; v) \big| \lesssim t^{-p+\frac{1}{2} - 2^{n+2} p \sigma} e^{C\lambda |\vec{z}|}
\end{equation*}
for a possibly different $C$. Note that all the bounds above are uniform in $\eps \in (0,1)$ and $v \in \vV_{\beta_\ell,0}$. Hence if $C \lambda$ and $\delta>0$ are sufficiently small to apply Assertion 1 for $\ell$, $p \in [1, 1+\delta)$, and $\sigma>0$ sufficiently small depending on $p$, we deduce there exists $\sigma'\in(0,1)$ such that 
\begin{equation} \label{e:Lp_small_time}
    \sup_{\eps \in (0,1)} \sup_{v \in \vV_{\beta_\ell,0}} \|\langle D^{n+2} \psi_t^{\eps,\ell}(v), \rho_\eps \rangle \|_{\lL_\lambda^p(\RR^n)} \lesssim t^{-\sigma'}\;, \quad t \leq 1\;.
\end{equation}
Then, combining \eqref{e:Lp_large_time} and \eqref{e:Lp_small_time}, we conclude Assertion 1 for $\ell+1$. The other three assertions for $\ell+1$ follow from Assertion 1 for $\ell+1$ in the same way as the case for $\ell=1$. This completes the induction and the proof of the theorem. 
\end{proof}

\begin{rmk} \label{rmk:derivative_psi_max_integrable}
By Lemma~\ref{lem:commutativity}, the interchange between differentiation and integration is valid if one can show that
\begin{equation} \label{e:derivative_psi_max_integrable}
    \int_{0}^{+\infty} \sup_{v\in\vV_{\beta_{\ell}}}\|D^{n} \psi^{\eps,\ell}_t (v)\|_{\lL^1(\RR^n)} \, \md t < +\infty
\end{equation}
for every $n \geq 0$. We now briefly explain how one can adopt the arguments above almost directly to prove \eqref{e:derivative_psi_max_integrable}. One needs to prove the bounds \eqref{e:Lp_large_time} and \eqref{e:Lp_small_time} with $\|\langle D^{n+2} \psi^{\eps,\ell}_t(v), \rho_\eps \rangle\|_{\lL^p_\lambda(\RR^n)}$ replaced by
\begin{equation*}
    \|D^n \psi^{\eps,\ell}_t(v)\|_{\lL^1(\RR^n)} =\int_{\RR^n} \big| D^n \psi_t^{\eps,\ell}(\vec{y}; v) \big| \, \md \vec{y}\;.
\end{equation*}
For $t \geq 1$, the pointwise bound \eqref{e:derivative_psi_large_time_pointwise} still holds for $D^{n}\psi^{\eps,\ell}_t (\vec{y}; v)$ with $n+2$ replaced by $n$. For $t<1$ and $P \in \emP([n])$, the corresponding version of $\Gamma_P(t, \vec{z}; v)$ in \eqref{e:Gamma} becomes
\begin{equation*}
    \Gamma_P(t, \vec{z}; v) = \int_{\RR^n} \Big( \prod_{B \in P} | p_t^{(|B|)} (\vec{y}_B, z_B; v)|\Big) \, \md \vec{y} = \prod_{B \in P} \int_{\RR^{|B|}} \big| p_t^{(|B|)}(\vec{y}_B, z_B; v) \big| \, \md \vec{y}_B\;,
\end{equation*}
and the desired bound for $\Gamma_P(t, \vec{z}; v)$ follows immediately. Thus, one gets the same bounds as \eqref{e:Lp_large_time} and \eqref{e:Lp_small_time} and proves \eqref{e:derivative_psi_max_integrable}. 
\end{rmk}

\subsection{Proof of Theorem~\ref{thm:derivativesOfPsicork}}

We now give one final ingredient to prove Theorem~\ref{thm:derivativesOfPsicork}. 

\begin{lem}\label{lem:derivativesOfPsibark}
Recall $\beta_\ell$ from \eqref{e:beta_ell}. For every $\ell\geq 0$, $\bar{\psi}^{\eps,\ell}:\vV_{\beta_\ell}\to\RR$ is well-defined and twice Fr\'echet differentiable. Both $D\bar{\psi}^{\eps,\ell}(v)$ and $D^2\bar{\psi}^{\eps,\ell}(v)$ have kernels for every $v\in\vV_{\beta_\ell}$ (still denoted by $D\bar{\psi}^{\eps,\ell}(v)$ and $D^2\bar{\psi}^{\eps,\ell}(v)$). Moreover, there exists $\lambda>0$ such that one has the pointwise bounds
    \begin{equation} \label{e:Dpsikt}
        \big| D \bar{\psi}^{\eps,\ell}(y;v) \big| \lesssim e^{-\lambda |y-\zeta(v)|}\;,
    \end{equation}
    and
    \begin{equation} \label{e:D2psikt}
        \big| D^2 \bar{\psi}^{\eps,\ell}(y_1, y_2; v) \big| \lesssim e^{-\lambda \, ( \, |y_1 - \zeta(v)| + |y_2 - \zeta(v)| \, )}\;.
    \end{equation}
\end{lem}
\begin{proof}
The $\ell =0$ case follows directly from Proposition~\ref{prop:derivativesOfZeta}. We focus on  $\ell\geq 1$. Assertions 2 and 4 with $n=0$ in Theorem~\ref{thm:derivativesOfPsik} imply $|\psi^{\eps,\ell}(v)|\lesssim \dist(v,\mM)$ uniformly over all $v\in\vV_{\beta_{\ell-1},0}$. Together with Assertion 3 in Theorem~\ref{thm:derivativesOfPsik}, this can be extended to all $v\in\vV_{\beta_{\ell-1}}$. Hence combining with the exponential convergence \eqref{e:deterministic_flow_convergence}, we can conclude that $\bar{\psi}^{\eps,\ell}(v)$ is well-defined in $\vV_{\beta_\ell}$ and 
\begin{equation}\label{e:barpsik_bound}
    |\bar{\psi}^{\eps,\ell}(v)|\lesssim 1\;.
\end{equation}
Now we turn to \eqref{e:Dpsikt} and \eqref{e:D2psikt}. We give details for \eqref{e:D2psikt}. By translation invariance, we can assume without loss of generality that $\zeta(v) = 0$. Recall that
    \begin{equation*}
        \bar{\psi}^{\eps,\ell}(v) = \int_{0}^{+\infty} \psi_t^{\eps,\ell}(v) \,\md t\;.
    \end{equation*}
    By Remark~\ref{rmk:derivative_psi_max_integrable}, $\bar{\psi}^{\eps,\ell}$ is twice Fr\'echet integrable in $\vV_{\beta_\ell}$, and we can interchange the differentiation with integration to get
    \begin{equation*}
        D^2 \bar{\psi}^{\eps,\ell} (y_1, y_2; v) = \int_{0}^{+\infty} D^2 \psi^{\eps,\ell}_t (y_1, y_2; v) \, \md t\;.
    \end{equation*}
    We separate the integration into $t \geq 1$ and $t<1$. For $t \geq 1$, by \eqref{e:derivative_psi_large_time_pointwise}, we have
    \begin{equation} \label{e:D2psit_large_time}
        \big| D^2 \psi_t^{\eps,\ell} (y_1, y_2; v) \big| \lesssim e^{-ct} e^{-\lambda \,( \,|y_1|+|y_2| \,)}\;.
    \end{equation}
    For $t < 1$, the integrand can be expressed by
    \begin{equation*}
    \begin{split}
        D^2\psi^{\eps,\ell}_t(y_1,y_2;v) = &\int_{\RR^2} D^2\psi^{\eps,\ell}\big(z_1,z_2;F^t(v)\big) \prod_{i=1}^{2} p_t(y_i,z_i;v) \,\md z_1\,\md z_2 \\
        &+ \int_\RR D\psi^{\eps,\ell}\big(z;F^t(v)\big) \, p_t^{(2)}(y_1,y_2,z;v)\,\md z\;.
    \end{split}
    \end{equation*}
    Applying Assertion 1 in Theorem~\ref{thm:derivativesOfPsik}, Lemma~\ref{lem:bound1} and \eqref{e:bound2} in Lemma~\ref{lem:decomposition2}, we directly have
    \begin{equation} \label{e:D2psit_small_time}
    \begin{split}
        \vert D^2\psi^{\eps,\ell}_t(y_1,y_2;v)\vert \leq &\Vert (D^2\psi^{\eps,\ell})\big(F^t(v)\big) \Vert_{\lL^p_{\lambda}(\RR^2)}  \prod_{i=1}^{2} \Vert p_t(y_i,\cdot\,;v)\Vert_{\lL^{p'}_{-\lambda}}\\
        &+ \Vert (D\psi^{\eps,\ell})\big(F^t(v)\big)\Vert_{\lL^1_{\lambda}} \Vert p_t^{(2)}(y_1,y_2,\cdot\,;v)\Vert_{\lL^\infty_{-\lambda}}\\
        \lesssim & t^{-\frac{1}{p}}e^{-\lambda \,( \,|y_1| + |y_2| \,)\,}\;
    \end{split}
    \end{equation}
    for some $p>1$, where $p'$ satisfies $\frac{1}{p}+\frac{1}{p'}=1$. The bound \eqref{e:D2psikt} then follows by combining \eqref{e:D2psit_small_time} and \eqref{e:D2psit_large_time} and plugging them back to the expression of $D^2\bar{\psi}^{\eps,\ell}$. 

    The proof for the bound \eqref{e:Dpsikt} follows the same arguments but is only simpler, so we omit the details.
\end{proof}

We now have all the ingredients to prove Theorem~\ref{thm:derivativesOfPsicork}. 
\begin{proof}[Proof of Theorem~\ref{thm:derivativesOfPsicork}]
    By Proposition~\ref{prop:derivativesOfZeta} and Lemma~\ref{lem:derivativesOfPsibark},  $\bar{\psi}^{\eps,\ell}$ and $\psi_{cor}^{\eps,\ell}$ are well-defined and twice Fr\'echet differentiable, and the Fr\'echet derivatives have kernels for every $v \in \vV_{\beta_\ell}$. The assertion \eqref{e:psi_cor_eps_ell_bound} follows directly from \eqref{e:barpsik_bound}.

    For \eqref{e:Dpsi_cor_eps_ell_bound}, taking Fr\'echet derivative of $\psi^{\eps,\ell}_{cor}$, we get the expression
    \begin{equation*}
        D\psi^{\eps,\ell}_{cor}(v) = \frac{\eps^{2\ell\gamma}}{2^\ell}\Big((a^{2\ell}_\eps)'\big( \zeta(v) \big) \, \bar{\psi}^{\eps,\ell}(v) \, D\zeta(v) + a^{2\ell}_\eps\big( \zeta(v) \big) \, D\bar{\psi}^{\eps,\ell}(v)\Big)\;.
    \end{equation*}
    It then follows from \eqref{e:barpsik_bound}, \eqref{e:Dpsikt},  and the bound for $D \zeta$ in Proposition~\ref{prop:derivativesOfZeta} that there exists $\lambda > 0$ such that
    \begin{equation*}
        \vert D\psi^{\eps,\ell}_{cor}(y;v)\vert \lesssim \eps^{2\ell\gamma}e^{-\lambda|y-\zeta(v)|}\;.
    \end{equation*}
    Then \eqref{e:Dpsi_cor_eps_ell_bound} follows from Lemma~\ref{lem:rho_boundedness}. For \eqref{e:D2psi_cor_eps_ell_bound}, we have the expression
    \begin{equation*}
    \begin{aligned}
        D^2&\psi^{\eps,\ell}_{cor}(v) = \frac{\eps^{2\ell\gamma}}{2^\ell} \bigg( \bar{\psi}^{\eps,\ell}(v) \Big((a^{2\ell}_\eps)'\big( \zeta(v) \big) D^2\zeta(v) + (a^{2\ell}_\eps)''\big( \zeta(v) \big) D\zeta(v)\otimes D\zeta(v)\Big)\\ 
        &+(a^{2\ell}_\eps)'\big( \zeta(v) \big)\Big( D\zeta(v)\otimes D\bar{\psi}^{\eps,\ell}(v)  + D\bar{\psi}^{\eps,\ell}(v)\otimes D\zeta(v) \Big)+ a^{2\ell}_\eps\big( \zeta(v) \big) D^2\bar{\psi}^{\eps,\ell}(v) \bigg)\;.
    \end{aligned}
    \end{equation*}
    This time, using the bounds for $D\zeta$ and $D^2 \zeta$ in Proposition~\ref{prop:derivativesOfZeta}, and combining with \eqref{e:barpsik_bound}, \eqref{e:Dpsikt} and \eqref{e:D2psikt}, we get
    \begin{equation*}
        \vert D^2\psi^{\eps,\ell}_{cor}(y_1,y_2;v)\vert \lesssim \eps^{2\ell\gamma}e^{-\lambda \, ( \, |y_1 - \zeta(v)| + |y_2 - \zeta(v)| \, )}
    \end{equation*}
    for some $\lambda > 0$. Then \eqref{e:D2psi_cor_eps_ell_bound} follows from Lemma~\ref{lem:rho_boundedness}.
\end{proof}

\subsection{Proof of Theorem~\ref{thm:zeta_continuity_bound}}

\begin{lem}\label{lem:hatpm_regularity}
Recall $c^*>0$ from the beginning of Section~\ref{sec:linear_PDE} and the definition of $\hat{p}_t(\cdot\,,\cdot\,;m)$ from Lemma~\ref{lem:decomposition1}. For every $\alpha\in(0,1)$, we have
    \begin{equ}
        \|\hat{p}_t(y,\cdot\,;m)-\hat{p}_t(y',\cdot\,;m)\|_{\lL^2} \lesssim (1+t^{-\frac{1}{4}-\frac{\alpha}{2}})\,e^{-c^*t} |y-y'|^{\alpha} \;.
    \end{equ}
\end{lem}
\begin{proof}
    By the definition of $\hat{p}_t(\cdot\,,\cdot\,;m)$, we have
    \begin{equ}
    \hat{p}_t(y,\cdot\,;m)-\hat{p}_t(y',\cdot\,;m) = (e^{-t\aA} - \pP)(\delta_{y}-\delta_{y'})\;.
\end{equ}
By Proposition~\ref{prop:Linfty_decay}, the smoothing effect of $e^{t\Delta}$ and the boundedness of $f'$, we have
\begin{equ}
    \|e^{-t\aA}-\pP\|_{\hH^{-\frac{1}{2}-\alpha}\to\lL^2}\lesssim (1+t^{-\frac{1}{4}-\frac{\alpha}{2}})\,e^{-c^*t}\;.
\end{equ}
Meanwhile, we have
\begin{equ}
    \|\delta_{y}-\delta_{y'}\|_{\hH^{-\frac{1}{2}-\alpha}}^2 \lesssim \int_\RR \big(1+|\theta| \big)^{-1-2\alpha} |e^{2\pi i \theta y} - e^{2 \pi i \theta y'}|^2 \, \md \theta \lesssim|y-y'|^{2\alpha}\;,
\end{equ}
where we use the pointwise bound
\begin{equation*}
    |e^{2 \pi i \theta y} - e^{2 \pi i \theta y'}|^2 \lesssim \big(|\theta| \cdot |y-y'| \big)^2 \wedge 1
\end{equation*}
to get the bound for the integral in the last step. The conclusion then follows.
\end{proof}

\begin{lem}\label{lem:Dzeta_m_regularity}
$D\zeta(m)$ is a smooth function with all derivatives bounded. For $D^2 \zeta (m)$, for every $\alpha \in (0,1)$, we have
\begin{equ}
    \big|D^2\zeta(y_1,y_2;m)  - D^2\zeta(y'_1,y'_2;m)\big|\lesssim |y_1-y'_1|^{\alpha}+ |y_2-y'_2|^{\alpha}\;.
\end{equ}
\end{lem}
\begin{proof}
By \eqref{e:Dnzeta_identity} and \eqref{e:lambda_v-1_expression}, we have
\begin{equation*}
    D \zeta (m) = - \frac{m'}{\|m'\|_{\lL^2}}\;.
\end{equation*}
The claim for $D \zeta (m)$ then follows Lemma~\ref{lem:statationary_exponential_decay}. As for $D^2 \zeta (m)$, by \eqref{e:Dnzeta_identity} and \eqref{e:lambda_v-2_expression}, we have
\begin{align*}
    &D^2\zeta(y_1,y_2;m)  - D^2\zeta(y'_1,y_2;m)\\
    =& -\frac{1}{\|m'\|_{\lL^2}^2}\int_0^{+\infty} \int_\RR f''(m(z))\big(p_t(y_1,z) - p_t(y'_1,z)\big)p_t(y_2,z)m'(z)\,\md z\,\md t\;.
\end{align*}
Here, we omit the dependence of $p_t$ on $m$ for notational simplicity. By the decomposition \eqref{e:decomposition1}, we have
\begin{align*}
    \big(p_t(y_1,z) - &p_t(y'_1,z)\big)p_t(y_2,z) =  \big( m'(z) \big)^2  \big(\Lambda_m^{(1)}(y_1)-\Lambda_m^{(1)}(y'_1)\big)\Lambda_m^{(1)}(y_2)\\
    &+ m'(z)  \big(\Lambda_m^{(1)}(y_1)-\Lambda_m^{(1)}(y'_1)\big)\hat{p}_t(y_2,z) + \big(\hat{p}_t (y_1,z)-\hat{p}_t (y'_1,z)\big)p_t (y_2,z)\;.
\end{align*}
Note that $D^2\zeta(m)$ involves integration in time over $\RR^+$, but the first term on the right-hand side above has no decay in time. However, since $m$ is odd, integrating $z$ variable out implies this term is actually $0$. Hence, it follows from the fact that $\big|\Lambda_m^{(1)}(y_1)-\Lambda_m^{(1)}(y'_1)\big|\lesssim|y_1-y'_1|^\alpha$, Lemmas~\ref{lem:decomposition1} and~\ref{lem:hatpm_regularity} that
\begin{equation*}
    \big|D^2\zeta(y_1,y_2;m)  - D^2\zeta(y'_1,y_2;m)\big|\lesssim |y_1-y'_1|^\alpha\;.
\end{equation*}
Since $D^2\zeta(m)$ is symmetric in its two variables, this already implies the claim and hence completes the proof of the lemma. 
\end{proof}

Now we have all the ingredients to prove Theorem~\ref{thm:zeta_continuity_bound}.

\begin{proof}[Proof of Theorem~\ref{thm:zeta_continuity_bound}]
    A similar version is proved in \cite[Theorem~7.2, Theorem~7.3, Lemma~7.1]{Fun95}. The well-posedness of the expressions in \eqref{e:alpha} follows from Lemma~\ref{lem:Dzeta_m_regularity} and the translation properties of $D\zeta$ and $D^2\zeta$ (Assertion~2 in Proposition~\ref{prop:derivativesOfZeta}). The expressions for $\alpha_1$ and $\alpha_2$ involving the kernel $p_t(m)$ follow from \eqref{e:Dnzeta_identity}, \eqref{e:lambda_v-1_expression} and \eqref{e:lambda_v-2_expression}. Moreover, it follows from Assertion~3 in Proposition~\ref{prop:derivativesOfZeta} and Lemma~\ref{lem:rho_boundedness} that
    \begin{equation*}
    \Big| \int_{\RR^2}  \Big( \prod_{j=1}^{2} D \zeta (y_j; v)-\prod_{j=1}^{2} D \zeta (y_j; m_{\zeta(v)}) \Big) \cdot \rho_\eps (y_1, y_2) \,\md y_1 \,\md y_2  \Big| \lesssim \dist(v,\mM) \;,
\end{equation*}
\begin{equation*}
    \Big| \frac{1}{2} \int_{\RR^2} \Big(\sum_{j=1}^{2} \big(y_j- \zeta(v) \big) \Big) \cdot \Big(D^2\zeta(y_1,y_2;v)-D^2\zeta(y_1,y_2;m_{\zeta(v)})\Big) \cdot \rho_\eps (y_1,y_2) \, \md y_1 \,\md y_2 \Big| \lesssim\dist(v,\mM)\;.
\end{equation*}
Then the conclusion follows from the translation property in Assertion~2 in Proposition~\ref{prop:derivativesOfZeta}, the decay and H\"older continuity in $\vec{y}$ of $D\zeta(y_1;m)D\zeta(y_2;m)$ and $D^2\zeta(y_1,y_2;m)$ (from \eqref{e:Dzeta_uniform} and Lemma~\ref{lem:Dzeta_m_regularity}) and Lemma~\ref{lem:rho_convergence}.
\end{proof}

\subsection{Proof of Theorem~\ref{thm:functional_equation}}\label{sec:proof_of_functional_equation}

\begin{proof}[Proof of Theorem~\ref{thm:functional_equation}]
We start with \eqref{e:functional_equation}. For every $\ell\geq 1$, we have
\begin{equation*}
    - \psi^{\eps,\ell}(v) = \left.\frac{\md}{\md t}\right|_{t=0} \int_t^\infty \psi^{\eps,\ell}_s(v) \,\md s = \left.\frac{\md}{\md t}\right|_{t=0} \int_0^{+\infty} \psi^{\eps,\ell}_s\big(F^t(v)\big) \,\md s\;.
\end{equation*}
Since $v\in\cC^{5/2}$, the map $t\mapsto F^t(v)\in\lL^\infty$ is Fr\'echet differentiable. Then by chain rule and Remark~\ref{rmk:derivative_psi_max_integrable}, we have
\begin{equation*}
    \int_0^{+\infty} \sup_{t\in[0,\delta]}\left|\frac{\md}{\md t}\psi^{\eps,\ell}_s\big(F^t(v)\big)\right|\,\md s \lesssim \int_0^{+\infty} \sup_{u\in\vV_{\beta_\ell}} \|D\psi_s^{\eps,\ell}(u)\|_{\lL^1}\,\md s<+\infty
\end{equation*}
for sufficiently small $\delta$. Hence, we can interchange the differentiation and time integration in the expression for $-\psi^{\eps,\ell}(v)$ above to get
\begin{equation*}
    \begin{split}
    - \psi^{\eps,\ell}(v) &= \int_0^{+\infty} \left.\frac{\md}{\md t}\right|_{t=0}\psi^{\eps,\ell}_s\big(F^t(v)\big) \,\md s =\int_0^{+\infty} \bracket{D\psi^{\eps,\ell}_s(v), \Delta v + f(v)}\,\md s\\
    &= \bracket{D\bar{\psi}^{\eps,\ell}(v),\Delta v + f(v)}\;,
    \end{split}
\end{equation*}
where in the last step we again used Remark~\ref{rmk:derivative_psi_max_integrable} to interchange time integration and the Fr\'echet differentiation $D$. Combining this with the definition of $\psi^{\eps,\ell}$ \eqref{e:psiellt}, we can get \eqref{e:functional_equation} for $\ell\geq 1$. 

For $\ell=0$, we have $\psi^{\eps,0} = \zeta$. By definition of $\zeta$, we have $\zeta \big( F^t (v) \big) = \zeta (v)$. Differentiating this identity in 
$t$ and taking $t=0$, we obtain
\begin{equation*}
    \langle D \zeta(v), \; \Delta v + f(v)  \rangle = 0\;.
\end{equation*}
This completes the proof for \eqref{e:functional_equation}. 

Now we turn to \eqref{e:corrector_cancel_general}. The cases $\ell=0$ and $1$ follow directly from \eqref{e:functional_equation}. Now we suppose $\ell\geq 2$. Let $\phi_\eps := \frac{\eps^{2 \gamma}}{2} a_\eps^2$. By definition of $\psi_{cor}^{\eps,\ell}$, and applying \eqref{e:functional_equation} as well as the fact that $\bracket{D\zeta(v), \Delta v+f(v)} = 0$, we get
\begin{equation*}
    \begin{aligned}
    &\bracket{D\psi_{cor}^{\eps,\ell}(v), \Delta v + f(v)}+ \frac{1}{2}\eps^{2\gamma}a^2_\eps\big(\zeta(v)\big)\bracket{D^2\psi_{cor}^{\eps,\ell-1}(v),\rho_\eps}\\
    =& (\ell-1) \phi_\eps'  \phi_\eps^{\ell-1} \Big\langle \Big( D \zeta(v) \otimes D \bar{\psi}^{\eps,\ell-1}(v) +  D \bar{\psi}^{\eps,\ell-1}(v)\otimes D \zeta(v)+\bar{\psi}^{\eps,\ell-1}(v) D^2 \zeta (v) \Big),\rho_\eps \Big\rangle\\
    &+ (\ell-1) \phi_\eps^{\ell-2} \Big( \phi_\eps'' \phi_\eps  + (\ell-2) ( \phi_\eps')^2 \Big)  \bar{\psi}^{\eps,\ell-1}(v) \bracket{  D \zeta(v)\otimes D\zeta(v), \rho_\eps}\;,
\end{aligned}
\end{equation*}
where all powers and derivatives of $\phi_\eps$ are evaluated at $\zeta(v)$. Note that each appearance of $\phi_\eps$ (including its derivatives) gives a $\eps^{2\gamma}$ factor, with an additional $\eps^{\frac{1}{2}}$ from each differentiation. This gives the right power of $\eps$ to match the right-hand side of \eqref{e:corrector_cancel_general}. The conclusion then follows from the decay property of $D\zeta$, $D^2\zeta$ and $D\bar{\psi}^{\eps,\ell-1}$ and Lemma~\ref{lem:rho_boundedness}. 
\end{proof}

\section{Bounding the distance from $\mM$ -- proof of Theorem~\ref{thm:Linftycloseness}}
\label{sec:closeness}

We prove Theorem~\ref{thm:Linftycloseness} in this section. \cite{Fun95} proved an analogous $\lL^2$-based statement by energy estimates with regularization. In what follows, we take a different approach by dynamically decomposing the solution into a small linear part and a regular remainder, resetting the initial data of the linear part every time. This can be viewed as a dynamical version of the da Prato-Debussche trick, and has similar spirit in some recent works for global well-posedness of singular SPDEs (see \cite{MW17, HR23} for example). 

Throughout the section, we fix the small parameters $\kappa', \kappa_1$ and $\kappa_2$ such that 
\begin{equ}
    0 < \kappa_1 \ll \kappa_2 \ll \kappa' \ll (\gamma \wedge 1) \;.
\end{equ}
We also fix $\alpha^* \in (0, \frac{1}{5})$ and $p^* \in (\kappa_1^{-2}, +\infty)$ such that $\alpha^* p^* > 1$. The two ingredients are Lemmas~\ref{lem:GoodEvent} (based on Lemma~\ref{lem:smallnessOfLinearSolution}) and~\ref{lem:PerturbationVersion} below, from which Theorem~\ref{thm:Linftycloseness} follows. 

Let $X_{\eps}$ be the solution to
\begin{equation*}
    \partial_t X_\eps = \Delta X_\eps + \eps^{\gamma} a_\eps \dot{W}\;, \quad X_\eps[0] =0\;.
\end{equation*}
The following lemma proves the smallness of $X_{\eps}$ at a suitable time scale. 

\begin{lem}\label{lem:smallnessOfLinearSolution}
    For every $q \geq 1$, we have 
    \begin{equ}
        \EE \sup_{t \in [0, \eps^{-\kappa_1}]} \|X_{\eps}[t] \|_{\wW^{\alpha^*, p^*}}^q \lesssim_{q} \eps^{q(\gamma - 2\kappa_1)} \;.
    \end{equ}
\end{lem}
\begin{proof}
The result is standard, except that we are working on the whole real line with a cutoff in noise. We briefly explain how we address the bounds at spatial infinity. Let $q_t(x) = \frac{1}{\sqrt{4\pi t}} \exp(-\frac{x^2}{4t})$ be the heat kernel. We have the expression
\begin{equation*}
    X_\eps(t,x) = \eps^{\gamma}\int_0^t\int_\RR q_{t-s}(x-y)a_\eps(y)\dot{W}(s,y)\,\md y\,\md s\;.
\end{equation*}
With the same argument as in \cite[Proposition~3.2.2]{DSS24}, one can get
\begin{equation*}
    \EE \big| X_\eps(t,x)-X_\eps(t',x') \big|^2 \lesssim \eps^{2\gamma} \big( \sqrt{|t-t'|} + |x-x'| \big)
\end{equation*}
for all $(t, x)$, $(t', x') \in \RR^+ \times \RR$. On the other hand, since $a_{\eps}$ is supported on $[-\eps^{-1/2}, \eps^{-1/2}]$, we have
\begin{equs}
    \EE |X_{\eps}(t, x)|^2 &= \eps^{2\gamma} \int_0^t \int_\RR q_{t-s}^2(x-y) a_{\eps}^2(y) \,\md y \,\md s \\
    &\lesssim \eps^{2\gamma} \int_0^t \int_\RR \frac{1}{t-s} \cdot \frac{(t-s)^{3/2}}{|x-y|^3} \one_{|y| \leq \eps^{-1/2}} \,\md y \,\md s \lesssim \frac{\eps^{2(\gamma-\kappa_1)}}{\bracket{x}^2}
\end{equs}
for all $|x| \geq 2 \eps^{-1/2}$ and $t\in [0,\eps^{-\kappa_1}]$, where $\bracket{x} := (1 + |x|^2)^{\frac{1}{2}}$. Combining the above two bounds and using Nelson's estimate, for every $q \geq 2$, we get 
\begin{equ}
     \big\|X_\eps(t,x)-X_\eps(t',x')\big\|_{\lL_\omega^q} \lesssim_q \, \frac{\eps^{\gamma-\kappa_1}}{\bracket{x}^{\kappa_1}} \cdot \big( \sqrt{|t-t'|} + |x-x'| \big)^{\frac{1-\kappa_1}{2}}\;,
\end{equ}
which holds uniformly over all $(t,x)$ and $(t',x')$ such that $|t-t'| \leq 1$, $|x-x'| \leq 1$ and $t,t'\in[0,\eps^{-\kappa_1}]$. Thus, by the Kolmogorov continuity theorem (see \cite[Theorem~A.2.1]{DSS24}), for every $\alpha\in(0, \frac{1-\kappa_1}{2})$ and $q \geq 2$, we have
\begin{equ}
\label{e:Linear_object_Holder}
    \Big\| \sup_{\substack{(t, x), (t', x') \\
    \in B((t_0, x_0), 1)}} \frac{|X_{\eps}(t, x) - X_{\eps}(t', x')|}{\big(\sqrt{|t - t'|} + |x - x'|\big)^{\alpha}} \Big\|_{\lL_\omega^q}  \lesssim_{\alpha,q} \, \frac{\eps^{\gamma - \kappa_1}}{\bracket{x_0}^{\kappa_1}}
\end{equ}
for all $(t_0, x_0) \in [0,\eps^{-\kappa_1}] \times \RR$. Combining it with $X_{\eps}(0, x) = 0$, we get 
\begin{equ}
\label{e:linear_object_Lp}
    \Big\| \sup_{t \in [0, \eps^{-\kappa_1}]} \sup_{x \in B(x_0, 1)} |X_{\eps}(t, x)| \, \Big\|_{\lL_\omega^q} \lesssim_{\alpha, q} \frac{\eps^{\gamma - 2\kappa_1}}{\bracket{x_0}^{\kappa_1}}
\end{equ}
for all $x_0\in\RR$. Since \eqref{e:Linear_object_Holder} and \eqref{e:linear_object_Lp} hold for arbitrary $q$, one can again taking supremum over $x_0 \in \RR$ inside the $\|\cdot\|_{\lL_\omega^q}$ norm with the cost of $\bracket{x_0}^{-\kappa_1}$ factor. Hence, for every $\alpha \in (0, \frac{1-\kappa_1}{2})$ and $q \geq 2$, one gets 
\begin{equ}
\label{e:linear_object_Holder2}
    \Big\| \sup_{t \in [0, \eps^{-\kappa_1}]} \|X_{\eps}[t] \, \|_{\cC^{\alpha}} \Big\|_{\lL_\omega^q} \lesssim_{\alpha, q} \eps^{\gamma - 2\kappa_1}\;.
\end{equ}
Moreover, it follows from \eqref{e:linear_object_Lp} that for all $p > 1 + \frac{2}{\kappa_1}$, we have
\begin{equ}
\label{e:linear_object_Lp2}
    \Big\| \sup_{t \in [0, \eps^{-\kappa_1}]} \|X_{\eps}[t]\|_{\lL^{p}} \Big\|_{\lL_\omega^q} \lesssim_{p,q} \eps^{\gamma - 2\kappa_1}\;. 
\end{equ}
Interpolating \eqref{e:linear_object_Holder2} and \eqref{e:linear_object_Lp2} with suitable $\alpha$ and $p$ gives the desired result. 
\end{proof}

Now we perform an iterative argument over the time interval $[0,\eps^{-N}]$. For each half-integer $k \in \frac{1}{2} \ZZ \cap [0,\eps^{-N+\kappa_1}]$, define $X_\eps^{(k)}$ on the time interval $[k \eps^{-\kappa_1}, (k+1) \eps^{-\kappa_1}]$ by
\begin{equation*}
    \partial_t X_\eps^{(k)} = \Delta X_\eps^{(k)} +\eps^{\gamma} a_\eps \dot{W}\;, \qquad X_\eps^{(k)}[k \eps^{-\kappa_1}] =0\;.
\end{equation*}
Applying Lemma~\ref{lem:smallnessOfLinearSolution} to $X^{(k)}_\eps$ on the interval $[k \eps^{-\kappa_1}, (k+1) \eps^{-\kappa_1}]$ for each half integer $k \leq \eps^{-N+\kappa_1}$, we get the following lemma. 

\begin{lem}\label{lem:GoodEvent}
    For $N>0$, let $A_\eps^{(N)}$ be the event 
    \begin{equation}
    \label{e:good_event}
        A_\eps^{(N)}:=\bigg\{ \omega: \sup \limits_{\substack{t \in [k \eps^{-\kappa_1}, (k+1)\eps^{-\kappa_1}] \\ k\in\frac{1}{2}\ZZ \cap [0, \eps^{-N + \kappa_1}]}} \Vert X_\eps^{(k)}[t] \Vert_{\wW^{\alpha^*, p^*}} \leq  \eps^{\gamma - \kappa_2}\bigg\}\;.
    \end{equation}
    Then for every $N, N' > 0$, there exists $C>0$ such that
    \begin{equation*}
        \PP \big( A_\eps^{(N)} \big) > 1 - C \eps^{N'}
    \end{equation*}
    for all $\eps \in (0,1)$. 
\end{lem}

\begin{lem}\label{lem:SG_regularity}
Recall $c^*>0$ from the beginning of Section~\ref{sec:linear_PDE}. For every $p\in[1,+\infty)$ and $0 \leq \alpha \leq \alpha' < 2$, we have
    \begin{equation*}
        \|e^{-t\aA} -\pP\|_{\wW^{\alpha,p} \to \wW^{\alpha',p}} \lesssim (1+t^{-\frac{\alpha'-\alpha}{2}}) \, e^{-c^*t}
    \end{equation*}
    for all $t\geq 0$. 
\end{lem}
\begin{proof}
    This is a $\wW^{\alpha,p}$ version of Proposition~\ref{prop:Linfty_decay}. The claim follows from Proposition~\ref{prop:Linfty_decay}, the smoothing effect of $e^{t\Delta}$ and the boundedness of $f'$.
\end{proof}

We will prove iteratively that on the event $A_\eps^{(N)}$, the quantity $\dist_{\wW^{\alpha^*, p^*}}(u_\eps[t], \mM)$ is small for $t \in [0,\eps^{-N}]$. A key ingredient is the following deterministic lemma. 

\begin{lem} \label{lem:PerturbationVersion}
    Let $w_\eps$ be the solution to
    \begin{equation*}
        \partial_t w_\eps = \Delta w_\eps  + f(w_\eps)+R_\eps\;, 
    \end{equation*}
    with initial data $w_\eps [0]$ and $R_\eps$ satisfying
    \begin{equation*}
        \dist_{\wW^{\alpha^*, p^*}}(w_\eps[0], \mM) \leq \eps^{\gamma - \kappa'}\;, \qquad \sup_{t \in [0,\eps^{-\kappa_1}]} \|R_\eps[t]\|_{\lL^{p^*}} \leq \eps^{\gamma - 2\kappa_2}\;.
    \end{equation*}
    Then, there exist $\eps_0, T>0$ such that
    \begin{equs}
    \label{e:closeness_one_step_reg_1}
        &\sup \limits_{t \in [T, \eps^{-\kappa_1}]} \dist_{\wW^{\alpha^*, p^*}}(w_\eps[t], \mM) \leq \frac{1}{2} \eps^{\gamma - \kappa'}
    \end{equs}
    for all $\eps \in (0, \eps_0)$. Moreover, if $w_\eps[0] \in\mM$, then \eqref{e:closeness_one_step_reg_1} holds for all $t \in [0, \eps^{-\kappa_1}]$.
\end{lem}
\begin{proof}
    Recall the functional $\eta$ defined in Proposition~\ref{prop:fermi}. Let $D_\eps := w_\eps - m_{\eta(w_\eps[0])}$. By the assumption on $w_\eps[0]$, there exists $\theta \in \RR$ such that $\|w_\eps[0]-m_\theta\|_{\wW^{\alpha^*,p^*}}\leq 2\eps^{\gamma-\kappa'}$. Then by Proposition~\ref{prop:fermi} and the embedding $\wW^{\alpha^*,p^*}\hookrightarrow \lL^\infty$, we have
    \begin{equation*} 
        |\eta(w_\eps[0])-\theta|\lesssim\|w_\eps[0]-m_\theta\|_{\lL^\infty}\lesssim\eps^{\gamma-\kappa'}\;.
    \end{equation*}
    It then follows from triangle inequality and smoothness of $m$ that
    \begin{equation} \label{e:D_initial}
        \|D_\eps[0]\|_{\wW^{\alpha^*, p^*}} \leq \|w_\eps[0] - m_\theta\|_{\wW^{\alpha^*, p^*}} + \|m_\theta - m_{\eta(w_\eps[0])}\|_{\wW^{\alpha^*, p^*}} \lesssim \eps^{\gamma-\kappa'}\;.
    \end{equation}
    Now without loss of generality we assume $\eta(w_\eps[0]) = 0$. Then $D_\eps$ solves the equation
    \begin{equation*}
        \d_t D_\eps = \Delta D_\eps +f(w_\eps)- f(m) +R_\eps = -\aA D_\eps + Q_\eps + R_\eps\;,
    \end{equation*}
    where $\aA = -\Laplace - f'(m)$ is defined in \eqref{e:op_A} and 
    $Q_\eps := f(w_\eps)- f(m) - f'(m)D_\eps$. By mean value theorem, we have
    \begin{equation} \label{e:Q_bound}
        |Q_{\eps}(t,x)| \leq \Vert f''\Vert_{\lL^\infty} \|D_\eps[t]\|_{\lL^\infty} |D_{\eps}(t, x)|\;.
    \end{equation}
    By Duhamel's principle, we have
\begin{equation*}
    D_\eps[t]  =  e^{-t\aA} D_\eps[0] + \int_0^t e^{-(t-s)\aA} \big( Q_\eps[s] + R_\eps[s] \big) \,\md s\;.
\end{equation*}
Taking $\wW^{\alpha^*, p^*}$-norm on both sides above and applying Lemma~\ref{lem:SG_regularity} and the bound \eqref{e:D_initial}, we get
\begin{equation*}
    \Vert D_\eps[t]\Vert_{\wW^{\alpha^*, p^*}} \lesssim \,e^{-c^* t} \eps^{\gamma-\kappa'} +\int_0^t \big(1+(t-s)^{-\frac{\alpha^*}{2}} \big) \big( \Vert Q_\eps[s]\Vert_{\lL^{p^*}}+\Vert R_\eps[s]\Vert_{\lL^{p^*}} \big)\,\md s\;.
\end{equation*}
Here we used $\bracket{D_\eps[0], m'} = 0$ by the definition of $\eta$ to get the exponential factor $e^{-c^* t}$. Applying the bound \eqref{e:Q_bound} for $Q_{\eps}$, the Sobolev embedding, and the smallness assumption on $R_{\eps}$, we deduce there exists $C_0>0$ such that
\begin{equation} \label{e:D_Linfty_bound1}
    \begin{split}
    \|D_{\eps}[t]\|_{\wW^{\alpha^*, p^*}} \leq C_0 \bigg( &(1+t) \, \eps^{\gamma-2\kappa_2} + e^{-c^*t}\eps^{\gamma - \kappa'}\\
    &+ \int_0^{t} (1+(t-s)^{-\frac{\alpha^*}{2}}) \Vert D_\eps[s]\Vert_{\wW^{\alpha^*, p^*}}^2 \,\md s \bigg)\;.
    \end{split}
\end{equation}
Let
\begin{equation*}
    T^* := \inf \left\{t: \|D_{\eps}[t]\|_{\wW^{\alpha^*, p^*}} \geq 2C_0 \eps^{\gamma-\kappa'} \right\}\;.
\end{equation*}
Then we get
\begin{equation*}
    2C_0 \eps^{\gamma-\kappa'} \leq \|D_\eps[T^*] \|_{\wW^{\alpha^*, p^*}} \leq C_0 \Big(  (1+T^*) \eps^{\gamma - 2\kappa_2} + \eps^{\gamma - \kappa'} + (1+T^*) \big(2C_0 \eps^{\gamma - \kappa'} \big)^2 \Big)\;.
\end{equation*}
Hence, if $\eps $ is sufficiently small, we necessarily have $T^* \geq \eps^{-\kappa_1}$, which in turn implies $\|D_\eps[t]\|_{\wW^{\alpha^*, p^*}} \leq 2C_0 \eps^{\gamma - \kappa'}$ for all $t \in [0, \eps^{-\kappa_1}]$. Combining this with \eqref{e:D_Linfty_bound1}, we get 
\begin{equation}
\label{e:closeness_conclusion}
    \|D_{\eps}[t] \|_{\wW^{\alpha^*, p^*}} \leq 2C_0 \Big( e^{-c^* t} \eps^{\gamma - \kappa'} + \eps^{\gamma - \kappa_1 - 2\kappa_2} + \eps^{-\kappa_1} \big(2C_0 \eps^{\gamma - \kappa'} \big)^2 \Big)\;.
\end{equation}
Then, for sufficiently small $\eps_0$ and sufficiently large $T$, we have 
\begin{equ}
    \label{e:closeness_Linfty}
    \|w_\eps[t] - m\|_{\wW^{\alpha^*, p^*}} = \|D_{\eps}[t] \|_{\wW^{\alpha^*, p^*}} \leq \frac{1}{2} \eps^{\gamma - \kappa'}
\end{equ}
for all $\eps \in (0,\eps_0)$ and $t \in [T, \eps^{-\kappa_1}]$. Also note that if $w_0 \in \mM$, then $D_{\eps}[0] = 0$ and we wouldn't have the first term in \eqref{e:closeness_conclusion}. Therefore, in this case \eqref{e:closeness_Linfty} holds for all $t \in [0, \eps^{-\kappa_1}]$. 
\end{proof}

Now we prove the following proposition that (together with Lemma~\ref{lem:GoodEvent}) implies Theorem~\ref{thm:Linftycloseness}.

\begin{prop}\label{prop:iteration}
    Let $N>0$ and $A^{(N)}_{\eps}$ be the event defined in \eqref{e:good_event}. There exists $\eps_0 > 0$ such that for every $\eps\in(0,\eps_0)$ and every $\omega \in A^{(N)}_{\eps}$, the solution $u_\eps$ of \eqref{e:main_eqn} satisfies
    \begin{equation*}
        \sup \limits_{t \in [0, \eps^{-N}]} \dist_{\wW^{\alpha^*, p^*}} (u_\eps[t], \mM) \leq \eps^{\gamma - \kappa'}\;.
    \end{equation*}
\end{prop}
\begin{proof}
We will use Lemma~\ref{lem:PerturbationVersion} repeatedly for translations of the time intervals. Define $w_\eps^{(0)}[t] := u_\eps[t] - X_\eps^{(0)}[t]$ for $t\in [0,\eps^{-\kappa_1}]$. It satisfies the equation
\begin{equation*}
    \partial_t w_\eps^{(0)} = \Delta w_\eps^{(0)}  + f(w_\eps^{(0)})+R_\eps\;, \quad w_\eps^{(0)}[0] \in \mM\;,
\end{equation*}
where $R_\eps := f(u_\eps) - f(w_\eps^{(0)})$ satisfies
\begin{equation*}
    |R_{\eps}(t,x)| \leq \Vert f' \Vert_{\lL^{\infty}}|X_{\eps}^{(0)}(t,x)|\;.
\end{equation*}
In particular, $R_\eps$ satisfies the assumption in Lemma~\ref{lem:PerturbationVersion}.
By Lemma~\ref{lem:PerturbationVersion} and the definition \eqref{e:good_event} of $A^{(N)}_{\eps}$, for $\eps$ small enough we have
\begin{equs}
    \,&\sup_{t \in [0, \eps^{-\kappa_1}]} \dist_{\wW^{\alpha^*, p^*}} (u_{\eps}[t], \mM) \leq \frac{1}{2} \eps^{\gamma - \kappa'} + \eps^{\gamma - \kappa_2}\leq \eps^{\gamma-\kappa'}\;.
\end{equs}
Now, define $w_\eps^{(1/2)}[t] := u_\eps[t] - X_\eps^{(1/2)}[t]$ for $t \in [\frac{1}{2} \eps^{-\kappa_1}, \frac{3}{2} \eps^{-\kappa_1}]$. 
Applying the same argument to $w_{\eps}^{(1/2)}$, we would get
\begin{equation*}
    \sup_{t \in [\eps^{-\kappa_1}, \frac{3}{2}\eps^{-\kappa_1}]} \dist_{\wW^{\alpha^*, p^*}}(u_\eps[t], \mM)  \leq \eps^{\gamma-\kappa'}\;.
\end{equation*}
The result follows by iterating this procedure.
\end{proof}

\begin{rmk} \label{rmk:large_noise}
    If $\gamma<0$, then the solution $u_\eps$ becomes very far from $\mM$ in $\lL^\infty$-norm at any $\oO(1)$ time (actually even $o_\eps(1)$ time depending on the value of $\gamma$). To see this, let us suppose there exist $T>0$ and $M>0$ independent of $\eps$ such that
    \begin{equation*}
        \sup_{t \in [0,T]} \|u_\eps[t]\|_{\lL^\infty} \leq M\;.
    \end{equation*}
    Then for the same decomposition $u_\eps = X_\eps + w_\eps$ with $X_\eps[0] = 0$ and $w_\eps[0] \in \mM$ as above, we see that $\|X_\eps[T]\|_{\lL^\infty}$ is at least $\eps^{\gamma}$. Note that the remainder $w_\eps$ satisfies
    \begin{equation*}
        \d_t w_\eps = \Delta w_\eps + f(u_\eps)\;, \qquad w_\eps[0] \in \mM\;.
    \end{equation*}
    Since $\|u_\eps\|_{\lL^\infty} \leq M$ on $[0,T]$ and that $w_\eps[0] \in \mM$, we see that $w_\eps[T]$ is $\oO(1)$ distance from $\mM$, and in particular its $\lL^\infty$-norm is bounded. But this contradicts with the assumption that $u_\eps$ has bounded $\lL^\infty$-norm on $[0,T]$ since $X_\eps$ explodes. 
\end{rmk}

\appendix

\section{Fr\'echet derivatives}\label{sec:Frechet_derivatives}

In this Section, we give a few properties of Fr\'echet derivatives that are frequently used in the main text. These are all well-known and we omit their proofs. Throughout, $\xX$ and $\yY$ are two Banach spaces, and $\sS$ is an open set of $\xX$. Recall that $\emL(\xX, \yY)$ is the class of all bounded linear operators from $\xX$ to $\yY$.

We say a map $\Phi: \sS \rightarrow \yY$ is Fr\'echet differentiable if for every $v \in \sS$, there exists $\Psi (v) \in \emL (\xX, \yY)$ such that
\begin{equation*}
    \Vert \Phi(v+h) - \Phi(v) - \Psi(v)h\Vert_{\yY} = o(\Vert h\Vert_{\xX})
\end{equation*}
as $\Vert h\Vert_{\xX}\to 0$. We call the map $\Psi: \sS \rightarrow \emL (\xX, \yY)$ the Fr\'echet derivative of $\Phi$, denoted by $D \Phi$.

\begin{lem}[Uniform convergence] \label{lem:UCProperty}
    Let $\Phi_n: \sS \rightarrow \yY$ be Fr\'echet differentiable. Suppose there exist $\Phi: \sS \rightarrow \yY$ and $\Psi: \sS \rightarrow \emL(\xX, \yY)$ such that $\Phi_n(v) \rightarrow \Phi(v)$ for every $v \in \sS$, and
    \begin{equation*}
        \sup_{v \in \sS} \|D \Phi_n (v) - \Psi(v)\|_{\emL(\xX,\yY)} \rightarrow 0
    \end{equation*}
    as $n \rightarrow +\infty$. Then, $\Phi$ is Fr\'echet differentiable and $D\Phi = \Psi$.
\end{lem}

\begin{lem}\label{lem:commutativity}
    Let $(\Phi_t)_{t \in\RR}$ be a family of maps from $\sS$ to $\yY$ which are Fr\'echet differentiable. Suppose there exists a dominating function $M(t) \geq 0$ such that $\int_\RR M(t)\,\md t <\infty$ and 
    \begin{equation*}
        \sup_{v\in \sS} \Vert D\Phi_t(v)\Vert_{\emL(\xX,\yY)} \leq M(t)\;.
    \end{equation*}
    If moreover $\Psi(v):=\int_\RR \Phi_t(v)\,\md t$ exists for all $v\in \sS$, then $\Psi$ is Fr\'echet differentiable and
    \begin{equation*}
        D\Psi(v) = \int_\RR D\Phi_t(v) \,\md t\;.
    \end{equation*}
\end{lem}

\begin{lem}[Chain rule]\label{lem:chainRule_n}
    Let $\xX$, $\yY$ and $\zZ$ be Banach spaces. Let $\sS$ be an open subset of $\xX$ and $\tT$ be an open subset of $\yY$. Let 
    \begin{equation*}
        \Phi: \sS \rightarrow \tT\;, \qquad \Psi: \tT \rightarrow \zZ
    \end{equation*}
    be maps that are $n$-times Fr\'echet differentiable. Then the map $\Psi\circ \Phi: \sS \rightarrow \zZ$ is also $n$-times Fr\'echet differentiable, and
    \begin{equation*}
        \bracket{D^n(\Psi\circ \Phi)(v), \vec{\varphi}_{[n]}  } = \sum_{P\in\emP([n])} \Big \langle(D^{|P|}\Psi)(\Phi(v)), \Big(\bracket{D^{|B|}\Phi(v), \vec{\varphi}_B} \Big)_{B \in P} \Big \rangle\;.
    \end{equation*}
\end{lem}

\begin{lem} \label{lem:restriction_derivative}
    Let $\xX_0$ be a subspace of $\xX$ equipped with norm $\|\cdot\|_{\xX_0}$ such that there exists a constant $C>0$ satisfying
    \begin{equation*}
        \|v\|_{\xX} \leq C \|v\|_{\xX_0}
    \end{equation*}
    for all $v \in \xX_0$. Let $\sS_0$ be an open set in $\xX_0$ that is also contained in the set $\sS$. Let $\Phi: \sS \rightarrow \yY$ be $n$-times Fr\'echet differentiable on $\sS$ with Fr\'echet derivatives $D_{\xX}^n \Phi: \sS \rightarrow \emL^n (\xX^n, \yY)$. 
    
    Then, the restriction of $\Phi$ on $\sS_0$ is $n$-times Fr\'echet differentiable with respect to the structure induced by $\xX_0$ with Fr\'echet derivatives $D_{\xX_0}^n \Phi$. Furthermore, for every $v_0 \in \sS_0$, $D_{\xX_0}^n \Phi(v_0) = D_{\xX}^n \Phi (v_0)$ on $\xX_0^n$. 
\end{lem}

\section{Fr\'echet derivatives of the functional $\eta$}
\label{sec:fermi}

This section is devoted to the properties of the functional $\eta$ introduced in Proposition~\ref{prop:fermi}. We first give a proof to that proposition concerning the existence and basic properties of $\eta$.

\begin{proof} [Proof of Proposition~\ref{prop:fermi}]
    Fix $\beta>0$ sufficiently small (with value to be determined later). Let $v \in \vV_\beta$. We assume without loss of generality that $\|v-m\|_{\lL^\infty}< \beta$. For this $v$, let
    \begin{equation*}
        \phi(\eta) := \bracket{v, m_\eta'} = \bracket{m, m'_\eta} + \bracket{v-m, m'_\eta} =: \phi_1(\eta) + \phi_2(\eta)\;.
    \end{equation*}
    We first show that $\phi$ has a zero point $\eta^*$ if $\beta$ is sufficiently small. To see this, we note that $m(\pm \infty) = \pm 1$ and that $m'_\eta$ integrates to $2$ on the real line. This implies $\phi_1 (\pm \infty) = \pm 2$. Also, 
    \begin{equation*}
        |\phi_2(\eta)| \leq \|v-m\|_{\lL^\infty} \|m_\eta'\|_{\lL^1} \leq 2 \beta\;,
    \end{equation*}
    where we used $\|m_\eta'\|_{\lL^1} = 2$. Hence, if $\beta \leq \frac{1}{2}$, we have
    \begin{equation*}
        \phi(-\infty) < 0 < \phi(+\infty)\;.
    \end{equation*}
    Since $\phi$ is continuous in $\eta$, there exists $\eta^* \in \RR$ such that $\phi(\eta^*) = 0$. 

    We now turn to uniqueness. First note that
    \begin{equation*}
        \phi_1'(\eta) = - \bracket{m,m''_\eta} = \bracket{m', m'_\eta} \geq 0
    \end{equation*}
    for all $\eta \in \RR$. Furthermore, there exists $c>0$ such that $\phi_1'(\eta) \geq c$ for all $|\eta| \leq 1$. Let $\eta_0 \in \RR$ be such that $\phi(\eta_0) = 0$. Then we have
    \begin{equation*}
        |\phi_1(\eta_0) - \phi_1(0)| = |\phi_1(\eta_0)| = |\phi_2(\eta_0)| \leq 2\beta \;.
    \end{equation*}
    The first equality follows from the oddness of $m$ (so that $\phi_1(0) = \bracket{m,m'} = 0$), and the second equality uses that $\eta_0$ is a zero point of $\phi$. On the other hand, if $|\eta_0| > 1$, we have
    \begin{equation*}
        |\phi_1 (\eta_0) - \phi_1(0)| = \Big| \int_{0}^{\eta_0} \phi_1'(\eta) \,\md \eta \Big| \geq c\;,
    \end{equation*}
    which is a contradiction if $\beta < \frac{c}{2}$. Hence, for $\beta < \frac{c}{2}$, any zero point $\eta_0$ of $\phi$ must be in the interval $[-1,1]$. Now we prove the uniqueness of zero point in $[-1,1]$. Note that we have 
    \begin{equation*}
        |\phi_2'(\eta)| \leq \|v-m\|_{\lL^\infty} \|m''\|_{\lL^1} \leq \beta \|m''\|_{\lL^1}\;.
    \end{equation*}
    Hence, if $\beta < \frac{c}{2 \|m''\|_{\lL^1}}$, we have $\phi'(\eta) \geq \frac{c}{2}$ for all $\eta \in [-1,1]$. This proves the uniqueness of zero of $\phi$ in $[-1,1]$ for $\beta<\frac{c}{2\|m''\|_{\lL^1}}$, and hence in $\RR$ for $\beta< \frac{c}{2}\big(1\wedge\frac{1}{\|m''\|_{\lL^1}}\big)$.

    For the last bound, we note that since the zero point $\eta^*$ satisfies $|\eta^*| \leq 1$ and that $\phi_1' \geq c$ on $[-1,1]$, we have
    \begin{equation*}
        c \, |\eta^*| \leq \big| \phi_1(\eta^*) - \phi_1(0) \big| \leq |\phi_2(\eta^*)| \leq 2 \|v-m\|_{\lL^\infty}\;.
    \end{equation*}
    This implies
    \begin{equation*}
        |\eta^*| \leq \frac{2}{c} \cdot \|v-m\|_{\lL^\infty}\;,
    \end{equation*}
    and completes the proof of the lemma. 
\end{proof}

We now prove the infinite differentiability of $\eta$.

\begin{lem}\label{lem:derivativesOfEta}
    The functional $\eta$ satisfies the following properties:
    \begin{enumerate}
        \item It is infinitely Fr\'echet differentiable in $\vV_\beta$, and $D^n \eta (v)$ has a kernel for every $n \in \NN$ and $v \in \vV_\beta$ (still denoted by $D^n \eta(v)$). Furthermore, there exists $\lambda>0$ such that for every $n \in \NN$ and $K>0$, the kernel satisfies
        \begin{equation*}
            \sup_{v \in \vV_{\beta,K}} \|D^n \eta(v)\|_{\lL_\lambda^1(\RR^n)} < +\infty\;.
        \end{equation*}
        
        
        \item We have the identities
        \begin{equation*}
            \begin{split}
            \bracket{D\eta(m_\theta),m'_\theta} &= -1\;,\\
            \sum_{P\in\emP([n])} \Big\langle D^{|P|}\eta(m_{\theta}), \, \big(m^{(|B|)}_{\theta} \big)_{B\in P}\Big\rangle &= 0\;, \quad n \geq 2\;.
            \end{split}
        \end{equation*}
    \end{enumerate}
\end{lem}
\begin{proof}
We first prove that $\eta$ is infinitely Fr\'echet differentiable in $v\in\vV_\beta$, and that its derivatives are of the form
\begin{equation} \label{e:derivativesOfEta}
    D^n\eta(v) = \sum_{\vec{k}, \vec{\ell}} C_{\vec{k}, \vec{\ell}}\;\frac{\prod\limits_{j\geq 3} \bracket{m^{(j)}_{\eta(v)},v}^{\ell_j}}{\bracket{m''_{\eta(v)},v}^{\ell_2}}\bigotimes\limits_{i=1}^n m^{(k_i)}_{\eta(v)}\;. 
\end{equation}
Here, $\bigotimes_i m_{\eta(v)}^{(k_i)}$ is the tensor product of $m_{\eta(v)}^{(k_i)} \in \emL (\lL^\infty, \RR)$ in the sense that
\begin{equation*}
    \Big\langle \bigotimes\limits_{i=1}^n m^{(k_i)}_{\eta(v)}\,, \; \vec{\varphi}_{[n]}  \Big\rangle = \prod_{i=1}^{n} \langle m_{\eta(v)}^{(k_i)}, \, \varphi_i \rangle
\end{equation*}
for $\varphi_1, \dots, \varphi_n \in \lL^\infty$, and $C_{\vec{k}, \vec{\ell}}$ are constants depending on $\vec{k} = (k_1, \dots, k_n)$ and $\vec{\ell} = (\ell_2, \ell_3, \dots)$ such that only finitely many $C_{\vec{k}, \vec{\ell}}$ are non-zero. 

We prove \eqref{e:derivativesOfEta} by induction on $n$. The case for $n=1$ was proven in \cite[Lemma~9.5]{Fun95} in an $\lL^2$ neighborhood of $\mM$, but the same proof works in the $\lL^\infty$ setting. We repeat it here for completeness. For $v \in \vV_\beta$ and $h \in \lL^\infty$ with $\|h\|_{\lL^\infty}$ sufficiently small, we have
\begin{equation*}
    m_{\eta(v+h)}' (y) - m_{\eta(v)}'(y) = \big( \eta(v) - \eta(v+h) \big) \int_{0}^{1} m_{(1-s) \eta(v)+s \eta(v+h)}'' (y) \, \md s\;.
\end{equation*}
Multiplying both sides by $v$, integrating out $x \in \RR$ and using \eqref{e:fermi_defn}, we get
\begin{equation*}
    \eta(v+h) - \eta(v) = \frac{\langle m_{\eta(v+h)}', h \rangle}{\big\langle \int_{0}^{1} m_{(1-s) \eta(v)+s \eta(v+h)}'' \,\md s, v \big\rangle}\;.
\end{equation*}
It is then straightforward to check that
\begin{equation*}
    \frac{1}{\|h\|_{\lL^\infty}} \left\| \eta(v+h) - \eta(v) - \frac{\langle m_{\eta(v)}', h \rangle}{\langle m_{\eta(v)}'', v \rangle}  \right\|_{\lL^\infty} \rightarrow 0
\end{equation*}
as $\|h\|_{\lL^\infty} \rightarrow 0$. This shows $\eta$ is differentiable in $\vV_\beta$ with
\begin{equation} \label{e:Deta_expression}
    D\eta(v) = \frac{1}{\langle m_{\eta(v)}'', v \rangle} \cdot m'_{\eta(v)}\;,
\end{equation}
which is of the form \eqref{e:derivativesOfEta} with $n=1$. Now, suppose $D^n \eta$ exists and is of the form \eqref{e:derivativesOfEta}, then differentiating another time with respect to $v$, and applying chain rule and \eqref{e:Deta_expression}, we see that $D^{n+1} \eta$ is again of the form \eqref{e:derivativesOfEta} with $n$ replaced by $n+1$. 

Hence $\eta(v)$ is infinitely Fr\'echet differentiable and \eqref{e:derivativesOfEta} holds for all $n\in\NN$. The bound for $\|D^n \eta(v)\|_{\lL_\lambda^1}$ then follows from the expression \eqref{e:derivativesOfEta} and the exponential decay of derivatives of $m$ in Lemma~\ref{lem:statationary_exponential_decay}. 


For Assertion 2, the definition of $\eta$ directly gives the identity
    \begin{equation*}
        \eta(m_\theta)=\theta\;.
    \end{equation*}
The assertion then follows from differentiating the above identity (in $\theta$) $n$ times and applying the chain rule Lemma~\ref{lem:chainRule_n}. 
\end{proof}

\section{Properties of $\rho_\eps$}\label{sec:rho}

Recall from Section~\ref{sec:corrector_construction} that $L_\eps = \frac{2}{\sqrt{\eps}}$, and the definition of $\rho_\eps$ from \eqref{e:rho} with a sufficiently large cutoff $N_\eps \gg e^{\frac{2}{\sqrt{\eps}}}$. For every $N\in\NN$, define $\theta_N: \TT \rightarrow \RR$ by
\begin{equ}
        \theta_N(z):=\sum_{k\in\ZZ}\chi^2(k/N)e_k(z) = \sum_{k\in\ZZ}\chi^2(k/N) e^{2\pi i kz}\;,
    \end{equ}
where $e_k (y)= e^{2\pi i k y}$. Then we have
\begin{equation*}
    \rho_\eps (y_1, y_2) = \frac{1}{L_\eps} \theta_{N_\eps} \Big( \frac{y_1 - y_2}{L_\eps} \Big)\;.
\end{equation*}

\begin{lem}\label{lem:theta_boundedness}
We have the bound
    \begin{equ}
        \sup_{N\in\NN} \int_{-\frac{1}{2}}^{\frac{1}{2}} |\theta_N(y)|\,\md y<+\infty\;.
    \end{equ}
\end{lem}
\begin{proof}
Noting that we have the identity
    \begin{align*}
        (e^{2\pi i y} -1)^2\theta_N(y) =\sum_{k\in\ZZ}\big(\chi^2(k/N) -2\chi^2((k-1)/N)+\chi^2((k-2)/N)\big)e^{2\pi i ky}\;.
    \end{align*}
    It then follows from the assumption $\chi \in \cC_c^\infty$ that
    \begin{equ}
        |\theta_N(y)| \lesssim \frac{1}{N|e^{2\pi i y} -1|^2}\lesssim\frac{1}{Ny^2}\;.
    \end{equ}
    Combining this with the bound $|\theta_N(y)|\lesssim N$, we can get the desired bound.
\end{proof}

\begin{lem}\label{lem:rho_boundedness}
    For every $\lambda >0$, we have
    \begin{equ}
    \int_{\RR^2} |\varphi(y_1,y_2)\cdot\rho_\eps(y_1,y_2)|\,\md y_1\,\md y_2\lesssim \|\varphi\|_{\lL^\infty_{\lambda}(\RR^2)}
\end{equ}
uniformly over all $\eps\in(0,1)$ and $\varphi\in\lL^\infty_{\lambda}(\RR^2)$.
\end{lem}
\begin{proof}
    Performing a change of variables $(y_1,y_2) \mapsto (y+\frac{L_\eps}{2}z, y -\frac{L_\eps}{2}z)$, we have
    \begin{equation}\label{e:rho_changeVariable}
    \int_{\RR^2} |\varphi (y_1, y_2) \cdot \rho_\eps(y_1, y_2)| \, \md y_1 \,\md y_2 = \int_{\RR^2} |\varphi \big( y + \frac{L_\eps z}{2}\,, y - \frac{L_\eps z}{2} \big) | \cdot |\theta_{N_\eps}(z)| \, \md z \,\md y\;.
    \end{equation}
    By Lemma~\ref{lem:theta_boundedness}, we deduce
    \begin{equation}\label{e:rho_bound_strip}
    \begin{aligned}
    \int_{\RR} \int_{k-\frac{1}{2}}^{k+\frac{1}{2}} |\varphi \big( y + \frac{L_\eps z}{2}\,, y - \frac{L_\eps z}{2} \big) | \cdot |\theta_{N_\eps}(z)| \, \md z \,\md y \lesssim e^{-\lambda (|k|-\frac{1}{2})^+ L_\eps}\|\varphi\|_{\lL^\infty_{\lambda}(\RR^2)}\;.
    \end{aligned}
    \end{equation}
    Then the conclusion follows from summing over $k\in\ZZ$.
\end{proof}

\begin{lem}\label{lem:theta_convergence}
    For every $\alpha\in(0,1)$, there exists $\nu>0$ such that
    \begin{equation*}
        \big| \langle \theta_N -\delta, \, \phi \rangle \big| \lesssim N^{-\nu} \|\phi\|_{\cC^\alpha(\TT)}\;.
    \end{equation*}
    Here, $\delta$ denotes the delta function on $\TT$, and the proportionality constant is independent of $N \in \NN$ and $\phi \in \cC^\alpha(\TT)$. 
\end{lem}
\begin{proof}
By Lemma~\ref{lem:theta_boundedness}, we have the uniform bound 
\begin{equation*}
    \big| \langle \theta_N -\delta, \, \phi \rangle \big| \lesssim \|\phi\|_{\cC^{\frac{\alpha}{2}}(\TT)}\;.
\end{equation*}
For smallness in $N$, we use Fourier characterization of $\hH^{-1}$-norm to control the object by
\begin{equation*}
    \big| \langle \theta_N -\delta, \, \phi \rangle \big| \leq \|\theta_N - \delta\|_{\hH^{-1}} \|\phi\|_{\hH^1} \lesssim N^{-\nu'} \|\phi\|_{\hH^1}
\end{equation*}
for some $\nu'>0$. The claim follows from interpolating the above two bounds. 
\end{proof}

\begin{lem}\label{lem:rho_convergence}
    For every $\alpha\in(0,1)$ and $\lambda > 0$, we have
    \begin{equ}
    \Big|\int_{\RR^2} \varphi(y_1,y_2)\cdot\rho_\eps(y_1,y_2)\,\md y_1\,\md y_2 - \int_\RR \varphi(y,y)\,\md y\Big|\lesssim \eps \big(\|\varphi\|_{\lL^\infty_{\lambda}(\RR^2)}+\|\varphi\|_{\cC^\alpha(\RR^2)}\big)
\end{equ}
uniformly over all $\eps\in(0,1)$ and $\varphi\in\lL^\infty_{\lambda}(\RR^2)\cap\cC^\alpha(\RR^2)$.
\end{lem}
\begin{proof}
According to the change of variable \eqref{e:rho_changeVariable} and the bound \eqref{e:rho_bound_strip} for $|k| \geq 1$, it suffices to show that
\begin{equation} \label{e:key_term_delta_cont}
     \bigg|\int_\RR \int_{-\frac{1}{2}}^{\frac{1}{2}}\varphi(y+\frac{L_\eps}{2} z,y-\frac{L_\eps}{2} z)\cdot\theta_{N_\eps}(z)\,\md z\,\md y - \int_\RR \varphi(y,y)\,\md y\bigg|\lesssim \eps\big(\|\varphi\|_{\lL^\infty_{\lambda}(\RR^2)}+\|\varphi\|_{\cC^\alpha(\RR^2)}\big)\;.   
\end{equation}
Let
\begin{equation*}
    \widetilde{\varphi}(y) := \varphi \big( y + \frac{L_\eps}{4}, y - \frac{L_\eps}{4} \big) - \varphi \big( y - \frac{L_\eps}{4}, y + \frac{L_\eps}{4} \big)\;,
\end{equation*}
and set
\begin{equation*}
    \psi_y(z) := \varphi \big( y + \frac{L_\eps z}{2}, y - \frac{L_\eps z}{2} \big) - \widetilde{\varphi}(y) \cdot z\;.
\end{equation*}
Then $\psi_y$ is H\"older continuous on $\TT$, and we have the bounds
\begin{align*}
    |\widetilde{\varphi}(y)| \lesssim &e^{-\lambda' (|y|+L_\eps)}\|\varphi\|_{\lL^\infty_{\lambda}(\RR^2)}\;, \quad \sup_{z \in \TT} |\psi_y(z)| \lesssim e^{-\lambda |y|}\|\varphi\|_{\lL^\infty_{\lambda}(\RR^2)}\;,\\
    &\|\psi_y\|_{\cC^{\alpha}(\TT)} \lesssim L_\eps^{\alpha}\big(\|\varphi\|_{\lL^\infty_{\lambda}(\RR^2)}+\|\varphi\|_{\cC^\alpha(\RR^2)}\big)
\end{align*}
for all $y \in \RR$. Interpolating the $\lL^\infty$ and $\cC^{\alpha}$ bounds of $\psi_y$ gives
\begin{equation*}
    \|\psi_y\|_{\cC^{\frac{\alpha}{2}}(\TT)} \lesssim L_\eps^{\frac{\alpha}{2}} \cdot e^{-\lambda''|y|} \big(\|\varphi\|_{\lL^\infty_{\lambda}(\RR^2)}+\|\varphi\|_{\cC^\alpha(\RR^2)}\big)
\end{equation*}
for some $\lambda''>0$. Hence, by Lemmas~\ref{lem:theta_convergence} and~\ref{lem:theta_boundedness}, we have
\begin{equation*}
    \begin{split}
    &\phantom{111}\Big| \int_{-\frac{1}{2}}^{\frac{1}{2}} \varphi \big( y + \frac{L_\eps z}{2}, y - \frac{L_\eps z}{2} \big) \cdot\theta_{N_\eps}(z) \, \md z - \varphi(y,y) \Big|\\
    &\leq \big| \langle  \theta_{N_\eps} - \delta, \psi_y \rangle \big| + |\widetilde{\varphi}(y)| \cdot \Big| \int_{-\frac{1}{2}}^{\frac{1}{2}} z \theta_{N_\eps}(z) \, \md z \Big|\\
    &\lesssim \left(N_\eps^{-\nu} L_\eps^{\frac{\alpha}{2}} \, e^{-\lambda'' |y|} + e^{-\lambda' (|y| + L_\eps)} \right)\big(\|\varphi\|_{\lL^\infty_{\lambda}(\RR^2)}+\|\varphi\|_{\cC^\alpha(\RR^2)}\big)
    \end{split}
\end{equation*}
for some $\nu>0$. Note that $N_\eps \gg e^{L_\eps}$. The bound \eqref{e:key_term_delta_cont} then follows by integrating $y \in \RR$. This finishes the proof of the lemma. 
\end{proof}

\endappendix

\bibliographystyle{Martin}
\bibliography{Refs}

\end{document}